%% file: elast_incomp_DG.tex
\documentclass[review]{amsart}

\input{ex_shared}
\usepackage{hyperref}

\newcommand\Q{\mathrm{Q}}
\renewcommand\H{\mathrm{H}}

\newcommand\err{\texttt{err}}
\newcommand\eff{\texttt{eff}}

\renewcommand\O{\Omega}

\renewcommand\H{\mathrm{H}}
\renewcommand\L{\mathrm{L}}

\begin{document}

\subjclass[2020]{35Q35,  65N15, 65N25, 65N30, 76D07}

\date{}

\dedicatory{}
\keywords{Eigenvalue problems, Discontinuous Galerkin method,  Error estimates, A posteriori analysis}

\begin{abstract}
This paper studies the family of interior penalty discontinuous Galerkin methods for solving the Herrmann formulation of the linear elasticity eigenvalue problem in heterogeneous media. By employing a weighted Lam\'e coefficient norm within the framework of non-compact operators theory, we prove convergence of both continuous and discrete eigenvalue problems as the mesh size approaches zero, independently of the Lam\'e constants. Additionally, we conduct an a posteriori analysis and propose a reliable and efficient estimator. The theoretical findings are supported by numerical experiments.
\end{abstract}

\maketitle

\section{Introduction}\label{sec:intro}
Recently in \cite{khan2023finite} the nearly incompressible elasticity eigenvalue problem has been analyzed. In this reference the authors have proved that the elasticity spectrum converges to the Stokes eigenvalues when the Poisson ratio tends to $1/2$.  From this fundamental observation, and with the aid of  inf-sup stable families of finite elements for Stokes, the authors approximated the spectrum accurately. Motivated by these findings and in alignment with our research trajectory, we now turn our attention to the application of the interior penalty discontinuous Galerkin method.

Interior Penalty Discontinuous Galerkin Finite Element Method (IPDG) represents a sophisticated approach to tackling eigenvalue problems, leveraging the advantages of both discontinuous Galerkin methods and interior penalty techniques. This method excels in handling problems associated with singularities, discontinuities, or complex geometries by incorporating penalty terms within the weak formulation. The inclusion of these penalty terms facilitates the imposition of continuity constraints across element boundaries, effectively minimizing numerical oscillations and ensuring stability. IPDG proves particularly adept at capturing eigenvalues in problems with variable coefficients or irregular domains. By introducing penalty terms to penalize jumps in the solution and its gradient across element interfaces, this method strikes a balance between accuracy and stability, making it a reliable choice for eigenvalue problems characterized by challenging features and intricate geometries.

The application of the IPDG methods for eigenvalue problems is a current subject of study and we refer to \cite{MR2220929,MR2263045,MR4623018,MR4077220,MR3962898,MR4673997} as our main references which explore applications encompassing the Laplace operator, the Maxwell's equations, Stokes, elasticity and acoustics.  These studies collectively demonstrate the broad applicability and efficacy of IPDG techniques in diverse scientific domains. Here the main advantage lies, in one hand, on the flexibility of the method where hanging nodes are allowed, and on the other, the easy computational implementation where  with no major difficulties is possible to consider high order approximations for the method. This inherent adaptability not only enables accurate approximation of the spectrum but also highlights the method's suitability for tackling complex eigenvalue problems. However, there is a price to pay when using this method and is related to the stabilization parameter that precisely penalizes the jumps between elements. While this parameter is crucial for ensuring stability and convergence, its determination can pose significant computational hurdles. Theoretical investigations \cite{MR4623018,MR4077220,MR3962898,MR4673997} have underscored the significance of the stabilization parameter in determining the stability of IPDG methods. The appropriate selection of this parameter depends on several factors, such as boundary conditions, geometrical assumptions and physical parameters, among others. The correct choice of the stabilization parameter is fundamental to achieve an accurate approximation to the physical spectrum, without spurious eigenvalues. The literature suggests that the stabilization parameters leading to accurate results are proportional to the square of the polynomial degree, although the proportionality constant varies depending on the specific physical parameters involved. This empirical observation highlights the intricate interplay between numerical methods and physical phenomena, and emphasizes the need for judicious parameter selection to ensure the fidelity of the computational results.


\subsection{Notations}

Throughout this work, $\O$ is a generic Lipschitz bounded domain of $\R^d$, $d\in\{2,3\}$. For $s\geq 0$,
$\norm{\cdot}_{s,\O}$ stands indistinctly for the norm of the Hilbertian
Sobolev spaces $\HsO$ or $[\HsO]^2$ with the convention
$\H^0(\O):=\LO$.  If ${X}$ and ${Y}$ are normed vector spaces, we write ${X} \hookrightarrow {Y}$ to denote that ${X}$ is continuously embedded in ${Y}$. We denote by ${X}'$ and $\|\cdot\|_{{X}}$ the dual and the norm of ${X}$, respectively.
Finally,
we employ $\0$ to denote a generic null vector and
the relation $\texttt{a} \lesssim \texttt{b}$ indicates that $\texttt{a} \leq C \texttt{b}$, with a positive constant $C$ which is independent of $\texttt{a}$, $\texttt{b}$, and the size of the elements in the mesh. The value of $C$ might change at each occurrence. We remark that we will write the constant $C$ and the estimates dependencies only when is needed.

Let us recall the following standard notations for any tensor field $\boldsymbol{\tau}=\left(\tau_{i j}\right)_{1\leq i, j\leq d}$, any vector field $\boldsymbol{v}=\left(v_i\right)_{1\leq i\leq d}$ and any scalar field $v$:
$$
\begin{gathered}
	\nabla\cdot \boldsymbol{v}:=\sum_{i=1}^{d}\partial_iv_i,  \quad \nabla v:=\left(\partial_i v\right)_{1\leq i\leq d},\quad \boldsymbol{\nabla}\cdot \btau:=\left(\sum_{j=1}^{d}\partial_j\tau_{ij}\right)_{1\leq i\leq d},\\
	\nabla \bv:= \left(\partial_i v_j\right)_{1\leq i,j\leq d},\quad	\boldsymbol{\tau}^{\mathrm{t}}:=\left(\tau_{j i}\right), \quad \operatorname{tr}(\boldsymbol{\tau}):=\sum_{i=1}^d \tau_{i i}, \quad \boldsymbol{\tau}^a: \boldsymbol{\tau}^b:=\sum_{i, j=1}^d \tau_{i j}^a \tau_{i j}^b.
\end{gathered}
$$

\subsection{Model problem}
Let $\Omega\subset\mathbb{R}^d$ with $d\in\{2,3\}$ be an open and bounded
domain with Lipschitz boundary $\partial\Omega$. Let us assume that
the boundary is divided into two parts $\partial\Omega:=\Gamma_D\cup\Gamma_N$,
where $|\Gamma_D|>0$. In particular, we assume that the structure is fixed on $\Gamma_D$ and free of stress on $\Gamma_N$.

The eigenvalue problem of interest is given by:  Find $(\widehat{\kappa}, \bu)$ such that
\begin{equation}
	\label{eq:eigenvalue-problem}
	\begin{aligned}
		-\boldsymbol{\nabla}\cdot\left[2\widehat{\mu}\beps(\bu) + \widehat{\lambda}\tr(\beps(\bu))\boldsymbol{I}\right] &= \widehat{\kappa}\rho\bu &\text{ in } \Omega,\\
		\bu &= \boldsymbol{0} & \mbox{ on } \G_D,\\
		\left[2\widehat{\mu}\beps(\bu) + \widehat{\lambda}\tr(\beps(\bu))\boldsymbol{I}\right]\bn  & =  {\boldsymbol{0} }& \mbox{ on } \G_N,
	\end{aligned}
\end{equation}
where $\boldsymbol{I}\in\mathbb{R}^{d\times d}$ represents the identity matrix, $\sqrt{\widehat{\kappa}}$ is the natural frequency, $\bu$ is the displacement and $\rho$ is the material density. The Lam\'e parameters $\widehat{\lambda}$ and $\widehat{\mu}$ are defined by
$$
	\widehat{\lambda}:=\frac{E(x)\nu}{(1+\nu)(1-2\nu)}\quad\text{and}\quad \widehat{\mu}:=\frac{E(x)}{2(1+\nu)},
$$
where $E(x)$ represents  the space-variable Young's modulus, $\nu$ represents  the Poisson ratio which we assume as a constant. The strain tensor is represented by $\beps(\cdot)$ and defined as $\beps(\bv):=\frac{1}{2}(\nabla\bv+(\nabla\bv)^{\texttt{t}})$, where $\texttt{t}$ represents the transpose operator.    Here and thereafter we   assume the existence of two positive constants
$E_{\min}$ and $E_{\max}$ such that   
$
E_{\min}\leq E(x)\leq E_{\max}.
$

For the analysis of \eqref{eq:eigenvalue-problem} we scale the system with the quantity  $(1+\nu)$ leading to the following problem. Find $(\kappa, \bu)$ with $\bu\neq\boldsymbol{0}$ such that
\begin{equation}
	\label{eq:eigenvalue-problem-scaled}
	\begin{aligned}
		-\boldsymbol{\nabla}\cdot\left[2\mu\beps(\bu) + \lambda\tr(\beps(\bu))\boldsymbol{I}\right] &= \kappa\rho\bu &\text{ in } \Omega,\\
		\bu &= \boldsymbol{0} & \mbox{ on } \G_D,\\
		\left[2\mu\beps(\bu) + \lambda\tr(\beps(\bu))\boldsymbol{I}\right]\bn  & =  \boldsymbol{0} & \mbox{ on } \G_N,
	\end{aligned}
\end{equation}
where the scaled eigenvalue and the Lam\'e coefficients are given by
\begin{equation}
\label{eq:scaled}
\kappa:=(1+\nu)\widehat{\kappa},\qquad\lambda:=\frac{E(x)\nu}{1-2\nu},\qquad \mu:=\frac{E(x)}{2}.
\end{equation}
Note that for a given $E(x)$, the eigenvalue problem \eqref{eq:eigenvalue-problem-scaled} has the advantage of having $\mu$ independent of $\nu$. The incompressibility of the material is explicitly given only by the behavior of the constant $\nu$ through the Lamé parameter $\lambda$.
 
With the scaled parameters specified in \eqref{eq:scaled} we rewrite system \eqref{eq:eigenvalue-problem} as the following system:
We seek the eigenvalue $\kappa\in\mathbb{R}^+$, a vector field $\bu\neq\boldsymbol{0}$, 
and a scalar field $p\neq 0$ such that
\begin{align}\label{stokes-cont}
-\boldsymbol{\nabla}\cdot(2\mu\beps(\bu)) +\nabla p & =
 \kappa \rho  \bu & \mbox{ in } \O, \nonumber\\
\nabla\cdot\bu +\frac{1}{\lambda}p& =  0 & \mbox{ in } \O, \nonumber  \\
\bu & = \boldsymbol{0} & \mbox{ on } \G_D,\\
\left[2\mu\beps(\bu) - p\bI\right]\bn & =  \boldsymbol{0} & \mbox{ on } \G_N. \nonumber
 \end{align}
We note from \eqref{eq:scaled} that if  the Poisson ratio approaches $1/2$,
then $\lambda \uparrow \infty$.
In such case, we recover the Stokes system.
This is an important matter in view of deriving robust stability and error bounds.   From now and on, and for simplicity,
in the analysis below we consider $\rho=1$.

\subsection{Main contribution}
To the best of our knowledge, the realm of IPDG methods remains uncharted territory in the context of the model problem (\ref{stokes-cont}). This paper marks the inaugural exploration in this direction. The primary focus is on the family of IPDG methods employed for solving the eigenvalue problem in nearly incompressible elasticity within heterogeneous media. A cornerstone of this research lies in the comprehensive examination of both a priori and a posteriori error analyses, shedding light on the method's performance. Notably, our proposed estimator exhibits remarkable robustness with respect to the Poisson ratio, underpinning its suitability for addressing complex, real-world scenarios. Furthermore, it is established that the spectrum of the nearly incompressible elasticity eigenvalue problem within heterogeneous media converges to the spectrum of the Stokes problems as the parameter $\lambda$ tends towards infinity


\subsection{Outline}
The subsequent sections of the paper are structured as follows: Section \ref{sec:prob_form} delves into the weak formulation and presents some initial findings. Section \ref{sec:DG} explores the family of IPDG methods concerning the model problem (\ref{stokes-cont}). Convergence results and error estimates are thoroughly examined in Section \ref{sec:conv_error}. A comprehensive a posteriori error analysis is detailed in Section \ref{sec:apost}. Lastly, Section \ref{sec:numerics} showcases numerical results, demonstrating the effectiveness of the IPDG methods.


\section{Variational formulation and preliminary results}\label{sec:prob_form}
 Let us begin by  defining  the spaces $\H:=\left\{\bv\in \H^1(\Omega)^d\;:\; \bv =\boldsymbol{0} \text{ on }\Gamma_D\right\}$ and $\Q:=\L^2(\Omega)$ where we will seek the displacement and the pressure, respectively. Now, by testing
 system \eqref{stokes-cont} with adequate
functions and imposing the boundary conditions,
we end up with the following  saddle point variational formulation:
Find $(\kappa,(\bu,p))\in\R^+\times\H\times\Q$ with $(\bu,p)\neq (\boldsymbol{0},0)$ such that
\begin{align*}
2\int_{\O}\mu(x)\beps(\bu):\beps(\bv)\, dx-\int_{\O}p\nabla\cdot\bv\, dx&=
\kappa\int_{\O}\bu\cdot\bv\, dx&
\quad\forall\bv\in\H,\\ 
-\int_{\O}q\nabla\cdot\bu\, dx-\int_{\O}\frac{1}{\lambda(x)}pq\,dx&=\,0&\quad\forall q\in\Q.
\end{align*}
This variational problem can be rewritten as follows:
{\em Find $(\kappa,(\bu,p))\in\R\times\H\times\Q$ such that}
\begin{equation}\label{def:limit_system_eigen_complete}
	\left\{
	\begin{array}{rcll}
a(\bu,\bv)      +b(\bv,p)&=&\kappa d(\bu,\bv)&\forall\bv\in\H,\\
b(\bu,q)     - c(p,q)  & =&\;0 &\forall q\in\Q,
\end{array}
\right.
\end{equation}
where the bilinear forms
$a:\H\times\H\to\R$,
$b:\H\times\Q\to\R$,
$c:\Q\times\Q\to\R$,
and
$d:\H\times\H\to\R$
are defined by
\begin{align*}
a(\bu,\bv)&:=2\int_{\O}\mu(x)\beps(\bu):\beps(\bv)\, dx,\quad
b(\bv,q):=-\int_{\O}q\nabla\cdot\bv\, dx,\\
d(\bu,\bv)&:=\int_{\O} \bu \cdot\bv\, dx,\quad  c(p,q):=\int_{\O} \frac{1}{\lambda(x)}p\, q\, dx,
\end{align*}
for all $\bu,\bv\in\H$, and $p,q\in\Q$.

To perform the analysis, we need a suitable
norm which in particular depends on the parameters of the problem.
With this in mind, and for all $(\bv,q)\in\H\times Q$,
we define the following weighted norm,
$$
	\vertiii{(\bv,q)}_{\H\times Q}^2:=\Vert\mu(x)^{1/2}\nabla\bv\Vert_{0,\O}^2 + \Vert\mu(x)^{-1/2}q\Vert_{0,\O}^2 +\Vert\lambda(x)^{-1/2}q\Vert_{0,\O}^2.
$$
Moreover, in what follows, we assume that $\mu(x),1/\lambda(x)\in \mathrm{W}^{1,\infty}(\Omega)$.

Let us introduce and define  the kernel of $b(\cdot,\cdot)$ by 
$$
\mathcal{K} := \{ \bv \in \H\,:\, b(\bv,q)=0,\quad \forall \, q\in Q\}\,=\,\{ \bv \in \H\,:\, \nabla\cdot \bv \,=\,0\,\,\,\, {\rm in}\,\, \Omega\},
$$
and let us recall that the bilinear form $b(\cdot,\cdot)$ satisfies the inf-sup condition:
$$
\sup_{\stackrel{\scriptstyle\bv\in\H}{\bv\ne0}}\frac{\vert
  b(\bv,q)\vert}{\Vert\mu(x)^{1/2}\nabla\bv\Vert_{0,\O}}\ge \beta_2
  \Vert\mu(x)^{-1/2}q\Vert_{0,\O} \quad\forall q\in\Q,
$$
with an inf-sup constant $\beta_2>0$ only depending on $\Omega$; see \cite{MR3973678}, for instance. 

Now, we are in a position to introduce the solution operator which we denote by
$\bT_{\nu}: \L^2(\Omega)^d\rightarrow \L^2(\Omega)^d$ and is such that $\boldsymbol{f}\mapsto \bT_{\nu} \boldsymbol{f}:=\widehat{\bu}$, where the pair $(\widehat{\bu},\widehat{p})\in \H\times \Q$ is the solution of the following well posed source problem: Given $\boldsymbol{f}\in \L^2(\O)^d$ find $(\widehat{\bu},\widehat{p})\in \H\times \Q$ such that
\begin{equation}\label{def:system_source_complete}
	\left\{
	\begin{array}{rcll}
a(\widehat{\bu},\bv)      +b(\bv,\widehat{p})&=&d(\boldsymbol{f} ,\bv)&\forall\bv\in\H,\\
b(\widehat{\bu},q)     - c(\widehat{p},q)  & =&\;0 &\forall q\in\Q.
\end{array}
\right.
\end{equation}
Thanks to the  Babu\v{s}ka-Brezzi theory
we have that $\bT_{\nu}$ is well defined and the following estimate holds $\Vert\bT_{\nu}\boldsymbol{f}\Vert_{0,\O}\le\vertiii{ (\bT_{\nu}\boldsymbol{f},\widehat{p})} \lesssim\|\boldsymbol{f}\|_{0,\Omega}$
with a constant depending on the inverse of $\mu$ and the Poincaré constant.
It is easy to check that $\bT_{\nu}$ is a selfadjoint
operator with respect to the $\L^2$ inner product. Also, let $\chi$ be a real number such that  $\chi\neq 0$. Notice that $(\chi,\boldsymbol{u})\in \mathbb{R}\times\H$ is an eigenpair of $\bT_\nu$ if and only if  there exists $p\in \Q$ such that,  $(\kappa,(\boldsymbol{u}, p))$ solves problem \eqref{def:limit_system_eigen_complete}  with $\chi:=1/\kappa$.

The following result states an additional regularity
result for problem \eqref{def:system_source_complete}.
\begin{lemma}
\label{rmrk:additional}
Let $(\widehat{\bu},\widehat{p})$ be the unique solution
of \eqref{def:system_source_complete}, then there exists
$s\in(0,1]$ such that $(\widehat{\bu},\widehat{p})\in \H^{1+s}(\O)^{d}\times \H^{s}(\O)$
and the following estimate holds (see for instance \cite{grisvard1986problemes} or  \cite{MR4430561})
$$
\|\widehat{\bu}\|_{1+s,\O}+\|\widehat{p}\|_{s,\O}\lesssim \|\boldsymbol{f}\|_{0,\O},
$$
where the hidden constant depends on $\Omega$ and the Lam\'e coefficients.
\end{lemma}
Let us remark that Lemma~\ref{rmrk:additional} holds true
even for $\lambda\uparrow\infty$ (Further comments on this fact can be seen in \cite{MR3962898}).

As a consequence of this result, we conclude the compactness of  $\bT_{\nu}$.
In fact, since $\H^{1+s}(\Omega)^d\hookrightarrow \L^2(\Omega)^d$
we have that $\bT_{\nu}$ is a compact operator. Hence, we are in position to establish the spectral characterization of $\bT_{\nu}$.
\begin{theorem}
The spectrum of $\bT_{\nu}$  satisfies $\sp(\bT_{\nu})=\{0\}\cup\{\chi_k\}_{k\in\mathbb{N}}$, where $\{\chi_k\}_{k\in\mathbb{N}}$ is a sequence of positive eigenvalues.
\end{theorem}

On the other hand, from  \cite{MR0975121} we have the following regularity result  for the eigenfunctions of the eigenproblem \eqref{def:limit_system_eigen_complete}.
\begin{theorem}
	\label{th:reg_velocity}
	If $(\kappa, (\bu,p))\in \mathbb{R}\times\H\times\Q$ solves \eqref{def:limit_system_eigen_complete}, then, there exists $r>0$ such that $\bu\in\H^{1+r}(\Omega)^d$ and $p\in \H^r(\Omega)$. Moreover, the following estimate  holds
	\begin{equation*}
		\|\bu\|_{1+r,\O}+\|p\|_{r,\O}\lesssim \|\bu\|_{0,\O}.
	\end{equation*}
\end{theorem}
Now,
Let us end this section introducing the bilinear form
$\mathcal{A}:(\H\times \Q)\times(\H\times\Q)\rightarrow\mathbb{R}$ defined by
$$
\mathcal{A}((\bu,p),(\bv,q)):=a(\bu,\bv)+b(\bv,p)+b(\bu,q)+ c(p,q),
$$
which allows us to rewrite the eigenvalue problem as follows. Find $\kappa\in\mathbb{R}$ and $(\boldsymbol{0},0)\neq (\bu,p)\in\H\times\Q$ such that
\begin{equation}
\label{eq:eigen_A}
\mathcal{A}((\bu,p),(\bv,q))=\kappa d(\bu,\bv)\quad\forall (\bv,q)\in\H\times Q.
\end{equation}
Let us remark that $\mathcal{A}(\cdot,\cdot)$ is a bounded bilinear form, in the sense that there exists a positive constant $\widehat{C}:=\max\{2,\mu_{min}^{-1/2}\mu_{max}^{1/2}\}$
such that
$$
\mathcal{A}((\bu,p),(\bv,q))\leq\widehat{C}\vertiii{(\bu,p)}_{\H\times Q}\vertiii{(\bv,q)}_{\H\times Q}\quad\forall(\bu,p),(\bv,q)\in\H\times\Q.
$$
Next we state the stability result for the bilinear form $\mathcal{A}(\cdot,\cdot)$.
\begin{lemma}\label{lem:stab-A}
For any $(\bu,p)\in\H\times\Q$, there exists $(\bv,q)\in\H\times\Q$ with $\vertiii{(\bv,q)}\lesssim \vertiii{(\bu,p)}$ such that
$$
\vertiii{(\bu,p)}^2\lesssim \mathcal{A}((\bu,p),(\bv,q)).
$$
\end{lemma}

\section{The DG method}
\label{sec:DG}
Now our aim is to  introduce the DG methods. In order to do this, we need to set some notations and definitions, inherent for these type of methods, as DG spaces, jumps, averages, and discrete bilinear forms.  For the sake of completeness, in this section we will rigorously monitor the constants in order to provide bounds that explicitly reflect the importance of the physical parameters in the stability of the method.
\subsection{Preliminaries}
Let $\mathcal{T}_h$ be a shape regular family of meshes which subdivide the domain $\bar \Omega$ into  
triangles/tetrahedra that we denote by $K$. Let us denote by $h_K$
the diameter of any element $K\in\mathcal{T}_h$ and let $h$ be the maximum of the diameters of all the
elements of the mesh, i.e. $h:= \max_{K\in \cT_h} \{h_K\}$.

Let $F$ be a closed set. We say that  $F\subset \overline{\Omega}$
is an interior edge/face if $F$ has a positive $(n-1)$-dimensional 
measure and if there are distinct elements $K$ and $K'$
such that $F =\bar K\cap \bar K'$. A closed 
subset $F\subset \overline{\Omega}$ is a boundary edge/face if
there exists $K\in \cT_h$ such that $F$ is an edge/face
of $K$ and $F =  \bar K\cap \G$. 
Let $\cF_h^0$ and $\cF_h^\partial$ be  the sets  of interior edges/faces
and  boundary edges/faces, respectively.
We assume that the boundary mesh $\cF_h^\partial$ is
compatible with the partition $\G = \G_{D} \cup \G_{N}$, namely,
\[
\bigcup_{F\in \cF_h^D} F = \G_{D} \qquad \text{and} \qquad \bigcup_{F\in \cF_h^N} F = \G_N,
\]
where $\cF_h^D:= \set{F\in \cF_h^\partial; \quad F\subset \G_D}$ and 
$\cF_h^N:= \set{F\in \cF_h^\partial: \quad F\subset \G_N}$.
Also we denote   $\cF_h := \cF_h^0\cup \cF_h^\partial$ and $\cF^*_h:= \cF_h^{0} \cup \cF_h^{D}$.
Also, for any element $K\in \cT_h$, we introduce the set
$\cF(K):= \set{F\in \cF_h:\,\, F\subset \partial K}$  of edges/faces composing the boundary of $K$.

For any $t\geq 0$, we define the following broken Sobolev space 
\[
 \H^t(\cT_h)^n:=
 \set{\bv \in [\L^2(\O)]^n: \quad \bv|_K\in [\H^t(K)]^n\quad \forall K\in \cT_h}.
\]
%
Also, the space of the skeletons of the triangulations
$\cT_h$ is defined  by $ \L^2(\cF_h):= \prod_{F\in \mathcal{F}_h} \L^2(F).$

In the forthcoming analysis, $h_\cF\in \L^2(\cF_h)$
will represent the piecewise constant function 
defined by $h_\cF|_F := h_F$ for all $F \in \cF_h$,
where $h_F$ denotes the diameter of an edge/face $F$.

Let  $\cP_m(\cT_h)$ be  the space of piecewise polynomials respect with to
 $\cT_h$  of degree at most $m\geq 0$;
namely,  
\[
 \cP_m(\cT_h) :=\left\{ \bv\in [\L^2(\O)]^d:\, v|_K \in [\cP_m(K)]^d, \forall K\in \cT_h \right\}. 
\]
Let $k\geq 1$.  To approximate the displacement we introduce the space 
$$
\H_h:=\{\bv_h\in [\L^2(\O)]^d\,:\quad \bv_h|_K\in [\cP_k(K)]^d,\,\,\,\forall K\in\CT_h\},
$$
and for the pressure, the finite element space is $\Q_h:=\{\bv_h\in 	\L^2(\O)\,:\, \bv_h|_K\in \cP_{k-1}(K),\,\,\,\forall K\in\CT_h\}$.

Let us assume $\cF(K)\cap \cF(K')\neq\emptyset$.  For a scalar field $q$, we define the average $\mean{q}\in \L^2(\cF_h)$ and the  jump  $\jump{q}\in \L^2(\cF_h)$ by $\mean{q} := (q_K + q_{K'})/2$ and $\jump{q} := q_K\n_K + q_{K'}\n_{K'} $, respectively,  where $\n_K$ is the outward unit normal vector to $\partial K$ and $q_K$ represents the restriction $q|_K$. Similarly, if $\bv$ is a vector field, we define the average $\mean{\bv}\in [\L^2(\cF_h)]^n$ by $\mean{\bv} := (\bv_K + \bv_{K'})/2$, the scalar jump $\jump{\underline{\bv}}\in \L^2(\cF_h)$ by  $\jump{\underline{\bv}} := \bv_K \cdot\n_K + \bv_{K'}\cdot\n_{K'}$, and the tensor or total jump  $\jump{\bv}\in [\L^2(\cF_h)]^{d\times d}$ as $\jump{\bv} := \bv_K \otimes\n_K + \bv_{K'}\otimes\n_{K'}$. Finally, if $\btau$ is a tensor field, we define the corresponding average and jump as $\mean{\btau} := (\btau_K + \btau_{K'})/2 \in [\L^2(\cF_h)]^{d\times d}$ and $\jump{\btau\bn} := \btau_K\n_K + \btau_{K'}\n_{K'}\in [\L^2(\cF_h)]^{d} $, respectively. If we have an element $K$ such that a facet $\cF$ satisfies $\cF\in \cF_h^\partial$, we can obtain the definition of average and jump in the domain boundary by taking $K=K'$ and $K'=0$ in the above definitions, respectively.


Inspired by the analysis of \cite{MR2220929}, let us define $\H(h):=\H+\H_h$ which we endow with the following norm
$$
\|\bv\|_{\H(h)}^2=\|\mu^{1/2}(x)\nabla_h\bv\|_{0,\O}^2+\|\mu^{1/2}(x)h_{\cF}^{-1/2}\jump{\bv}\|_{0,\mathcal{F}_h}^2,
$$
which coincides with the natural norm of $\H$. On the other hand, we also introduce the following norm for $\Q_h$
$$
\|q_h\|_{\Q}^2:=\|\mu^{-1/2}(x)q_h\|_{0,\O}^2+\|\lambda^{-1/2}(x)q_h\|_{0,\O}^2\quad\forall q_h\in \Q_h,
$$
where the natural product norm for the discrete spaces is defined by  $\vertiii{(\bv,q)}_{\H(h)\times \Q}^2:=\|\bv_h\|_{\H(h)}^2+\|q\|_\Q^2$.
Finally we introduce the following trace inequality  \cite{MR4673997}
\begin{equation}\label{eq:dg_trace}
\|h_{\mathcal{F}}^{1/2}\mean{\xi(x)v}\|_{0,\cF}\leq C_{\texttt{tr}}\|v\|_{0,\O}\qquad\forall v\in\cP_k(\mathcal{T}_h),
\end{equation}
where $\xi(x)\in \L^{\infty}(\O)$ and  $C_{\texttt{tr}}>0$ is a constant depending on $\|\xi(x)\|_{0,\infty}$.
\subsection{Symmetric and nonsymmetric DG schemes} With the discrete spaces defined previously at hand, we introduce the discrete counterpart of \eqref{def:limit_system_eigen_complete} as follows: Find $\kappa_h\in\mathbb{C}$ and $(\boldsymbol{0},0)\neq(\bu_h,p_h)\in\H_h\times\Q_h$ such that
\begin{equation}\label{eq:weak_stokes_system_dg}
\begin{array}{rcll}
 a_h(\bu_h,\bv_h)+b_h(\bv_h,p_h) & = & \kappa_h(\bu_h,\bv_h) &\forall\ \bv_h\in \H_h,\\
\ds b_h(\bu_h,q_h)-c(p_h,q_h) & = & 0 &\forall\ q_h\in \Q_h,
\end{array}
\end{equation}
where the sesquilinear form $a_h:\H_h\times\H_h\rightarrow\mathbb{C}$ is defined by
\begin{multline}
\label{eq:Ah}
a_h(\bu_h,\bv_h):=\int_{\CT_h}2\mu(x)\beps_h(\bu_h):\beps_h(\bv_h)+ \int_{\cF^*_h}\frac{\texttt{a}_S}{h_{\cF}}2\mu(x)\jump{\bu_h}\cdot\jump{\bv_h}\\
- \int_{\cF^*_h}\mean{2\mu(x)\beps_h(\bu_h)}\cdot\jump{\bv_h}
-\epsilon \int_{\cF^*_h}\mean{2\mu(x)\beps_h(\bv_h)}\cdot\jump{\bu_h},
\end{multline}
where $\beps_h:=\frac{1}{2}(\nabla_h(\cdot)+(\nabla_h(\cdot))^{\texttt{t}})$ and  $\texttt{a}_S>0$ is a positive constant which we will refer as the stabilization parameter and $ \varepsilon\in\{-1,0,1\}$.
On the other hand,  we define the bounded bilinear form   $b_h:\H_h\times\Q_h\rightarrow \mathbb{R}$  by
$$
b_h(\bv_h,q_h):=-\int_{\CT_h}\div_h\bv_h q_h+\int_{\cF_h^*}\mean{q_h}\jump{\underline{\bv_h}},\quad\forall\bv\in\H_h,\,\forall q_h\in Q_h,
$$
and $c:\Q_h\times \Q_h\rightarrow\mathbb{R}$ is defined by 
$\ds c(p_h,q_h):=\int_{\O}\frac{p_hq_h}{\lambda(x)}$.

We have that $a_h(\cdot,\cdot)$ is  bounded on the broken weighted norm.  Indeed, for $\bu_h,\bv_h\in\H_h$ and  applying \eqref{eq:dg_trace} we have 
\begin{multline*}
|a_h(\bu_h,\bv_h)|\leq \left|\int_{\O} 2\mu(x)\beps_h(\bu_h):\beps_h(\bv_h)\right|+
\left| \int_{\cF_h^*}\frac{\texttt{a}_S}{h_{\cF_h^*}}2\mu(x)\jump{\bu_h}\cdot\jump{\bv_h}\right|\\
+\left|\int_{\cF_h^*}\mean{2\mu(x)\beps_h(\bu_h)}\cdot\jump{\bv_h}\right|+\left|\epsilon\int_{\cF_h^*}\mean{2\mu(x)\beps_h(\bv_h)}\cdot\jump{\bu_h}\right|\\
\leq 2C\|\mu^{1/2}(x)\nabla_h\bu_h\|_{0,\O}\|\mu^{1/2}(x)\nabla_h\bv_h\|_{0,\O}\\
+\texttt{a}_S\|\mu^{1/2}h_{\cF^*}^{-1/2}\jump{\bu_h}\|_{0,\cF_h^*}\|h_{\cF_h}^{-1/2}h_{\cF_h^*}\jump{\bv_h}\|_{0,\cF_h^*}\\
+2\|h_\cF^{1/2}\mean{\mu(x)\nabla\bu_h}\|_{0,\cF_h^*}\|h_{\cF}^{-1/2}\jump{\bv_h}\|_{0,\cF_h^*}
+2\epsilon\|h_\cF^{1/2}\mean{\mu(x)\nabla\bv_h}\|_{0,\cF_h^*}\|h_{\cF}^{-1/2}\jump{\bu_h}\|_{0,\cF_h^*}\\
\leq 2C\|\mu^{1/2}(x)\nabla_h\bu_h\|_{0,\O}\|\mu^{1/2}(x)\nabla_h\bv_h\|_{0,\O}\\
+\texttt{a}_SC_{E_{\min}}\|\mu^{1/2}(x)h_{\cF^*}^{-1/2}\jump{\bu_h}\|_{0,\cF_h^*}\|\mu^{1/2}(x)h_{\cF}\jump{\bv_h}\|_{0,\cF_h^*}\\
+2C_{\texttt{tr}}C_{E_{\min}}^2\|\mu^{1/2}(x)\nabla\bu_h\|_{0,\O}\|\mu^{1/2}(x)h_{\cF}^{-1/2}\jump{\bv_h}\|_{0,\cF_h^*}\\
+2\epsilon C_{\texttt{tr}}C_{E_{\min}}^2\|\mu^{1/2}(x)\nabla\bv_h\|_{0,\O}\|\mu^{1/2}(x)h_{\cF}^{-1/2}\jump{\bu_h}\|_{0,\cF_h^*}\\
\leq C_1\left(\|\mu^{1/2}(x)\nabla_h\bu_h\|_{0,\O}^2 +\|\mu^{1/2}(x)h_{\cF}^{-1/2}\jump{\bu_h}\|_{0,\cF_h^*}^2\right)^{1/2}\\
\times \left(\|\mu^{1/2}(x)\nabla_h\bv_h\|_{0,\O}^2 +\|\mu^{1/2}(x)h_{\cF}^{-1/2}\jump{\bv_h}\|_{0,\cF_h^*}^2\right)^{1/ 2}= C_1\|\bu_h\|_{\H(h)}\|\bv_h\|_{\H(h)},
\end{multline*}
where $C_1:=2\max\{C,C_{E_{\min}}\texttt{a}_S, C_{\texttt{tr}} C_{E_{\min}}^2, \epsilon C_{\texttt{tr}} C_{E_{\min}}^2,1\}$.

For $b_h(\cdot,\cdot)$ the continuity is as follows
\begin{multline*}
|b_h(\bv_h,q_h)|\leq \left|\int_{\O}\div_h\bv_h q_h\right|+\left|\int_{\cF_h^*}\mean{q_h}\jump{\underline{\bv_h}}\right|\\
\leq \|\div_h\bv_h\|_{0,\O}\|q_h\|_{0,\O}+\|h^{1/2}_{\cF}\mean{q_h}\|_{0,\cF_h^*}\|h_{\cF}^{-1/2}\jump{\underline{\bv_h}}\|_{0,\cF_h^*}\\
\leq C\|\nabla_h\bv_h\|_{0,\O}\|q_h\|_{0,\O}+C_{\texttt{tr}}C_{E_{\min}}\|q_h\|_{0,\O}\|\mu^{1/2}(x)h_{\cF}^{-1/2}\jump{\underline{\bv_h}}\|_{0,\cF_h^*}\\
\leq CC_{E_{\min}}C_{E_{\max}}\|\mu^{1/2}(x)\nabla_h\bv_h\|_{0,\O}\|\lambda^{-1/2}(x)q_h\|_{0,\O}\\+C_{\texttt{tr}}C_{E_{\max}}C_{E_{\min}}\|\mu^{-1/2}(x)q_h\|_{0,\O}\|\mu^{1/2}h_{\cF}^{-1/2}\jump{\underline{\bv_h}}\|_{0,\cF_h^*}\\
\leq C_2 \left(\|\mu^{1/2}(x)\nabla_h\bv_h\|_{0,\O}^2 +\|\mu^{1/2}(x)h_{\cF}^{-1/2}\jump{\underline{\bv_h}}\|_{0,\cF_h^*}^2\right)^{1/ 2}\\
\times (\|\mu^{-1/2}(x) q_h\|_{0,\O}^2+\|\lambda^{-1/2}(x) q_h\|_{0,\O}^2)^{1/2}=C_2\vertiii{(\bv_h,q_h)}_{\H(h)\times Q},
\end{multline*}
where $C_2:=\max\{CC_{E_{\min}}C_{E_{\max}},C_{\texttt{tr}}C_{E_{\max}}C_{E_{\min}} \}$.
\begin{remark}
Observe that the parameter  $ \epsilon\in\{-1,0,1\}$ dictates  if the IPDG method is symmetric or non-symmetric. More precisely,  if $\epsilon=1$ we obtain
the classic symmetric interior penalty method (SIP),
if $\epsilon=-1$ we obtain the non-symmetric interior penalty method (NIP)
and if $\epsilon=0$ the incomplete interior penalty method (IIP). Clearly when non-symmetric methods are considered, complex computed eigenvalues are expected when
the spectrum is approximated. Moreover, the SIP attains  optimal order of convergence for the approximation of the eigenvalues and eigenfunctions, whereas
the NIP and IIP methods deliver suboptimal orders.  
\end{remark}

Finally, the IPDG discretization of \eqref{eq:eigen_A} reads as follows: Find $\kappa_h\in\mathbb{C}$ and $(\boldsymbol{0},0)\neq(\bu_h,p_h)\in\H_h\times\Q_h$ such that
\begin{equation}
\label{eq:dg_eigen}
\mathcal{A}_h((\bu_h,p_h),(\bv_h,q_h))=\kappa_h(\bu_h,\bv_h)\quad\forall (\bv_h,q_h)\in\H_h\times\Q_h.
\end{equation}

To establish the well posedness of our discrete problem \eqref{eq:weak_stokes_system_dg}, we have from   \cite[Proposition 10]{MR1886000} the following discrete inf-sup condition
$$
\displaystyle\sup_{\btau_h\in\H_h}\frac{b_h(\btau_h,q_h)}{\|\btau_h\|_{ \H(h)}}\geq\beta\|\mu^{-1/2}(x)\,q_h\|_{0,\O}\quad\forall q_h\in\Q_h,
$$
where the constant $\beta>0$ is independent of $h$. On the other hand, let us define the discrete kernel $\mathcal{K}_h$ of $b_h(\cdot,\cdot)$ as follows
$$
\mathcal{K}_h:=\{\btau_h\in\H_h\,:\,  b_h(\btau_h,q_h)=0\,\,\,\,\forall q_h\in\Q_h\}.
$$
With this space at hand, we prove the following coercivity result for $a_h(\cdot,\cdot)$.
\begin{lemma}[ellipticity of $a_h(\cdot,\cdot)$]
\label{lmm:elipt_disc}
For any $\varepsilon\in\{-1,0,1\}$, there exists a positive parameter $\texttt{a}^*$ such that for all $\texttt{a}_S\geq \texttt{a}^*$ there holds
$$a_h(\bv_h,\bv_h)\geq \alpha\|\bv_h\|_{\H(h)}^2\quad\forall\bv_h\in  \mathcal{K}_h,$$
where $\alpha>0$ is independent of $h$.
\end{lemma}
\begin{proof}Let $\bv_h\in \mathcal{K}_h$. Hence, applying Cauchy-Schwarz and  Korn's inequalities and the the well known identity $\displaystyle ab\leq\frac{a^2}{2\eta}+\frac{\eta b^2}{2}$ for $a,b\in\mathbb{R}$ and $\eta>0$, we have 
\begin{multline*}
a_h(\bv_h,\bv_h)=\int_{\O}2\mu(x)\beps_h(\bv_h):\beps_h(\bv_h)+\int_{\cF_h^*}\frac{\texttt{a}_S}{h_{\cF}}2\mu(x)\jump{\bv_h}\cdot\jump{\bv_h}\\
-2(1+\epsilon)\int_{\cF_h^*}\mean{\mu(x)\beps_h(\bv_h)}\cdot\jump{\bv_h}\\
\geq 2 C_{\texttt{Korn}}E_{\min}\|\widetilde{\mu}^{1/2}\nabla_h\bv_h\|_{0,\O}^2+2\texttt{a}_S\|\mu^{1/2} h_{\cF}^{-1/2}\jump{\bv_h}\|_{0,\cF_h^*}^2\\
+2(1+\epsilon)E_{\max}\left(-\frac{\|h_{\cF}^{1/2}\widetilde{\mu}^{1/2}\beps(\bv_h)\|_{0,\cF_h^*}^2}{2\eta}-\eta\frac{\|h_{\cF}^{-1/2}\jump{\widetilde{\mu}^{1/2}\bv_h}\|_{0,\cF_h^*}^2}{2}\right)\\
\geq \frac{1}{E_{\max}}\underbrace{\left(2C_{\texttt{Korn}}E_{\min}-\frac{2(1+\epsilon) C_{E_{\max}}}{2\eta}\right)}_{:=\widehat{C}_1}\|\mu^{1/2}(x)\nabla_h\bv_h\|_{0,\O}^2\\
+\underbrace{\left(\texttt{a}_S-\frac{E_{\max}(1+\epsilon)\eta}{2E_{\min}}\right)}_{:=\widehat{C}_2}\|\mu^{1/2}(x)h_{\cF}^{-1/2}\jump{\bv_h}\|_{0,\cF_h^*}^2\geq \min\{\widehat{C}_1,\widehat{C}_2\}\|\bv_h\|_{\H(h)}^2.
\end{multline*}
We observe that to conclude the proof, the parameter $\eta$ must be chosen as 
$$
\eta>\frac{(1+\epsilon) C_{E_{\max}}}{2C_{\texttt{Korn}}E_{\min}}\quad \text{and hence}\quad \texttt{a}_S>\texttt{a}^*:=\frac{(1+\epsilon)^2 C_{E_{\max}}E_{\max}}{4 E_{\min}^2 C_{\texttt{Korn}}}. 
$$
This leads to the positiveness of $\widehat{C}_1$ and $\widehat{C}_2$ and hence,  the proof.
\end{proof}

Let us introduce the discrete solution operator defined by
$$
	\bT_h:\H\rightarrow \H_h,\qquad \boldsymbol{f}\mapsto \bT_h\boldsymbol{f}:=\widehat{\bu}_h, 
$$
where given $\boldsymbol{f}\in\H$, the pair $(\widehat{\bu}_h,\widehat{p}_h)\in\H_h\times\Q_h$ solves the following discrete source problem
$$
\mathcal{A}_h((\widehat{\bu}_h,\widehat{p}_h),(\bv_h,q_h))=(\boldsymbol{f},\bv_h)\quad\forall (\bv_h,q_h)\in\H_h\times\Q_h.
$$

It is easy to check that there exists a positive constant $\widehat{\alpha}$ depending on the constant provided by Lemma \ref{lmm:elipt_disc} such that
\begin{equation}
\label{eq:elipt_Ah}
\mathcal{A}_h((\widehat{\bu}_h,\widehat{p}_h),(\widehat{\bu}_h,\widehat{p}_h))\geq \widehat{\alpha}\vertiii{(\bu_h,p_h)}^2_{\H(h)\times\Q_h}\quad\forall (\bu_h,p_h)\in\H(h)\times\Q_h.
\end{equation}
\section{Convergence and error estimates}
\label{sec:conv_error}
The aim of this section is to derive convergence results and error estimates  for our DG methods. Despite to the fact that $\bT$ is compact, the classic theory of compact operators is not enough to conclude the convergence in norm between the continuous and discrete solution operators, since the numerical method of our interest is non conforming. Is this reason, and following the spirit of \cite{MR2220929},  why  we resort to the theory of non compact operators of \cite{MR483400,MR483401}. 

In what follows, we will denote by $\norm{\cdot}_{\mathcal{L}(\H(h),\H(h))}$
the corresponding norm acting from $\H(h)$ into the same space.
In addition, we will denote by $\norm{\cdot}_{\mathcal{L}(\H_h,\H(h))}$
the norm of an operator restricted to the discrete subspace $\H_h$;
namely, if $\boldsymbol{L}:\H(h)\to \H(h)$, then
$$
\norm{\boldsymbol{L}}_{\mathcal{L}(\H_h, \H(h))}:=\sup_{\0\neq\btau_h\in\H_h}
\frac{\norm{\boldsymbol{L}\btau_h}_{\H(h)}}{\norm{\btau_h}_{\H(h)}}.
$$
Our first task  is to prove the following properties
\begin{itemize}
\item P1. $\norm{\bT-\bT_h}_{\mathcal{L}(\H_h, \H(h))}\to0$ as $h\to0$.
\item P2. $\forall\btau\in\H$, there holds $\inf_{\btau_h\in \H_h}\norm{\btau-\btau_h}_{\H(h)}\,\,\to0\quad\text{as}\quad h\to0$.
\end{itemize}
We need this properties in order to establish spectral correctness (see \cite{MR483400})
for all the discrete methods (symmetric or nonsymmetric). Property P2 is immediate as a consequence of the density of continuous piecewise degree $k$ polynomial functions in $\H_h$. The task now is to prove P1.

The following convergence result holds for the continuous and discrete solution operators.
\begin{lemma}
\label{lmm:TTh}
For all $\boldsymbol{f}\in \H$, the following estimate holds
$$
\|(\bT-\bT_h)\boldsymbol{f}\|_{\H(h)}\lesssim  h^s\|\boldsymbol{f}\|_{0,\O},
$$
where $s>0$ and the hidden constant is independent of $h$.
\end{lemma}
\begin{proof}
For $\boldsymbol{f}\in\H$ we have $\widehat{\bu}:=\bT\bF$ and $\widehat{\bu}_h:=\bT_h\bF$. Then, there exists a constant $\widetilde{C}>0$ such that the following C\'ea estimate holds
\begin{multline*}
\|(\bT-\bT_h)\bF\|_{\H(h)}\leq \vertiii{(\widehat{\bu}-\widehat{\bu}_h,\widehat{p}-\widehat{p}_h)}_{\H(h)\times \Q_h}\\
\leq \widetilde{C}\inf_{(\widehat{\bv}_h,\widehat{q}_h)\in\H(h)\times Q_h}\vertiii{(\widehat{\bu}-\widehat{\bv}_h,\widehat{p}-\widehat{q}_h)}_{\H(h)\times \Q_h}\\
=\widetilde{C}\left(\inf_{\widehat{\bv}_h\in\H(h)}\|\widehat{\bu}_h-\widehat{\bv}_h\|_{\H(h)}+\inf_{\widehat{q}_h\in\Q_h}\|\widehat{p}-\widehat{q}_h\|_{\Q_h} \right).
\end{multline*}

For the first infimum on the identity above, if $\Pi_h$ represents the Lagrange interpolation operator, we have
\begin{align}
\label{eq:inf1}
\inf_{\widehat{\bv}_h\in\H(h)}&\|\widehat{\bu}_h-\widehat{\bv}_h\|_{\H(h)}\leq \|\widehat{\bu}_h-\Pi_h\widehat{\bu}_h\|_{\H(h)}\nonumber\\
&=\|\mu^{1/2}(x)\nabla (\widehat{\bu}_h-\Pi_h\widehat{\bu}_h)\|_{0,\O}+\|\mu^{1/2}(x)h_{\cF}^{-1/2}\jump{\widehat{\bu}_h-\Pi_h\widehat{\bu}_h}\|_{0,\cF_h^*}\nonumber\\
&\leq C_{E_{\max}}\| \widehat{\bu}_h-\Pi_h\widehat{\bu}_h\|_{1,\O}
\leq C_{E_{\max}} C\|\widehat{\bu}\|_{1+s}\leq C_{E_{\max}} C\|\bF\|_{0,\O}.
\end{align}
For the second infimum, if $\textrm{P}_h$ represents the $\L^2$-projection operator, we have
\begin{align}
\label{eq:inf2}
\inf_{\widehat{q}_h\in\Q_h}\|\widehat{p}-\widehat{q}_h\|_{\Q}& \leq \|\widehat{p}-\textrm{P}_h\widehat{p}\|_\Q=\|\mu^{-1/2}(x)(\widehat{p}-\textrm{P}_h\widehat{p})\|_{0,\O}+\|\lambda^{-1/2}(x)(\widehat{p}-\textrm{P}_h\widehat{p})\|_{0,\O}\nonumber\\
&\leq (E_{\min} +E_{\min})\|\widehat{p}-\textrm{P}_h\widehat{p}\|_{0,\O}\leq C(2E_{\min})\|\widehat{p}\|_{s,\O}
\leq CE_{\min}\|\bF\|_{0,\O}.
\end{align}
Then, combining \eqref{eq:inf1} and \eqref{eq:inf2} we conclude the proof.
\end{proof}

Now we prove the analogous of the previous lemma, but considering discrete sources. Since the proof is analogous to the above lemma, we do not include the details and present the result as the following corollary.
\begin{corollary}
\label{lmm:felipe1}
For all $\boldsymbol{f}_h\in \H_h$, the following estimate holds
\begin{equation*}
\|(\bT-\bT_h)\boldsymbol{f}_h\|_{\H(h)}\lesssim h^s\|\boldsymbol{f}_h\|_{\H(h)},
\end{equation*}
where $s$ is as in Theorem \ref{th:reg_velocity}   and the hidden constant is independent of $h$.
\end{corollary}
Now we are in position to establish P1 and its proof is identical to the one in \cite[Lemma 3]{MR4623018}.
\begin{lemma}
\label{lmm:P1}
There  following estimate holds 
\begin{equation*}\label{P1}
\norm{\boldsymbol{T}-\bT_h}_{\mathcal{L}(\H_h, \H(h))}\lesssim h^s,
\end{equation*}
where the hidden constant is independent of $h$.
\end{lemma}
%

From now on, $\mathcal{D}$ denotes the unitary disk defined
in the complex plane by $\mathcal{D}:=\{z\in\mathbb{C}\,:\, |z|\leq 1\}$. The following result proves that the continuous resolvent
is bounded in the $\H(h)$ norm.

\begin{lemma}\label{TDG}
There exists a constant $C>0$ independent of $h$
such that for all  $z \in\mathcal{D}\setminus \sp(\bT)$  there holds
\[
\norm{(z \bI - \bT) \boldsymbol{f}}_{\H(h)} \geq C|z|\,
\norm{\boldsymbol{f}}_{\H(h)} \quad \forall \boldsymbol{f} \in \H(h).\]
\end{lemma}

\begin{remark}
Lemma~\ref{TDG} implies that the resolvent of $\bT$ is bounded. In other words, 
 if $J$ is a compact subset of $\mathcal{D}\setminus \sp(\bT)$, then there
exists a positive  constant $C$ independent of $h$, satisfying the estimate 
\begin{equation*}\label{resDisc}
\displaystyle\norm{(z\bI-\bT)^{-1}}_{\mathcal{L}(\H(h),\H(h))}\leq C\qquad\forall z\in J.
\end{equation*}
\end{remark}

Our next goal is to derive the boundedness of the discrete
resolvent, when  $h$ is  small enough.
\begin{lemma} \label{ThDG}
If $z \in\mathcal{D}\setminus \sp(\bT)$, there exists $h_0>0$ such that for all $h\leq h_0$,
\[
\norm{(z \bI - \bT_h) \boldsymbol{f}}_{\H(h)} \geq C \, \norm{\boldsymbol{f}}_{\H(h)} \quad \forall \boldsymbol{f}\in \H(h),
\]
with $C>0$ independent of $h$ but depending on $|z|$.
\end{lemma}

The previous lemma states that if we consider a compact subset $E$
of the complex plane such that $E\cap\sp(\bT)=\emptyset$ for $h$
small enough and for all $z\in E$, operator 
$z \bI - \bT_h$ is invertible. Moreover, there exists a positive constant
$C$ independent of $h$ such that $\norm{(z \bI - \bT_h)^{-1}}_{\mathcal{L}(\H(h),\H(h))}\leq C$
for all $z\in E$. This fact is important since it
determines that the numerical method is spurious free for $h$  small enough.
This is summarized in the following result proved in \cite{MR483400}.

\begin{theorem}\label{free}
Let $E\subset\mathbb{C}$ be a compact subset not intersecting $\sp(\bT)$.
Then, there exists $h_0>0$ such that, if $h\leq h_0$, then $E\cap\sp(\bT_h)=\emptyset.$ 
\end{theorem}

We recall the definition of the resolvent operator of $\bT$ and $\bT_h$ respectively:
$$
\begin{aligned}
		&(z\bI-\bT)^{-1}\,:\, \H \to \H\,, \quad z\in\mathbb{C}\setminus \sp(\bT), \\
	&(z\bI-\bT_h)^{-1}\,:\, \H_h \to \H_h\,, \quad z\in\mathbb{C}\setminus\sp(\bT_h) .
\end{aligned}
$$


We introduce the definition of the \textit{gap} $\hdel$ between two closed
subspaces $\CM$ and $\CN$ of $\L^2(\O)$:
$$
\hdel(\CM,\CN)
:=\max\big\{\delta(\CM,\CN),\delta(\CN,\CM)\big\},
$$
where
$$
\delta(\CM,\CN)
:=\sup_{x\in\CM:\ \left\|x\right\|_{0,\O}=1}
\left(\inf_{y\in\CN}\left\|x-y\right\|_{0,\O}\right).
$$

Let $\kappa$ be an isolated eigenvalue of $\bT$
and let $D$ be an open disk in the complex plane with boundary $\gamma$
such that $\kappa$ is the only eigenvalue of $\bT$ lying in $D$ and $\gamma\cap\sp(\bT)=\emptyset$.
We introduce the spectral projector corresponding to the continuous
and discrete solution operators $\bT$ and $\bT_h$, respectively
$$
\begin{aligned}
	&\boldsymbol{\mathfrak{E}}:=\frac{1}{2\pi i}
\int_{\gamma}\left(z\bI-\bT\right)^{-1}\, dz:\H(h)\longrightarrow \H(h),\\
&\boldsymbol{\mathfrak{E}}_h:=\frac{1}{2\pi i}
\int_{\gamma}\left(z\bI-\bT_h\right)^{-1}\, dz:\H(h)\longrightarrow \H(h).
\end{aligned}
$$
The following approximation result for the spectral projections holds.
\begin{lemma}
\label{eq:E-E_h}
There holds
$$
	\lim_{h\rightarrow 0}\|\boldsymbol{\mathfrak{E}}-\boldsymbol{\mathfrak{E}}_h\|_{\mathcal{L}(\H_h,\H(h))}=0.
$$
\end{lemma}
\begin{proof}
See \cite[Theorem 5.1]{MR2220929}.
\end{proof}

The next result provides,  for the proposed IPDG methods,  an error estimate for the eigenfunctions and a double order of convergence for the eigenvalues. More precisely, the result presents 
estimates for symmetric and nonsymmetric methods, where optimal and suboptimal order of convergence are attained, respectively.
\begin{theorem}\label{teo:double-order}
There exists a strictly positive constant $h_0$ such that, for  $\varepsilon\in\{-1,0,1\}$ and $h<h_0$ there holds
$$
\widehat{\delta}_h(\boldsymbol{\mathfrak{E}}(\H,\boldsymbol{\mathfrak{E}}_h(\H_h))\lesssim h^{\min\{r,k\}},\quad \text{ and }\quad
\max_{1\leq i\leq m} |\kappa - \kappa_{i, h}|\lesssim \, h^{\sigma\min\{r, k\}},
$$
where $\sigma:=\frac{1}{2}\left[(\varepsilon+2)^{\varepsilon}-1\right]+1$ and the hidden constants are independent of $h$.
\end{theorem}
\begin{proof}
The proof for the eigenfunction estimate follows the same arguments of \cite[Lemma 7]{MR4623018}. Let $(\kappa_{i,h},(\bu_h,p_h))\in\mathbb{C}\times\H_h\times \Q_h$ be the solution of \eqref{eq:weak_stokes_system_dg} with $i=1,\ldots,m$ and let 
$(\kappa,(\bu,p))\in\mathbb{R}\times\H\times \Q$. 
Let us consider the following well known algebraic identity
\begin{equation}
\label{eq:padra}
\mathcal{A}_h((\bu-\bu_{h},p-p_h),(\bu-\bu_{h},p-p_h))
- \kappa (\bu-\bu_{h},\bu-\bu_{h})
=\left(\kappa_{i,h}-\kappa\right) (\bu_{h},\bu_{h}).
\end{equation}
Hence, taking modulus to $\mathcal{A}(\cdot,\cdot)$ we have
\begin{align}
\label{eq:boundA}
	|\mathcal{A}_h((\bu-\bu_{h},p-p_h),(\bu-\bu_{h},p-p_h))|&\hspace*{-0.05cm}\leq\hspace*{-0.05cm} |a_h(\bu-\bu_{h},,\bu-\bu_{h},)|\hspace*{-0.05cm}+\hspace*{-0.05cm}|c(p-p_h,p-p_h)|\nonumber\\
&\leq \|\bu-\bu_{h}\|^2+\|p-p_h\|^2\\
&\leq \widehat{\delta}_h(\boldsymbol{\mathfrak{E}}(\H,\boldsymbol{\mathfrak{E}}_h(\H_h))\lesssim h^{2\min\{r,k\}}.\nonumber
\end{align}
On the other hand
\begin{equation}
\label{eq:boundL2}
\|\bu-\bu_h\|_{0,\O}^2\leq \widehat{\delta}_h(\boldsymbol{\mathfrak{E}}(\H,\boldsymbol{\mathfrak{E}}_h(\H_h))\lesssim h^{2\min\{r,k\}}.
\end{equation}
Also, invoking \eqref{eq:elipt_Ah} we have
\begin{equation}
\label{eq:stimaL2}
\displaystyle (\bu_{h},\bu_{h})=\frac{\mathcal{A}_h((\bu-\bu_{h},p-p_h),(\bu-\bu_{h},p-p_h))}{\vert\kappa_{i,h}\vert}
\geq\frac{\widehat{\alpha}\vertiii{(\bu_h,p_h)}^2_{\H(h)\times\Q_h}}{\vert\kappa_{i,h}\vert}\geq\widehat{C}>0.
\end{equation}
Hence, plugin \eqref{eq:boundA}, \eqref{eq:boundL2}, and \eqref{eq:stimaL2} in \eqref{eq:padra} we conclude the proof.
\end{proof}

\section{A posteriori error analysis}
\label{sec:apost}
This section aims to present a viable residual-based error estimator for the Stokes/elasticity eigenvalue problem. The primary objective is to demonstrate the equivalence between the proposed estimator and the actual error. Furthermore, in the subsequent analysis, our focus will be exclusively on eigenvalues with a multiplicity of 1. Having established the a priori error estimates accurately in terms of the stabilization constant, the bounding constants will not be followed up during this section.
\subsection{The local and global indicators}
In this section, we introduce an a posteriori error estimator based on residuals for the elasticity eigenvalue problem. The proposed estimator also works with the limit case $\lambda\rightarrow\infty$. For non-symmetric IPDG methods, the eigenvalue can be complex. However, from \cite{MR4077220} it is observed that the imaginary part is negligible. Hence, we provide the analysis for the real part of the spectrum, which works for symmetric and non-symmetric IPDG.
\par
Consider an eigenpair approximation $(\kappa_h,(\bu_h, p_h)) \in \mathbb{R}^+\times \H_h \times \Q_h $. For every element $K \in \mathcal{T}_h$, the interior residual estimator $\eta_{R_K}$ is defined as follows:
 \begin{align*}
 \eta_{R_K}^2 :=&\|(2\mu_h)^{-1/2}h_K(\kappa_h \bu_h+\nabla\cdot (2\mu_h\beps_h (\bu_h) ) -\nabla p_h)\|_{0,K}^2\\
 &+\|((2\mu_h)^{-1}+\lambda^{-1})^{-1/2}[\nabla\cdot \bu_h+(1/\lambda) \,p_h]\|_{0,K}^2,
 \end{align*}
and the facet residual estimator $\eta_{\cF_K}$ by
$$
\begin{aligned}
	\eta_{\cF_K}^2 &:=\sum_{\cF\in \partial K \cap \cF_h^0} \|(2\mu_h)^{-1/2}h_{\cF}^{1/2} [\![(p_h\boldsymbol{I}-2\mu_h\beps_h (\bu_h) )\textbf{n}]\!]\|^{2}_{0,\cF}\\
	&\hspace{4cm} +\sum_{\cF\in \partial K \cap \cF_h^N} \|(2\mu_h)^{-1/2}h_{\cF}^{1/2} (p_h\boldsymbol{I}-2\mu_h\beps_h (\bu_h) )\textbf{n}\|^{2}_{0,\cF},
\end{aligned}
$$
where $\boldsymbol{I}$ represents the identity matrix. Moving forward, we present the estimator $\eta_{J_K}$, designed to quantify the jump in the approximate solution $\bu_h$.
$$
\eta_{J_K}^2 := \sum_{\cF\in\partial K\cap\cF_h^0}  \|(2\mu_h)^{1/2}\gamma_h^{1/2}[\![\bu_h]\!]\|^{2}_{0,\cF}+ \sum_{\cF\in\partial K\cap\cF_h^D}  \|(2\mu_h)^{1/2}\gamma_h^{1/2}\bu_h\otimes\bn\|^{2}_{0,\cF},
$$
with $\gamma_h=\texttt{a}_S/h_{\cF}$, where $\texttt{a}_S$ denotes the penalty parameter as discussed in the section \ref{sec:DG}. The local error indicator, obtained by summing the three aforementioned terms, is defined as:
$$
\eta_K^2 :=\eta_{R_K}^2+\eta_{\cF_K}^2+\eta_{J_K}^2.
$$
 Lastly, we present the (global) a posteriori error estimator
 \begin{equation}\label{errest1}
 \eta :=\left(\sum_{K\in\mathcal{T}_h}\eta_K^2\right)^{1/2},
 \end{equation}
 and the data oscillation term is given by 
$$
 \Theta :=\left(\sum_{K\in\mathcal{T}_h}\Theta_K^2\right)^{1/2},
$$
 where $\Theta_K := \|(2\mu_h)^{-1/2} (\mu-\mu_h)\beps_h (\bu_h) \|_{0,K}$.
\subsection{Reliability}
Following the approach in \cite{HSW,GKDS}, we define $\H_h^c=\H_h\cap \H$. 
The orthogonal complement of $\H_h^c$ in $\H_h$ with respect to the norm $\vertiH{\cdot}$ is denoted as $\H_h^{\perp}$. 
Consequently, we have the decomposition $\H_h=\H_h^c\oplus \H_h^{\perp}$, allowing us to uniquely decompose the DG velocity approximation into $ \bu_h=\bu_h^c+\bu_h^r$,
where $\bu_h^c\in \H_h^c$ and $\bu_h^r\in \H_h^{\perp}$.
Applying the triangle inequality, we can express:
$$
\vertiH{\bu-\bu_h}\le  \vertiH{\bu-\bu_h^c}+\vertiH{\bu_h^r}.
$$
Referring to \cite[Proposition 4.1]{HSW}, we derive the upper bound for the second term.
\begin{equation}\label{realilem1}
\vertiH{\bu_h^r}\lesssim \left(\sum_{K\in\mathcal{T}_h}\eta_{J_K}^2\right)^{1/2}.
\end{equation}
It is  important to remark that the DG form $a_{h}(\cdot,\cdot)$ lacks a well-defined definition for functions $\bu,\bv$ within $\H$. This challenge can be addressed by employing an appropriate lifting operator, as outlined in \cite{BCGKDS,GKDS}. However, we explore an alternative approach here, wherein the DG form $a_{h}(\cdot,\cdot)$ is decomposed into multiple components.
$$
a_{h}(\bu,\bv) = (2\mu\beps_h (\bu), \beps_h (\bv) ) + C_{h}(\bu,\bv) + J_h(\bu,\bv),
$$
where $C_h$ and $J_h$ are defined by
 \begin{align*}
 C_{h}(\bu,\bv)&:=
- \int_{\cF^*_h}\mean{2\mu(x)\beps_h(\bu)}\cdot\jump{\bv}
-\epsilon \int_{\cF^*_h}\mean{2\mu(x)\beps_h(\bv)}\cdot\jump{\bu},\\
 J_h(\bu,\bv)&:= \int_{\cF^*_h}\frac{\texttt{a}_S}{h_{\cF}}{2\mu(x)}\jump{\bu}\cdot\jump{\bv}.
\end{align*}
The following result provides an estimate for the term  $C_{h}(\cdot,\cdot)$ in terms of the indicator $\eta_{J_K}$.
\begin{lemma}\label{aplem11}
Suppose $\bu_h\in \H_h$ and $\bv_h^c\in \H_h^c$; then, the following relation holds:
$$
C_{h}(\bu_h,\bv_h^c)\lesssim|\epsilon| \texttt{a}_S^{-1/2}\left(\sum_{K\in\mathcal{T}_h}\eta_{J_K}^2\right)^{1/2}\vertiH{\bv^c_h}.
$$
\end{lemma}
\begin{proof}
Since  $\bv_h^c\in \H_h^c$, it follows that 
$$
C_{h}(\bu_h,\bv_h^c)=-\epsilon \int_{\cF^*_h}\mean{2\mu(x)\beps_h(\bv)}\cdot\jump{\bu}.
$$
Utilizing the Cauchy-Schwarz inequality yields:
$$
\begin{aligned}
	C_{h}(\bu_h,\bv_h^c)&\lesssim |\epsilon|\left(\sum_{\cF\in\cF_h^*} \gamma_h^{-1} \|[\![ (2\mu)^{1/2} \nabla\bv_h^c]\!]\|_{0,\cF}^2\right)^{1/2}
\left(\sum_{\cF\in\cF_h^*} \gamma_h \|(2\mu)^{1/2}[\![\bu_h]\!]\|_{0,\cF}^2\right)^{1/2}\hspace{-0.35cm}.
\end{aligned}
$$
Applying a trace estimate in conjunction with a discrete inverse inequality for an edge/face $\cF\in\cF_h^*$, where $\cF=K_1\cap K_2$ if $\cF_h^0$ and $\cF = K_1$, $K_2=\emptyset$ if $\cF\subset\partial\Omega$, results in:
$$
\|\nabla\bv_h^c\|_{0,E}\lesssim h_{K}^{-1/2}\|\nabla\bv_h^c\|_{0,K_1\cup K_2}.
$$
Thus we have
$$
C_{h}(\bu_h,\bv_h^c)\lesssim |\epsilon|\texttt{a}_S^{-1/2} \left(\sum_{K\in\mathcal{T}_h}\|(2\mu)^{1/2}\nabla\bv_h^c\|^2_{0,K}\right)^{1/2}\,\left(\sum_{K\in\mathcal{T}_h}\eta_{J_K}^2\right)^{1/2}. 
$$
This concludes the proof.
\end{proof}
Let us denote by  ${\Pi}_h : \H \to \H_h^c$ the Scott-Zhang interpolation operator \cite{MR1011446} {(see also   \cite[eq. (4.5) and eq. (4.6)]{GKDS})}, which is stable
$\|\nabla({\Pi}_h\bv)\|_0\lesssim\|\nabla\bv\|_0$
and satisfies the following
interpolation property  
 \begin{equation}\label{approxlem12}
\sum_{K\in\mathcal{T}_h} h_{K}^{-2}\|\bv-{\Pi}_h\bv\|^{2}_{0,K} + 
\sum_{\cF\in\cF_h^*} h_{\cF}^{-1}\|\bv-\Pi_h\bv\|^{2}_{0,E} \lesssim \|\nabla\bv\|_0^2,
 \end{equation}
 for any $\bv\in\H$.
\begin{lemma}\label{reallem123}
Let  $\bv_h^c=\Pi_h\bv\in\H_h^c$, where $\Pi_h$ denotes the Scott-Zhang interpolation of $\bv\in \H$. For any $\bu_h\in \H_h$, $p_h\in Q_h$, and $\kappa_h\in\mathbb{R}^+$, the following estimate holds:
$$
\kappa_h(\bu_h,\bv-\bv_h^c)-(2\mu_h\beps_h(\bu_h),\beps_h(\bv-\bv_h^c))+(p_h,\nabla\cdot(\bv-\bv_h^c))
\lesssim \eta\vertiH{\bv}.
$$
\end{lemma}
\begin{proof}
Employing integration by parts on each element $K\in\mathcal{T}_h$, we obtain:
\begin{multline*}
\kappa_h(\bu_h,\bv-\bv_h^c)-(2\mu_h\beps_h(\bu_h),\beps_h(\bv-\bv_h^c))+(p_h,\nabla\cdot(\bv-\bv_h^c))\\
\qquad=\sum_{K\in\mathcal{T}_h}\int_{K}(\kappa_h\bu_h+\nabla\cdot(2\mu_h\beps_h( \bu_h))-\nabla p_h)(\bv-\bv_h^c)d\bx\\
\qquad\qquad
+\sum_{K\in\mathcal{T}_h}\int_{\partial K}(p_h\bI - 2\mu_h\beps_h({\bu}_h)\textbf{n}_K\cdot(\bv-\bv_h^c)d\bs=T_1+T_2.
\end{multline*}
Now the task is to estimate the terms $T_1$ and $T_2$. Let us begin with $T_1$. By applying the Cauchy-Schwarz inequality and \eqref{approxlem12}, we arrive at:
\begin{multline*}
T_1\lesssim\left(\sum_{K\in\mathcal{T}_h}\|h_K(2\mu_h)^{-1/2}(\kappa_h\bu_h+\nabla\cdot2\mu_h\beps_h( \bu_h)-\nabla p_h)\|_{0,K}^2\right)^{1/2}\\
\times
\left(\sum_{K\in\mathcal{T}_h}h_{K}^{-2}\|(2\mu_h)^{1/2}(\bv-\bv_h^c)\|^{2}_{0,K}\right)^{1/2}
\lesssim\left(\sum_{K\in\mathcal{T}_h}\eta_{R_K}^2\right)^{1/2}\vertiH{\bv}.
\end{multline*}
Given that $(\bv-\bv_h^c)|_{\partial\Omega}=\boldsymbol{0}$, we can express $T_2$ in terms of a sum over interior facets:
$$
T_2= \sum_{\cF_h^0}\int_E[\![(p_h\bI-2\mu_h\beps\bu_h)\bn]\!](\bv-\bv_h^c)d\bs.
$$
Once more, the application of the Cauchy-Schwarz inequality and \eqref{approxlem12} results in:
\begin{multline*}
T_2\lesssim\left(\sum_{\cF_h^0}h_{\cF}\|(2\mu_h)^{-1/2}[\![(p_h\bI-2\mu_h\beps\bu_h)\bn]\!]\|^{2}_{0,\cF}\right)^{1/2}\\
\times\left(\sum_{\cF_h^0}
h_{\cF}^{-1}\|(2\mu_h)^{1/2}(\bv-\bv_h^c)\|^{2}_{0,\cF}\right)^{1/2}
\lesssim \left(\sum_{K\in\mathcal{T}_h}\eta_{\cF_K}^2\right)^{1/2}\vertiH{\bv}.
\end{multline*}
Consolidating the aforementioned estimates establishes the desired result.
\qed\end{proof}

\begin{lemma}\label{lemma5}
Let $(\kappa,(\bu,p))\in\mathbb{R}^+\times \H\times \Q $ be the solution of  the continuous eigenvalue problem \eqref{def:limit_system_eigen_complete}
and  let $(\kappa_h,(\bu_h,p_h))\in\mathbb{R}^+\times \H_h\times \Q_h$ be the solution of the IPDG discretization given by  \eqref{eq:weak_stokes_system_dg}. Then we have the following upper bound
for the conforming velocity and pressure errors
$$
\vertiii{(\bu-\bu_h^c,p-p_h)}_{\H(h)\times Q}\lesssim \eta + \|(2\mu_h)^{-1/2}(\kappa\bu-\kappa_h\bu_h)\|_{0}+\Theta.
$$
\end{lemma}
\begin{proof}
Using \cite[Lemma 2.2]{khan2023finite}, there exists a pair $(\bv,q)\in \H\setminus\{0\}\times \Q$ such that
$$
\vertiii{(\bu-\bu_h^c,p-p_h)}_{\H(h)\times Q}^2\lesssim {\mathcal{A}}(\bu-\bu_h^c,p-p_h;\bv,q),
$$
and $\vertiii{(\bv,q)}_{\H(h)\times Q}\lesssim \vertiii{(\bu-\bu_h^c,p-p_h)}_{\H(h)\times Q}$.
Since $\bu,\bu_h^c,\bv\in \H$, we have
\begin{multline*}
{\mathcal{A}}(\bu-\bu_h^c,p-p_h;\bv,q)=(2\mu\beps(\bu-\bu_h^c),\beps(\bv))
-(p-p_h,\nabla\cdot\bv)\\
-(q,\nabla\cdot(\bu-\bu_h^c))
 -(q, 1/\lambda(p-p_h)).
\end{multline*}
Based on \eqref{eq:eigen_A}, we derive
\begin{multline*}
\mathcal{A}(\bu-\bu_h^c,p-p_h;\bv,q)=\kappa(\bu,\bv) 
-(2\mu\beps(\bu_h^c),\beps(\bv))+(p_h,\nabla\cdot\bv)
+(q,\nabla\cdot\bu_h^c)+(q, 1/\lambda p_h)\\
=\kappa_h(\bu_h,\bv)+(\kappa\bu-\kappa_h\bu_h,\bv) 
-(2\mu_h\beps_h(\bu_h),\beps(\bv))+(p_h,\nabla\cdot\bv)\\
+(q,\nabla\cdot\bu_h)+(q, 1/\lambda p_h)-(q,\nabla\cdot\bu_h^r)+(2\mu\beps_h(\bu_h^r),\beps(\bv))
-(2(\mu-\mu_h)\beps_h(\bu_h),\beps(\bv)).
\end{multline*}

Consider $\bv_h^c=\Pi_h\bv\in\H_h^c$ as the Scott-Zhang interpolation of $\bv$. The utilization of (\ref{eq:dg_eigen}) results in:
\[0=\kappa_h(\bu_h,\bv_h^c)-(2\mu\beps_h(\bu_h),\beps_h(\bv_h^c))-C_h(\bu_h,\bv_h^c)+(p_h,\nabla\cdot\bv_h^c).\] 
Then we have $\mathcal{A}(\bu-\bu_h^c,p-p_h;\bv,q)=T_1+T_2+T_3+T_4$,
where the terms $T_i$, $i=1,2,3,4$ are defined by 
\begin{align*}
T_1&:=\kappa_h(\bu_h,\bv-\bv_h^c)
-(2\mu_h\beps_h(\bu_h),\beps(\bv-\bv_h^c))+(p_h,\nabla\cdot\bv-\bv_h^c)\\
&+(q,\nabla\cdot\bu_h)+(q, 1/\lambda p_h),\\
T_2&:=(q,\nabla\cdot\bu_h^r)+(2\mu\beps_h(\bu_h^r),\beps(\bv)),\\
T_3&:=C_h(\bu_h,\bv_h^c),\\
T_4&:=(\kappa\bu-\kappa_h\bu_h,\bv)-(2(\mu-\mu_h)\beps_h(\bu_h),\beps(\bv)). 
\end{align*}
By employing Lemma \ref{reallem123}, we obtain $T_1\lesssim \eta\vertiH{\bv}.$
The Cauchy-Schwarz inequality, along with \eqref{realilem1}, demonstrates:
$$
T_2\lesssim\vertiH{\bu_h^r}\vertiii{(\bv,q)}_{\H(h)\times Q}\lesssim\eta\vertiii{(\bv,q)}_{\H(h)\times Q}.
$$
By applying Lemma \ref{aplem11} to establish the bound of $T_3$, we obtain $T_3\lesssim\texttt{a}_S^{-1/2}\eta\vertiH{\bv}$.
The application of Cauchy-Schwarz and the Poincaré inequality results in:
$$
T_4\lesssim\left( \|(2\mu_h)^{-1/2}(\kappa\bu-\kappa_h\bu_h)\|_{0}+\Theta\right) \vertiH{\bv}
$$

Finally, with the combination of the aforementioned information with the estimate $\vertiii{(\bv,q)}_{\H(h)\times Q} \lesssim \vertiii{(\bu-\bu_h^c,p-p_h)}_{\H(h)\times Q}$, 
we attain the desired outcome. 
\end{proof}
\begin{theorem}\label{realiab}
Let $(\kappa,(\bu,p))\in\mathbb{R}^+\times \H\times \Q $ be the solution of  the continuous eigenvalue problem \eqref{def:limit_system_eigen_complete}
and  let $(\kappa_h,(\bu_h,p_h))\in\mathbb{R}^+\times \H_h\times \Q_h$ be the solution of the IPDG discretization given by  \eqref{eq:weak_stokes_system_dg}. Then the following a posteriori error bound is derived:
$$
\vertiii{(\bu-\bu_h^c,p-p_h)}_{\H(h)\times Q}\lesssim \eta+ \|(2\mu_h)^{-1/2}(\kappa\bu-\kappa_h\bu_h)\|_{0}+\Theta,
$$
where the hidden constant in the expression is independent 
of the penalty parameter $\texttt{a}_S$, provided that $\texttt{a}_S$ is sufficiently large and satisfies $\texttt{a}_S\geq 1$.
\end{theorem}
\begin{proof}
The assertion is readily established by jointly applying Lemma~\ref{lemma5} and \eqref{realilem1}.
\end{proof}
Applying the algebraic identity in conjunction with Theorem \ref{realiab} results in the following theorem.
\begin{theorem}\label{cor11}
{If $\bu\in \H^2(\mathcal{T}_h)$ and $p\in \H^{1}(\mathcal{T}_h)$,}  then the eigenvalue error satisfies
$$
 |\kappa-\kappa_h|\lesssim \eta^2+ \|(2\mu_h)^{-1/2}(\kappa\bu-\kappa_h\bu_h)\|_{0}^2+\Theta^2.
$$
\end{theorem}

 \subsection{Efficiency}
The efficiency analysis adheres to conventional arguments grou-nded in bubble functions \cite{MR1885308,MR3059294,MR4071826,khan2023finite}. Consequently, we solely provide the theorem statement for efficiency without presenting its proof.
\begin{theorem}\label{effice}
Assuming $(\kappa,(\bu, p))$ is the solution in the functional spaces $\mathbb{R}_+\times\H \times \Q $ for the eigenvalue problem\eqref{def:limit_system_eigen_complete}, and $(\kappa_h,(\bu_h, p_h))$ is its DG-approximation obtained through \eqref{eq:weak_stokes_system_dg}, with $\eta$ the a posteriori error estimator defined in \eqref{errest1}, the following efficiency bound holds
$$
\eta\lesssim \vertiii{(\bu-\bu_h^c,p-p_h)}_{\H(h)\times Q}+\Theta+h.o.t.
$$
The hidden constant in the expression is independent 
of the penalty parameter $\texttt{a}_S$, provided that $\texttt{a}_S$ is sufficiently large and satisfies $\texttt{a}_S\geq 1$.
\end{theorem}
\section{Numerical experiments}
\label{sec:numerics}

In this section we report some numerical tests in order to assess the performance of the proposed numerical method in the computation of the eigenvalues of problem \eqref{eq:weak_stokes_system_dg}. These results have been obtained using the FEniCSx project \cite{barrata2023dolfinx,scroggs2022basix}, while the meshes are constructed using GMSH \cite{geuzaine2009gmsh}.The convergence rates of the eigenvalues have been  obtained with a standard least square fitting and highly refined meshes.  

We denote by $N$ the mesh refinement level, whereas $\texttt{dof}$ denotes the number of degrees of freedom.  We denote by $\kappa_{h_i}$ the i-th discrete eigenfrequency, and by $\widehat{\kappa}_{h_i}:=\kappa_{h_i}/(1+\nu)$ the unscaled version of $\kappa_{h_i}$.

Hence, we denote the error on the $i$-th eigenvalue by $\err( \widehat{\kappa}_i)$ with 
$$
	\err(\widehat{\kappa}_i):=\vert \widehat{\kappa}_{h_i}-\widehat{\kappa}_{i}\vert.
$$
We remark that computing $\sqrt{\widehat{\kappa}_{h_i}}$ or $\sqrt{\widehat{\kappa}_{i}}$ gives the corresponding discrete and exact eigenfrequency.

With the aim of assessing the performance of our estimator,  we consider domains   with singularities in two and three dimensions in order to observe the improvement of the convergence rate.  On each adaptive iteration, we use the blue-green marking strategy to refine each $T'\in \CT_{h}$ whose indicator $\eta_{T'}$ satisfies
$$
\eta_{T'}\geq 0.5\max\{\eta_{T}\,:\,T\in\CT_{h} \}.
$$
We define the effectivity indexes with respect to $\eta$ and the eigenvalue $\widehat{\kappa_i}$ by 
$\eff(\widehat{\kappa_i}):=\err(\widehat{\kappa_i})/\eta^2$.

We note that other spectral analysis using DG methods introduce spurious eigenvalues when the stabilization parameter is not correctly chosen (see for instance \cite{MR3962898,MR4077220}). In the first part of this section, we study the safe range where the methods is spurious-free by comparing with references values from the literature. We will analyze the influence stabilization parameter $\texttt{a}_S$.  More precisely, the stabilization parameter $\texttt{a}_S$ in the  bilinear form $a_h(\cdot,\cdot)$ (cf. \eqref{eq:Ah}) will be chosen proportionally to the square of the polynomial degree $k$ as $\texttt{a}_S=\texttt{a} k^2$ with $\texttt{a}>0$. Also, taking into account the theoretical convergence rates presented Theorem \ref{teo:double-order}, the results in Table \ref{table-square2D-SIP} below and the results from the a posteriori analysis, the experiments will be focused on the SIP method in order to observe the recovery of the optimal double convergence rate.

\subsection{Square with bottom clamped boundary condition}\label{subsec:2D_bottomc_square}
We borrow this test from \cite{MR4077220}. The computational  domain is the unit square domain $\Omega:=(0,1)^2$, with $\Gamma_D:=\{(x,y)\in\partial\Omega\,:\, y=0\}$ and $\Gamma_N:=\partial\Omega\backslash\Gamma_D$. The Poisson ratio is taken to be $\nu=0.35$ and $E=1$. The boundary conditions are $\bu= \boldsymbol{0}$ on $\Gamma_D$ and  $\left(2\mu\beps(\bu)-p\boldsymbol{I}\right)\bn=\boldsymbol{0}$ on $\Gamma_N$. The test consider a uniform mesh with $N=8$. The spurious study for this case have been previously studied in \cite{MR4077220} in the context of a mixed discontinuous Galerkin scheme.

In Table \ref{tabla2:square-SIP-k1} we report the results of computing the first 10 lowest eigenvalues for different values of the stabilization parameter and different polynomial orders. The spurious eigenvalues are highlighted in bold numbers. It notes that the higher the penalization value is, the less spurious eigenvalues appear. Also, for $k=3$ the spurious appearance is considerably low.  For the NIP and IIP methods we observed a similar behavior as those of \cite{MR4077220}, so we refer to that study for further details. The experiment ends with Table \ref{table-square2D-SIP}, in which we study the convergence of the first two eigenvalues. Since we have a mixed boundary condition setting, then there is a singularity in the corners where Dirichlet changes to Neumann boundary. This is reflected in the convergence rate, where we observe a $\mathcal{O}(h^r),\, r\geq 1.2$ for all the values of $\nu$ selected. 

\begin{table}[hbt!]
	\centering
	{\setlength{\tabcolsep}{4.pt}\footnotesize
		\caption{Test  \ref{subsec:2D_bottomc_square}. Computed eigenfrequencies $\sqrt{\kappa_{h,i}}$ by using the  SIP method ($\varepsilon=1$) for polynomial degrees $k=1,2,3$ and mesh level $N=8$ on the unit square domain. The Poisson ratio for this case is $\nu=0.35$.}
		\begin{tabular}{|c|c|cccccc|}
			\toprule
			\hline
			$k$&Ref. \cite{MR4077220} &$\texttt{a}=$1/4 & $\texttt{a}=$1/2    & $\texttt{a}=$1    & $\texttt{a}=$2      & $\texttt{a}=$4 & $\texttt{a}=$8 \\ \hline\hline
			\multirow{9}{0.02\linewidth}{1}& 0.6808 & 0.6555 & 0.6612 & 0.6631 & 0.6738 & 0.6809 & 0.6881 \\
			& 1.6993 & $\bf{1.0131}$ & 1.6781 & 1.6914 & $\bf{0.7873}$ & 1.6991 & 1.7065 \\
			& 1.8222 & 1.6719 & 1.7973 & 1.7771 & 1.6887 & 1.8310 & 1.8426 \\
			& 2.9477 & 1.7880 & $\bf{2.0082}$ & $\bf{1.8554}$ & $\bf{1.8162}$ & 2.9702 & 3.0085 \\
			& 3.0181 & $\bf{2.0217}$ & $\bf{2.0735}$ & $\bf{2.7417}$ & 1.8440 & 3.0434 & 3.0720 \\
			& 3.4433 & 2.8190 & 2.8669 & 2.9291 & $\bf{1.9966}$ & 3.4740 & 3.5052 \\
			& 4.1418 & 2.9139 & 2.9480 & 2.9524 & 2.9651 & 4.2061 & 4.2830 \\
			& 4.6312 & $\bf{3.2805}$ & 3.3931  & $\bf{3.3762}$ & 3.0285 & 4.6810 & 4.7681 \\
			& 4.7616 & 3.3557  & 3.8363 & 3.4530 & 3.4592 & 4.8488 & 4.9367 \\
			& 4.7887 & $\bf{3.4927}$ & 4.3341 & $\bf{3.5020}$ & $\bf{3.5864}$ & 4.8720 & 4.9775 \\
			\hline
			\multirow{9}{0.02\linewidth}{2}& 0.6808 & 0.6776 & 0.6768 & 0.6701 & 0.6803 & 0.6816 & 0.6825 \\
			& 1.6993 & 1.6963 & $\bf{0.7360}$ & 1.6839 & 1.6985 & 1.7002 & 1.7013 \\
			& 1.8222 & 1.8213 & 1.6974 & 1.8217 & 1.8225 & 1.8226 & 1.8229 \\
			& 2.9477 & $\bf{2.7325}$ & 1.8208 & $\bf{2.5678}$ & 2.9481 & 2.9485 & 2.9489 \\
			& 3.0181 & 2.9482 & $\bf{2.2148}$ & 2.9309 & 3.0176 & 3.0203 & 3.0225 \\
			& 3.4433 & 3.0332 & 2.9466 & 3.0554 & 3.4435 & 3.4444 & 3.4450 \\
			& 4.1418 & 3.4410 & 3.0229 & $\bf{3.0947}$ & 4.1436 & 4.1459 & 4.1479 \\
			& 4.6312 & $\bf{3.8448}$ & 3.4425 & 3.4309 & 4.6336 & 4.6364 & 4.6395 \\
			& 4.7616 & $\bf{4.0832}$ & 4.1361 & $\bf{3.9605}$ & 4.7635 & 4.7665 & 4.7686 \\
			& 4.7887 & 4.1458 & 4.6241 & 4.1301 & 4.7899 & 4.7954 & 4.7992 \\
			\hline
			\multirow{9}{0.02\linewidth}{3}	& 0.6808 & 0.6760 & 0.6724 & 0.6808 & 0.6808 & 0.6814 & 0.6817 \\
			& 1.6993 & 1.6972 & 1.6929 & 1.6989 & 1.6993 & 1.6999 & 1.7003 \\
			& 1.8222 & 1.8220 & 1.8221 & 1.8222 & 1.8222 & 1.8223 & 1.8224 \\
			& 2.9477 & $\bf{2.0530}$ & 2.9477 & 2.9477 & 2.9477 & 2.9477 & 2.9477 \\
			& 3.0181 & 2.9476 & 2.9990 & 3.0179 & 3.0181 & 3.0191 & 3.0198 \\
			& 3.4433 & 3.0176 & 3.4420 & $\bf{3.2003}$ & 3.4433 & 3.4435 & 3.4436 \\
			& 4.1418 & 3.4427 & 4.1407 & 3.4434 & 4.1419 & 4.1422 & 4.1424 \\
			& 4.6312 & 4.1412 & 4.6258 & 4.1418 & 4.6313 & 4.6320 & 4.6325 \\
			& 4.7616 & 4.6289 & $\bf{4.6365}$ & 4.6312 & 4.7617 & 4.7619 & 4.7619 \\
			& 4.7887 & 4.7612 & 4.7614 & 4.7618 & 4.7887 & 4.7898 & 4.7904 \\
			\hline
			\bottomrule
	\end{tabular}}
	\label{tabla2:square-SIP-k1}
\end{table}

\begin{table}[hbt!]
	\centering 
	{\footnotesize
		\begin{center}
			\caption{Test \ref{subsec:2D_bottomc_square}. Convergence behavior of the first three lowest computed eigenvalues by using the  SIP method ($\varepsilon=1$) on the unit square domain . The stabilization parameter is set to be $\texttt{a} = 10$.}
			\begin{tabular}{|c|c c c c |c| c|}
				\toprule
				\hline
				$\nu$&$N=20$             &  $N=30$         &   $N=40$         & $N=50$ & Order & $\widehat{\kappa}_{\text{extr}}$ \\ 
				\hline
				\hline
				\multicolumn{7}{|c|}{$k=1$}\\
				\hline
				\multirow{2}{0.05\linewidth}{0.35}	&  0.6832  &     0.6821  &     0.6817  &     0.6815  & 1.50 &     0.6809  \\
				&1.7018  &     1.7008  &     1.7003  &     1.7000  & 1.41 &     1.6993  \\
				\hline
				\multirow{2}{0.05\linewidth}{0.49}	 & 0.7033  &     0.7018  &     0.7011  &     0.7008  & 1.27 &     0.6996  \\
				&1.8439  &     1.8412  &     1.8399  &     1.8392  & 1.33 &     1.8373  \\
				\hline
				\multirow{2}{0.05\linewidth}{0.50}	&  0.7056  &     0.7040  &     0.7033  &     0.7029  & 1.26 &     0.7016  \\
				&1.8557  &     1.8528  &     1.8515  &     1.8508  & 1.33 &     1.8487  \\
				\hline
				\bottomrule             
			\end{tabular}
	\end{center}}
	\smallskip
	\label{table-square2D-SIP}
\end{table}


\subsection{Clamped square with different materials} \label{subsec:square2D_two_domain}
In this test we study the behavior of the numerical scheme when we consider an heterogeneous media along the domain. For this, we consider two different materials on the unit square domain. The boundary conditions for this experiments are
$$
\bu = \boldsymbol{0} \text{ on } \Gamma_D:=(0,y) \cup (1,y), \, y\in[0,1],
$$ 
while $\left(2\mu\beps(\bu)-p\boldsymbol{I}\right)\bn=\boldsymbol{0}$ is assumed on the rest of the boundary.  The computational domain is decomposed such that $\Omega:=\Omega_1\cup\Omega_2$, where $\Omega_1:=(0,0)\times (0,1/2)$ and $\Omega_2:=(0,1/2)\times (1,1)$. 
A sample mesh for this decomposition is portrayed in Figure \ref{fig:square2D-sample-mesh}. The Young modulus and density parameters are then given by
\begin{equation}
	\label{eq:young-modulus}
	E(x):=\left\{
\begin{aligned}
	&7.72\cdot 10^{10} \text{Pa} \text{ in } \Omega_1,\\
	&1.10\cdot 10^{11} \text{Pa} \text{ in } \Omega_2,
\end{aligned},
\right.
\qquad
\rho(x) = \left\{
\begin{aligned}
	&19300\,\text{kg/m}^3 \text{ in } \Omega_1,\\
	&8850\, \text{kg/m}^3  \text{ in } \Omega_2,
\end{aligned}
\right.
\end{equation}
which consist of gold in $\Omega_1$, and copper on $\Omega_2$. By the results from the previous example, it is enough to consider $k=1$ to study the convergence on the a priori case. 

The convergence with uniform meshes is presented in Table \ref{table-square2D_different_materials-SIP}. Here, we observe a clear suboptimal rate of convergence for the four lowest computed eigenvalues. By observing Figure \ref{fig:square2D-sample-mesh} we note that the heterogeneity of the medium yields to a considerable difference in the displacement field for the first eigenmode with $\nu=0.35$ and $\nu=0.5$.  Also, the pressure contour shows, along with the corner singularities, two zones of high gradients near the intersection of the material discontinuity and the boundary domain.

When studying the a posteriori refinements for this case, we consider $k=1,2$ and a total of $15$ adaptive iterations. Figure \ref{fig:lshape-2D-adaptive} depicts several intermediate meshes for the chosen polynomials degrees.  We note that the refinements concentrate along the singularities  and the number of marked elements is lower for the higher degree. The adaptive refinement also makes evident the difference between the materials on the domain. Those results are followed by the error curves, presented in Figure \ref{fig:lshape2D-error}, where we observe an experimental rate of $\mathcal{O}(h^{k+1})$. The reliability and the efficiency of the estimator are evidenced on Figure \ref{fig:lshape2D-effectivity}, where a proper boundedness of $\texttt{eff}(\cdot)$ and a decay behaving like $\mathcal{O}(h^{k+1})$ for $\eta^2$ are observed.
\begin{figure}[hbt!]
	\centering
		\begin{minipage}{0.24\linewidth}\centering
	\includegraphics[scale=0.1, trim=22cm 3cm 22cm 5cm,clip]{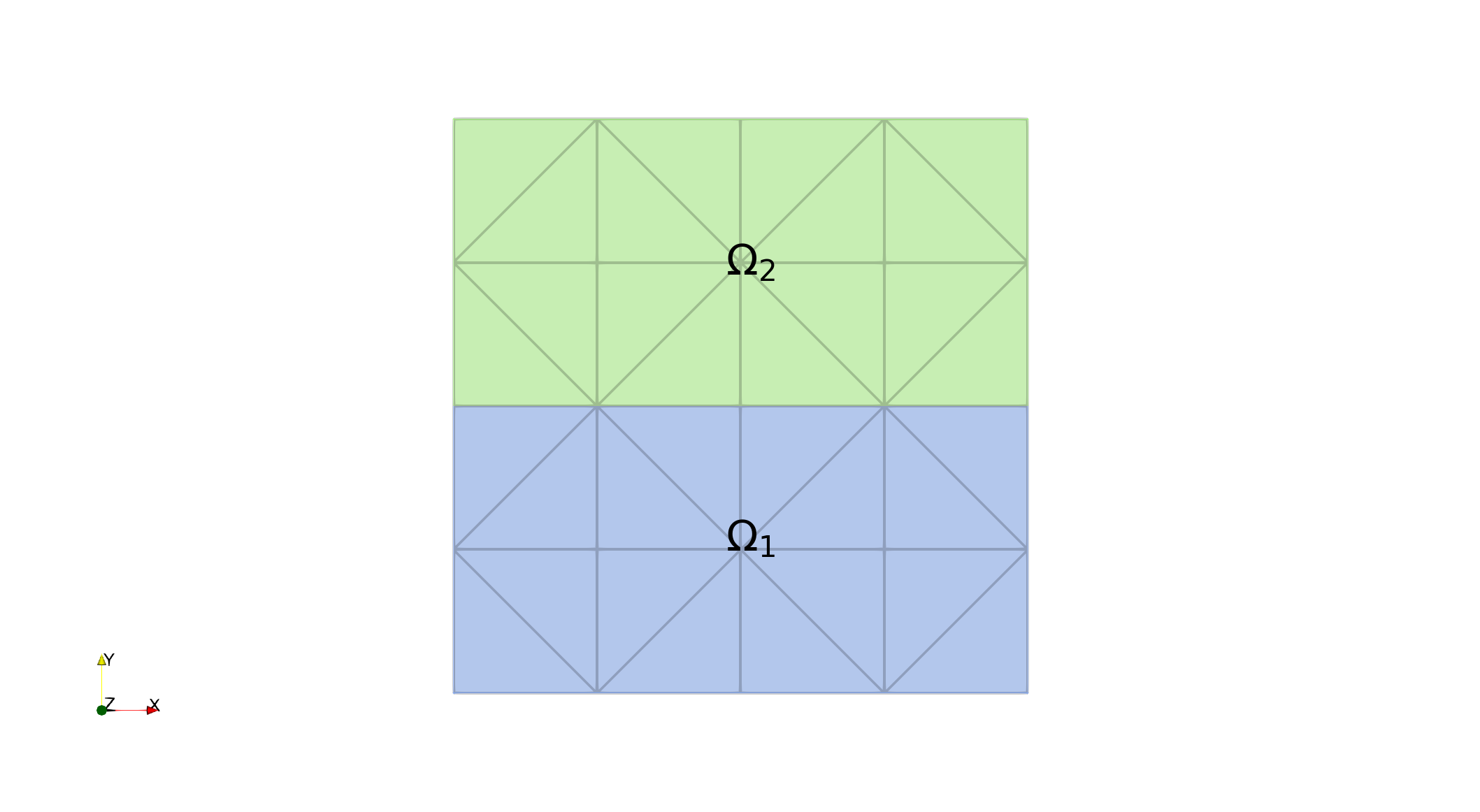}
\end{minipage}
		\begin{minipage}{0.24\linewidth}\centering
	\includegraphics[scale=0.1, trim=22cm 3cm 22cm 5cm,clip]{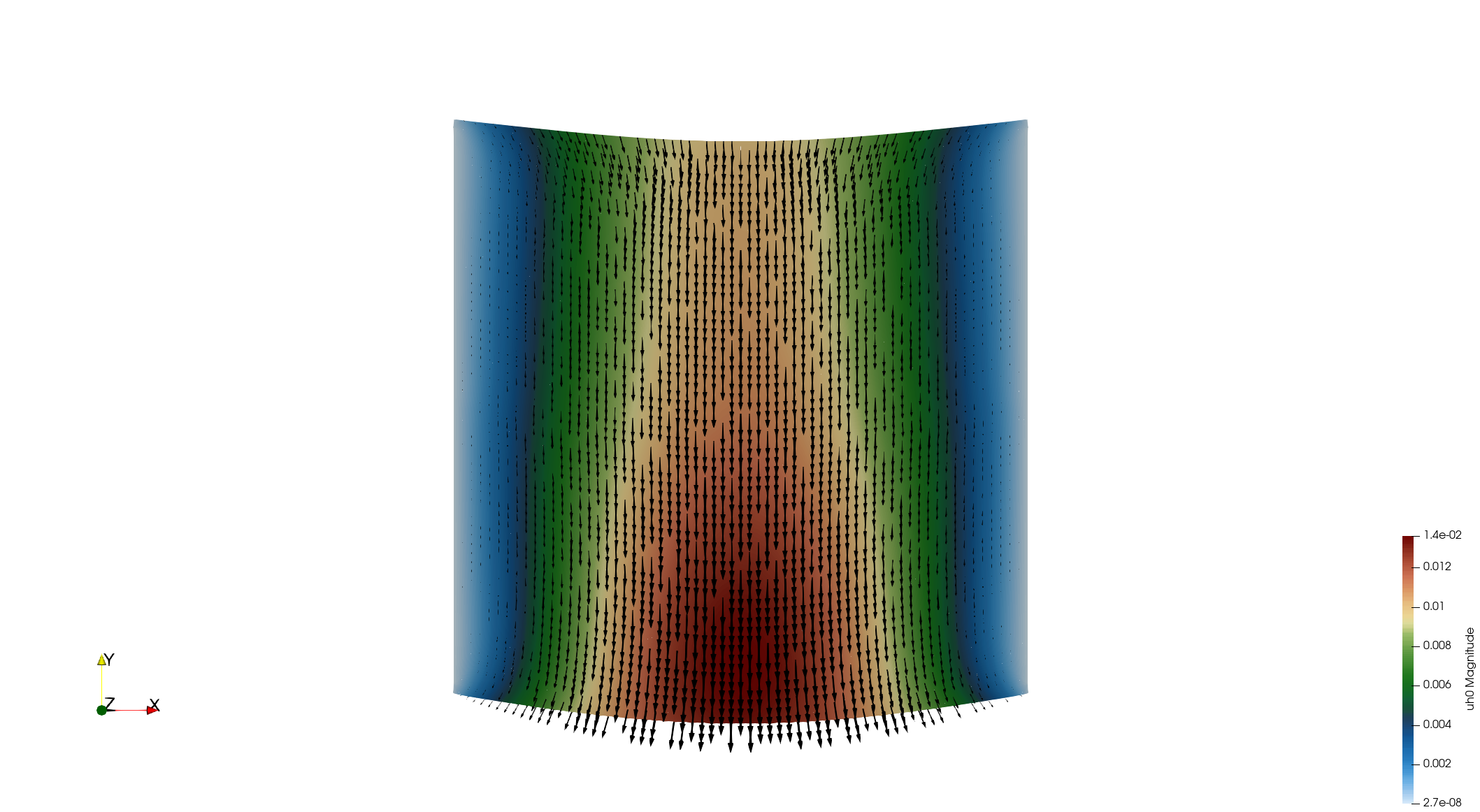}
\end{minipage}
		\begin{minipage}{0.24\linewidth}\centering
	\includegraphics[scale=0.1, trim=22cm 3cm 22cm 5cm,clip]{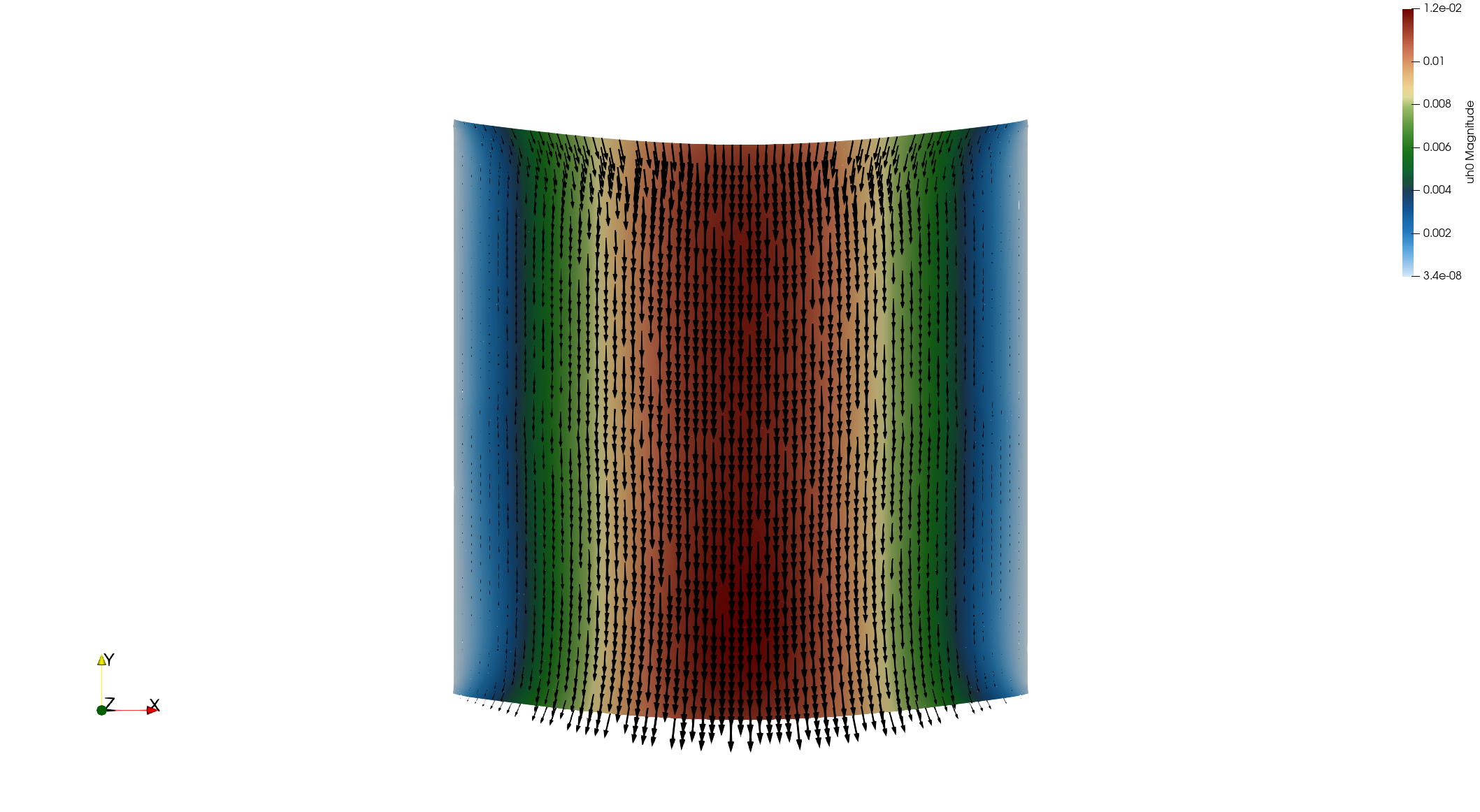}
\end{minipage}
		\begin{minipage}{0.24\linewidth}\centering
	\includegraphics[scale=0.1, trim=22cm 3cm 22cm 5cm,clip]{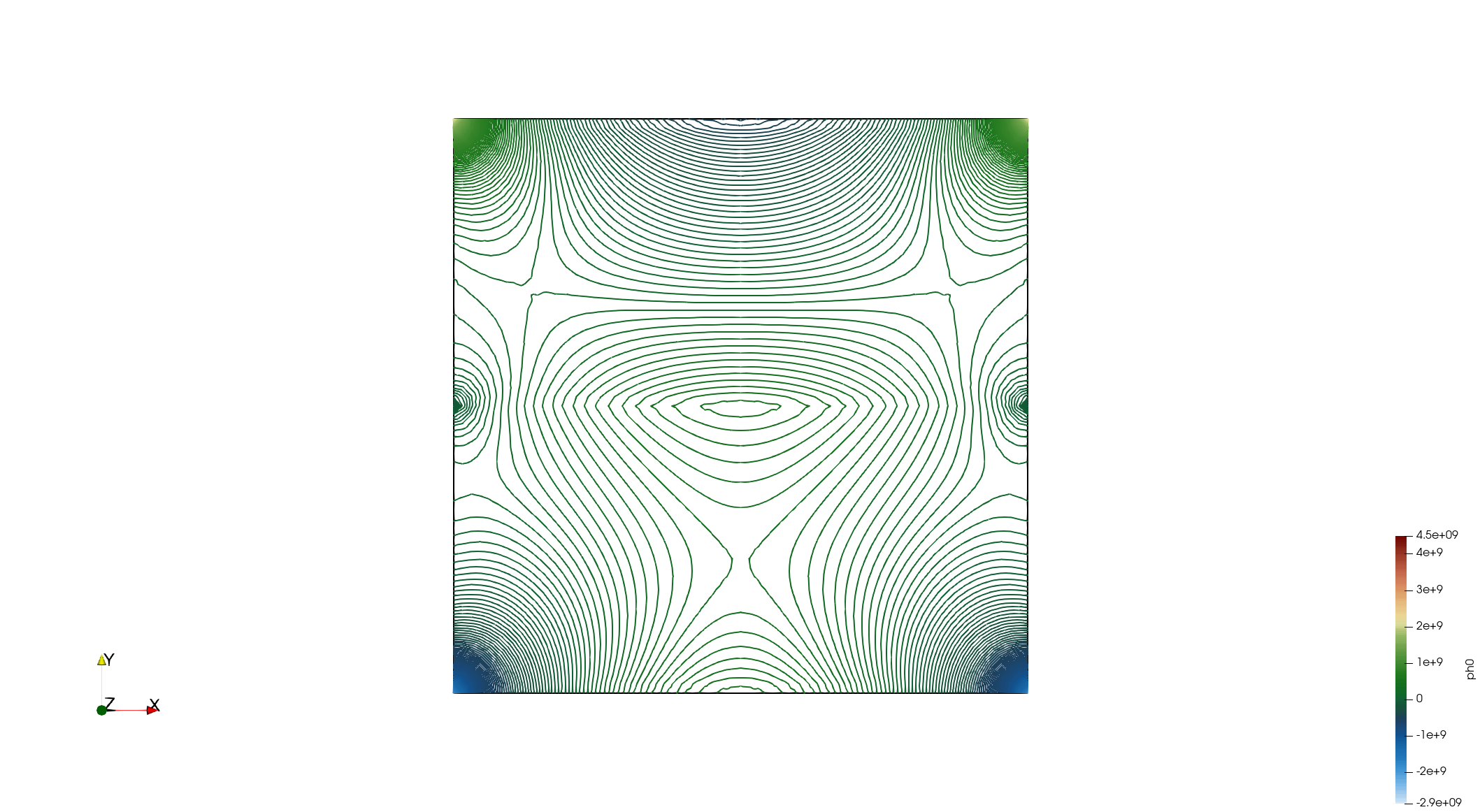}
\end{minipage}
	\caption{Test \ref{subsec:square2D_two_domain}. Left: Sample mesh for the subdivided computational domain with two materials. The next two figures correspond to the warped domain using the lowest mode $\bu_{1,h}$ for $\nu=0.35$ and $\nu=0.5$, respectively. The right figure correspond to the first pressure mode, which is similar for $\nu=0.35$ and $\nu=0.5$. }
	\label{fig:square2D-sample-mesh}
\end{figure}
\begin{table}[hbt!]
	\centering 
	{\footnotesize
		\begin{center}
			\caption{Test \ref{subsec:square2D_two_domain}. Convergence behavior of the first fourth lowest computed eigenvalues by using the  SIP method ($\varepsilon=1$) for $k=1$ on the unit square domain with different materials. The stabilization parameter is set to be $\texttt{a} = 10$.}
			\begin{tabular}{|c|c c c c |c| c|}
				\toprule
				\hline
				$\nu$&$N=8$             &  $N=16$         &   $N=32$         & $N=64$ & Order & $\widehat{\kappa}_{\text{extr}}$ \\ 
				\hline
				\hline
				\multirow{2}{0.05\linewidth}{0.35}	&   4506.7208  &  4457.1445  &  4439.2239  &  4433.2168  & 1.50 &  4429.6821  \\
				&7546.2296  &  7454.7702  &  7421.5220  &  7410.1874  & 1.49 &  7403.5352  \\
				&8043.1998  &  7864.1790  &  7812.3865  &  7798.2104  & 1.82 &  7792.2188  \\
				&10533.2314  & 10301.4824  & 10223.0867  & 10200.2366  & 1.62 & 10187.2085  \\
				\hline
				\multirow{2}{0.05\linewidth}{0.49}	 &   4462.9780  &  4419.4198  &  4401.2061  &  4394.2151  & 1.30 &  4389.1357  \\
				&7623.5227  &  7454.1769  &  7396.1249  &  7376.9013  & 1.57 &  7366.7357  \\
				&8692.8982  &  8521.8508  &  8468.1473  &  8451.5251  & 1.69 &  8443.9533  \\
				&11173.4976  & 10698.5980  & 10551.1393  & 10507.2968  & 1.74 & 10488.3190  \\
				\hline
				\multirow{2}{0.05\linewidth}{0.50}	&    4464.8940  &  4421.5136  &  4403.1190  &  4395.9793  & 1.28 &  4390.6467  \\
				&7611.7557  &  7442.4339  &  7384.3312  &  7365.0352  & 1.57 &  7354.9169  \\
				&8771.4492  &  8591.9091  &  8535.1211  &  8517.4054  & 1.68 &  8509.2642  \\
				&11166.2453  & 10690.8606  & 10543.1916  & 10499.1463  & 1.73 & 10479.2863  \\
				\hline
				\bottomrule             
			\end{tabular}
	\end{center}}
	\smallskip
	\label{table-square2D_different_materials-SIP}
\end{table}
\begin{figure}[hbt!]
	\centering
	\begin{minipage}{0.24\linewidth}
		\centering
		{\tiny $\nu=0.35, k=1, iter =10$}\\
		\includegraphics[scale=0.08,trim=56cm 2cm 56cm 2cm,clip]{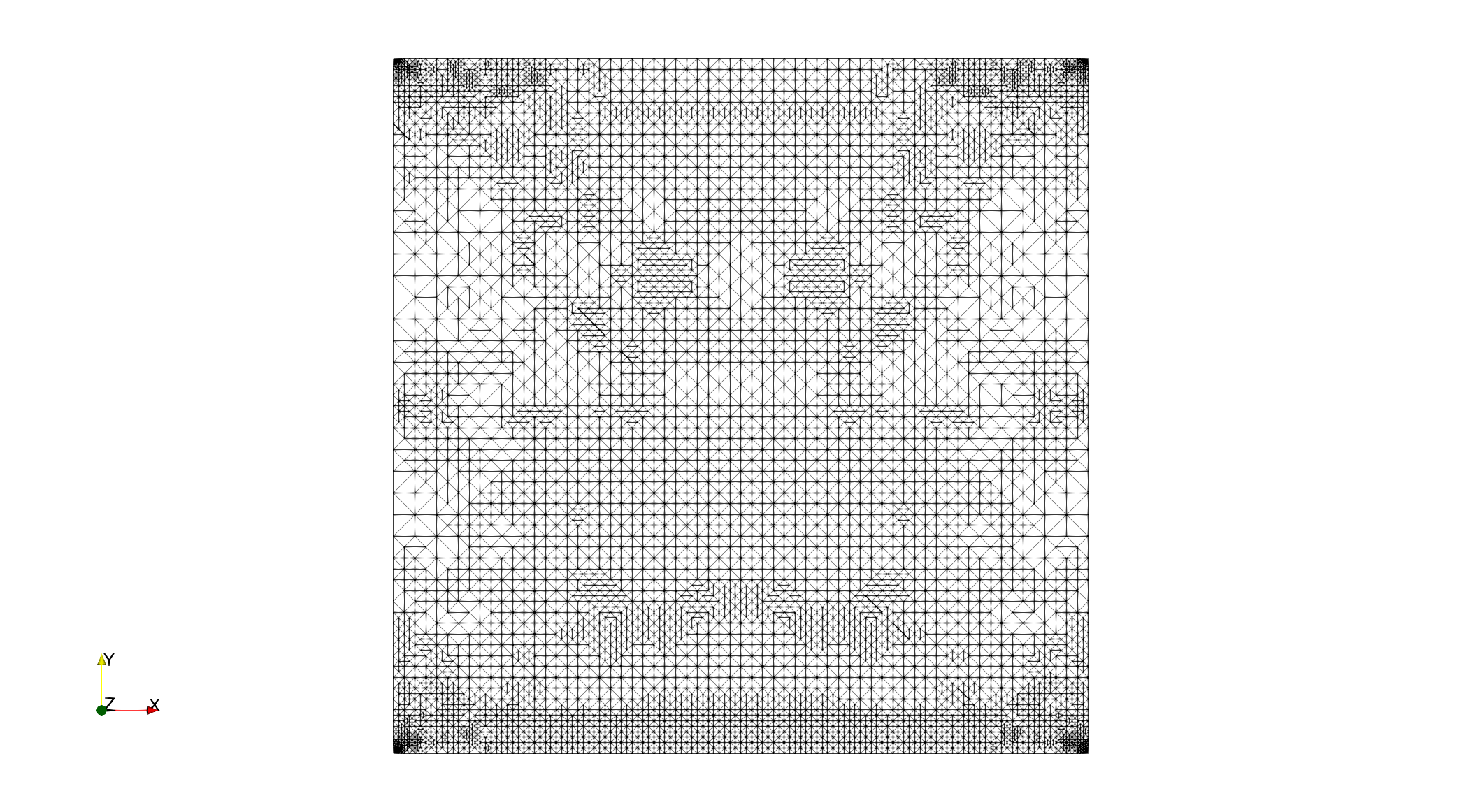}\\
	\end{minipage}
	\begin{minipage}{0.24\linewidth}
		\centering
		{\tiny $\nu=0.35, k=2, iter =10$}\\
		\includegraphics[scale=0.08,trim=56cm 2cm 56cm 2cm,clip]{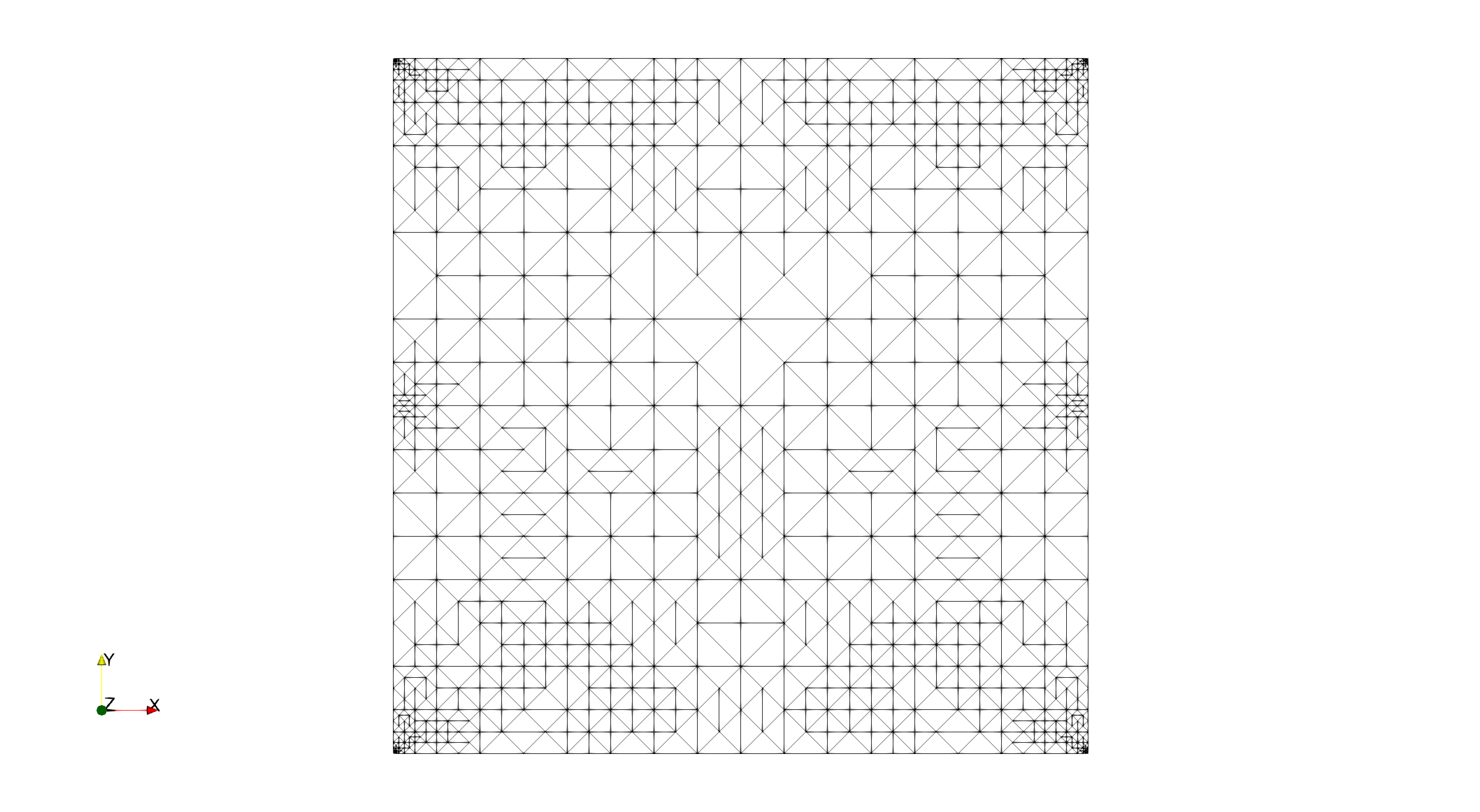}\\
	\end{minipage}
	\begin{minipage}{0.24\linewidth}
		\centering
		{\tiny $\nu=0.35, k=1, iter =15$}\\
		\includegraphics[scale=0.08,trim=56cm 2cm 56cm 2cm,clip]{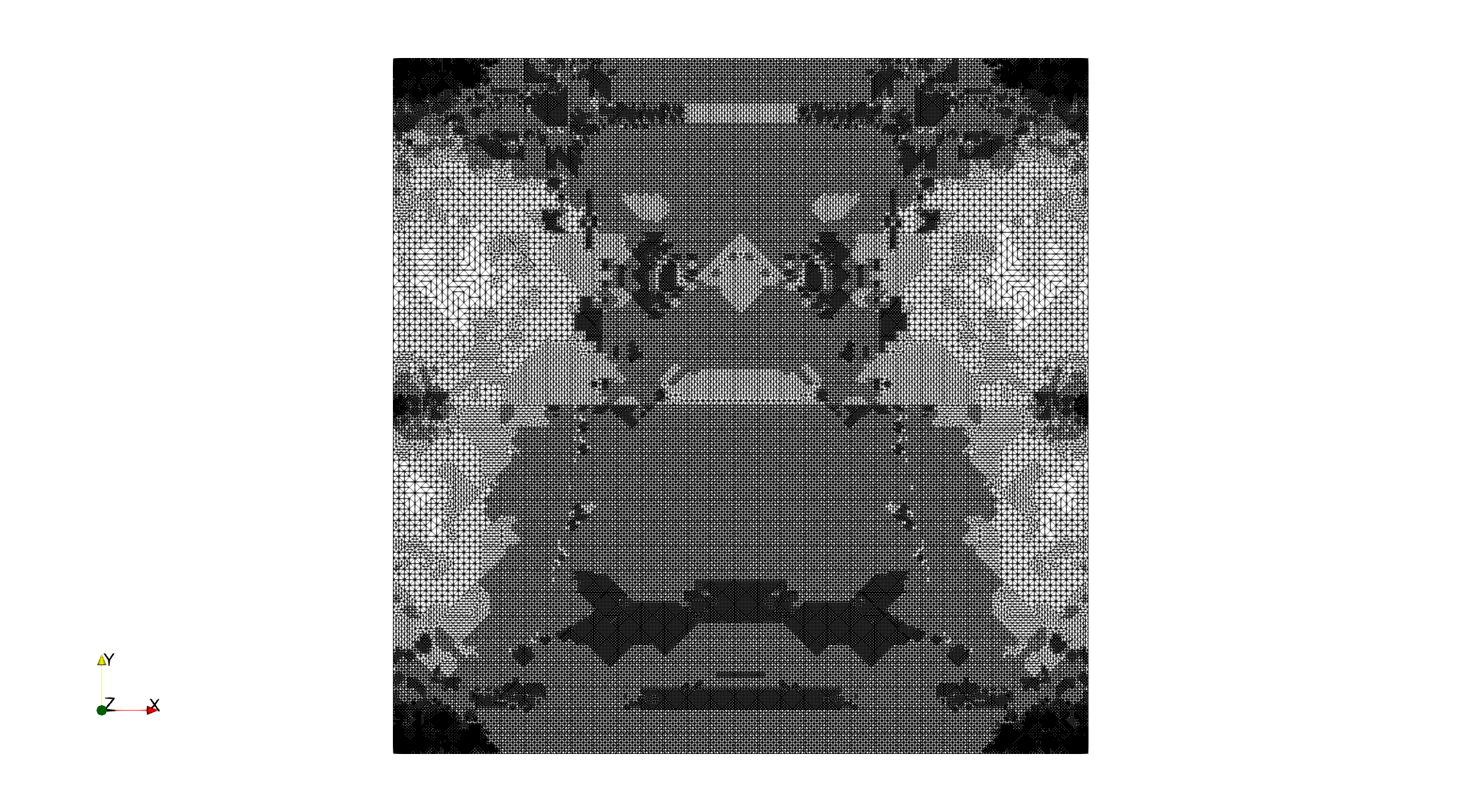}\\
	\end{minipage}
	\begin{minipage}{0.24\linewidth}
		\centering
		{\tiny $\nu=0.35, k=2, iter =15$}\\
		\includegraphics[scale=0.08,trim=56cm 2cm 56cm 2cm,clip]{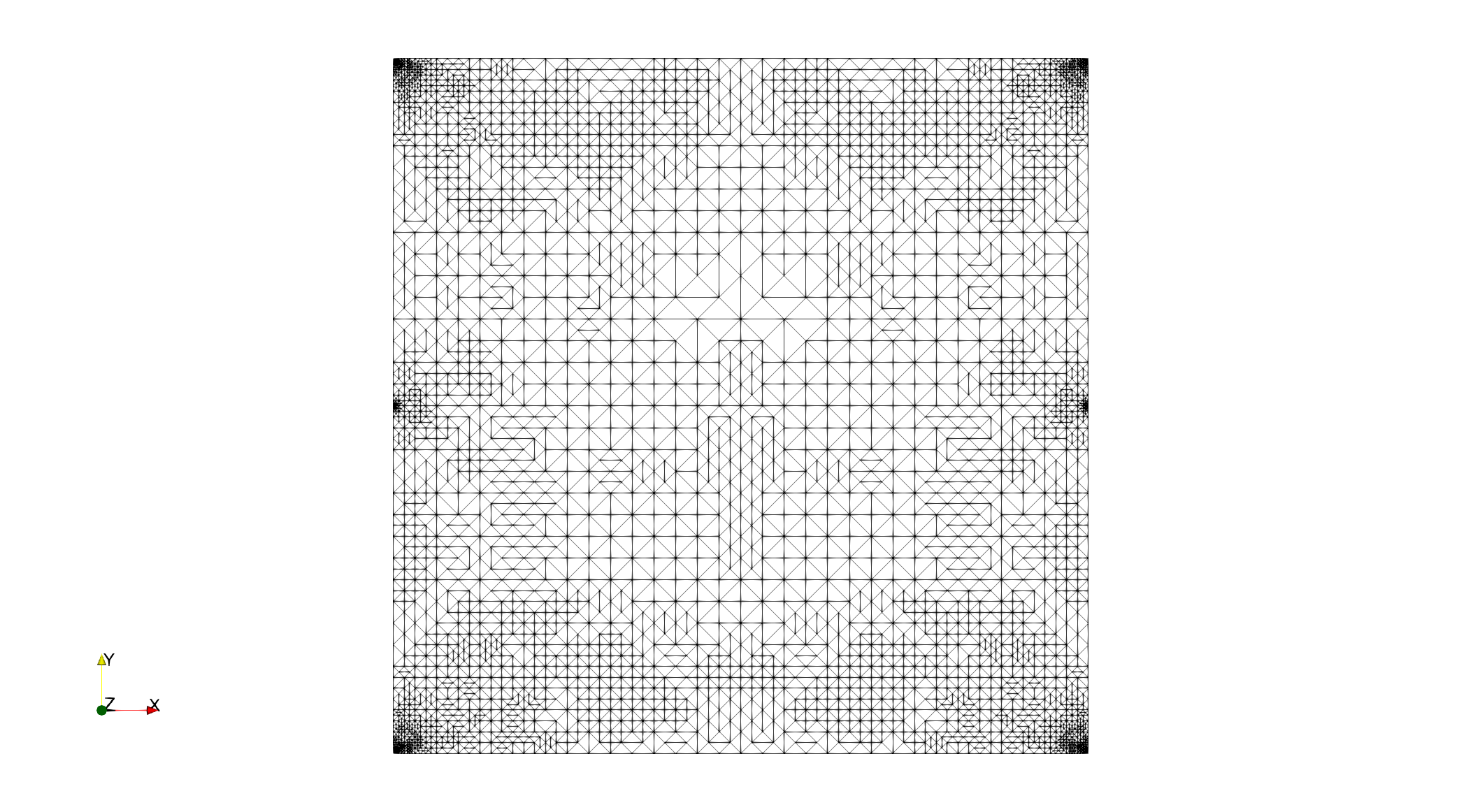}\\
	\end{minipage}\\
	\begin{minipage}{0.24\linewidth}
		\centering
		{\tiny $\nu=0.50, k=1, iter =10$}\\
		\includegraphics[scale=0.08,trim=56cm 2cm 56cm 2cm,clip]{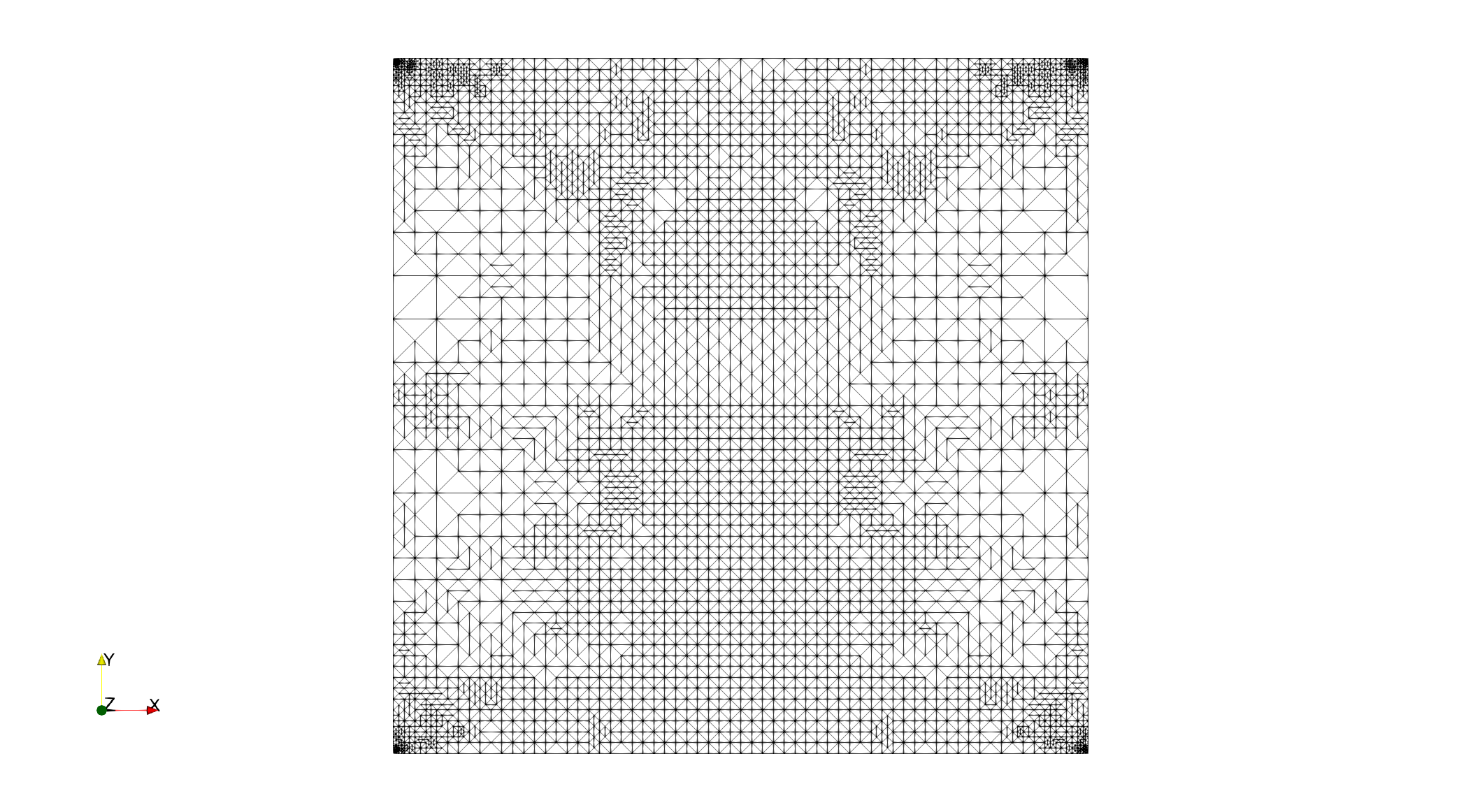}\\
	\end{minipage}
	\begin{minipage}{0.24\linewidth}
		\centering
		{\tiny $\nu=0.50, k=2, iter =10$}\\
		\includegraphics[scale=0.08,trim=56cm 2cm 56cm 2cm,clip]{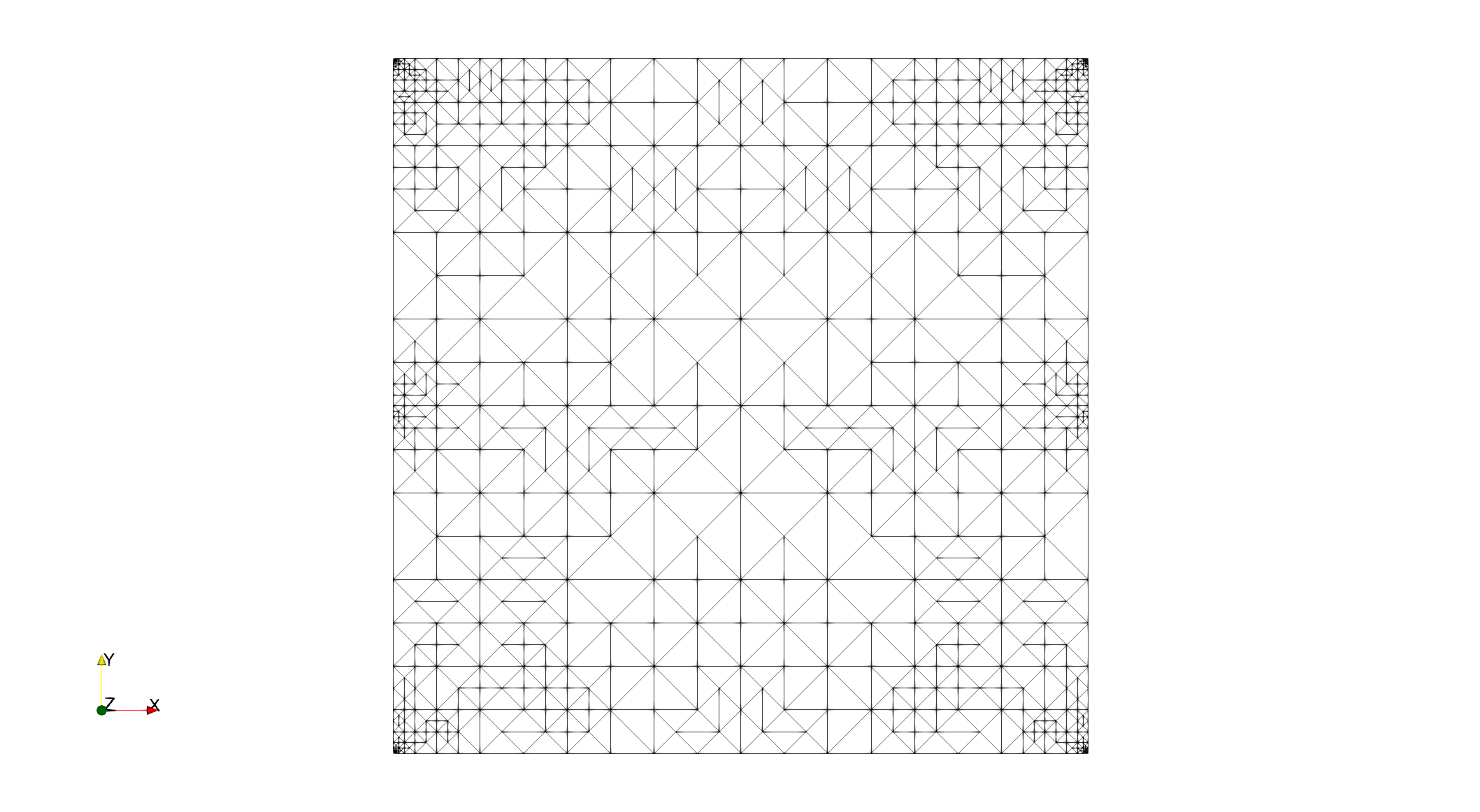}\\
	\end{minipage}
	\begin{minipage}{0.24\linewidth}
		\centering
		{\tiny $\nu=0.50, k=1, iter =15$}\\
		\includegraphics[scale=0.08,trim=56cm 2cm 56cm 2cm,clip]{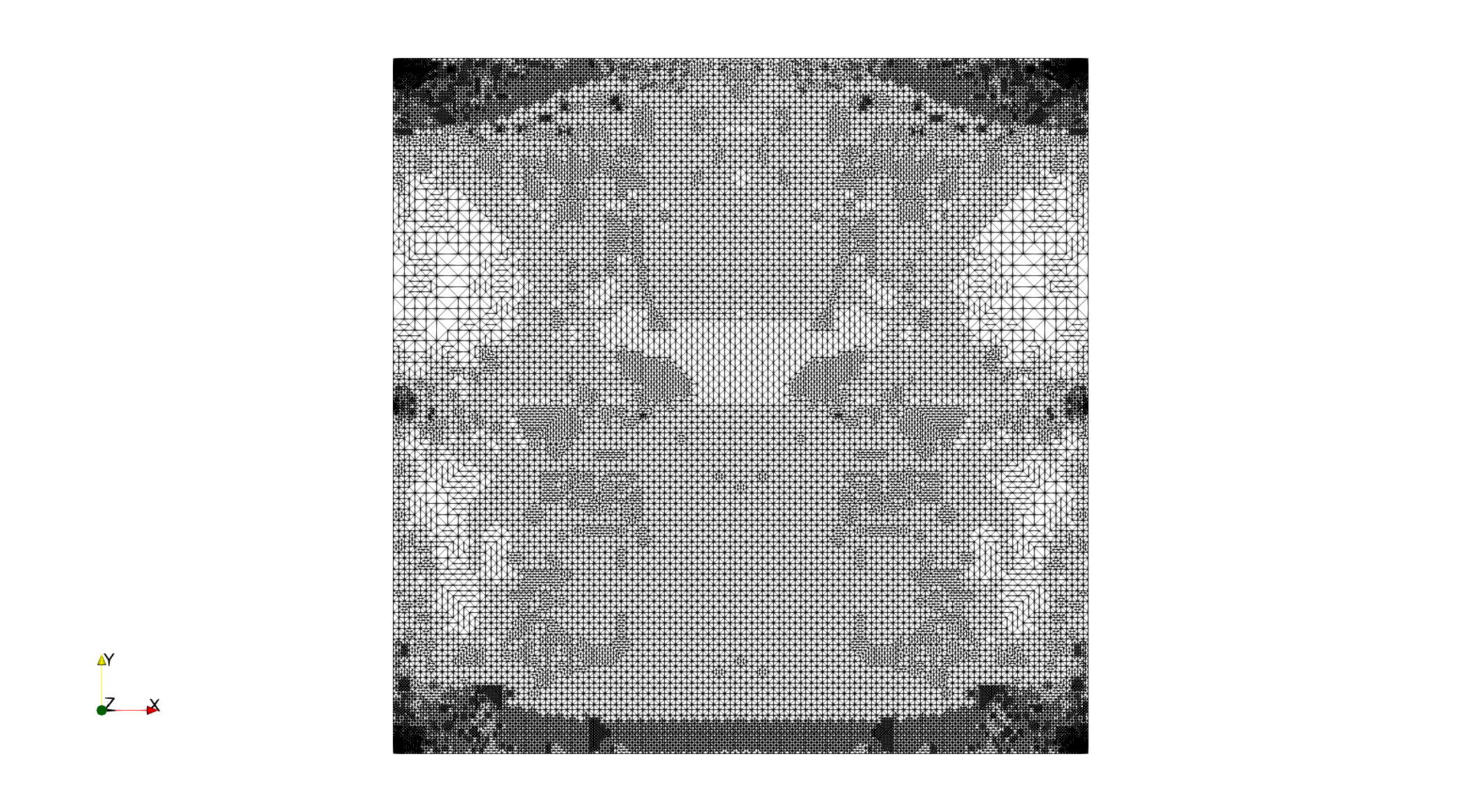}\\
	\end{minipage}
	\begin{minipage}{0.24\linewidth}
		\centering
		{\tiny $\nu=0.50, k=2, iter =15$}\\
		\includegraphics[scale=0.08,trim=56cm 2cm 56cm 2cm,clip]{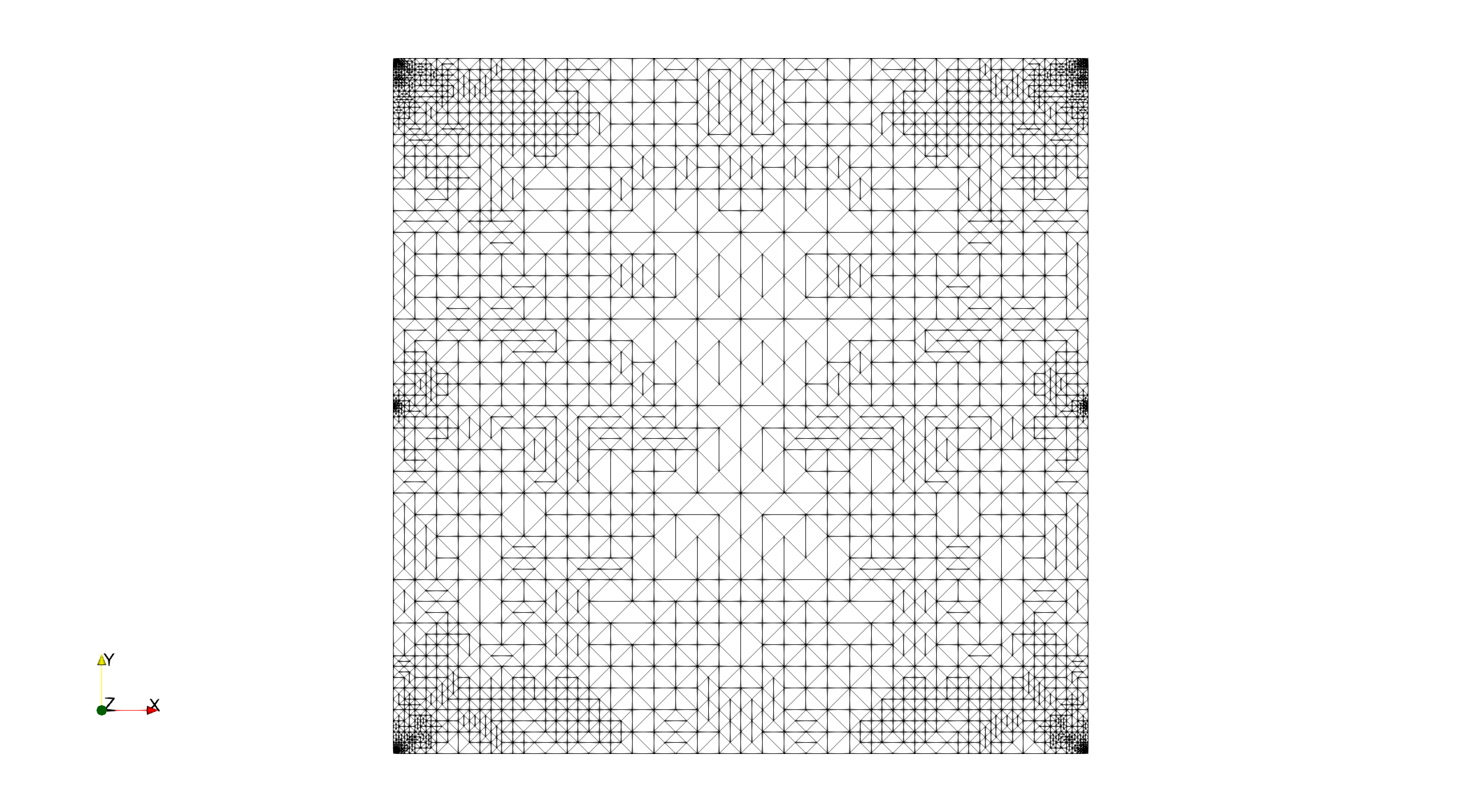}\\
	\end{minipage}
	\\
	\caption{Test \ref{subsec:square2D_two_domain}. Intermediate meshes of the square domain with different materials obtained with the adaptive algorithm and different values of $\nu$, with $k=1,2$. }
	\label{fig:lshape-2D-adaptive}
\end{figure}

\begin{figure}[hbt!]
	\centering
	\includegraphics[scale=0.42]{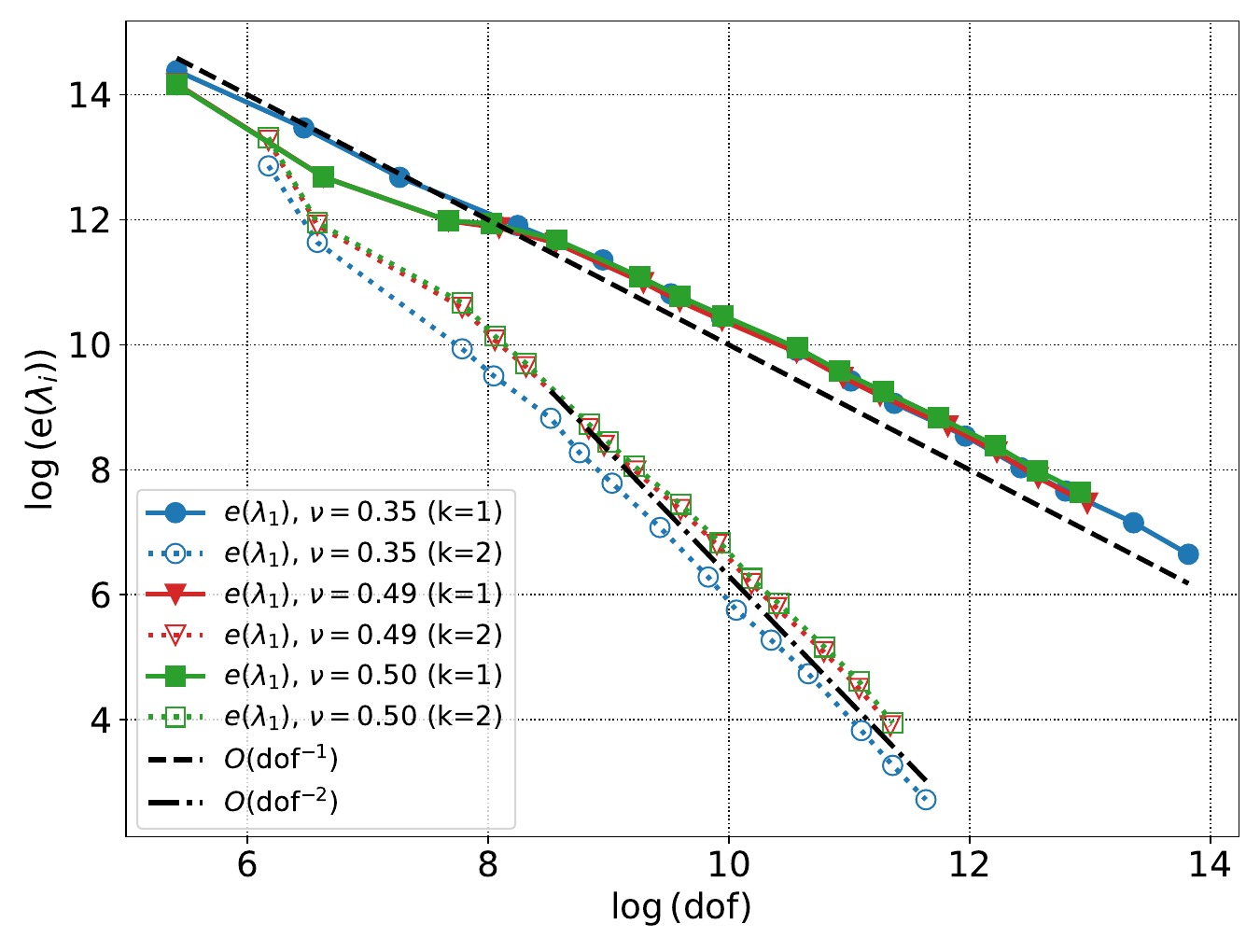}
	\caption{Test \ref{subsec:square2D_two_domain}.  Error curves obtained from the adaptive algorithm in the square domain compared with the lines $\mathcal{O}(\texttt{dof}^{-1})$ and $\mathcal{O}(\texttt{dof}^{-2})$.}
	\label{fig:lshape2D-error}
\end{figure}

\begin{figure}[hbt!]
	\centering
	\begin{minipage}{0.49\linewidth}\centering
		\includegraphics[scale=0.4]{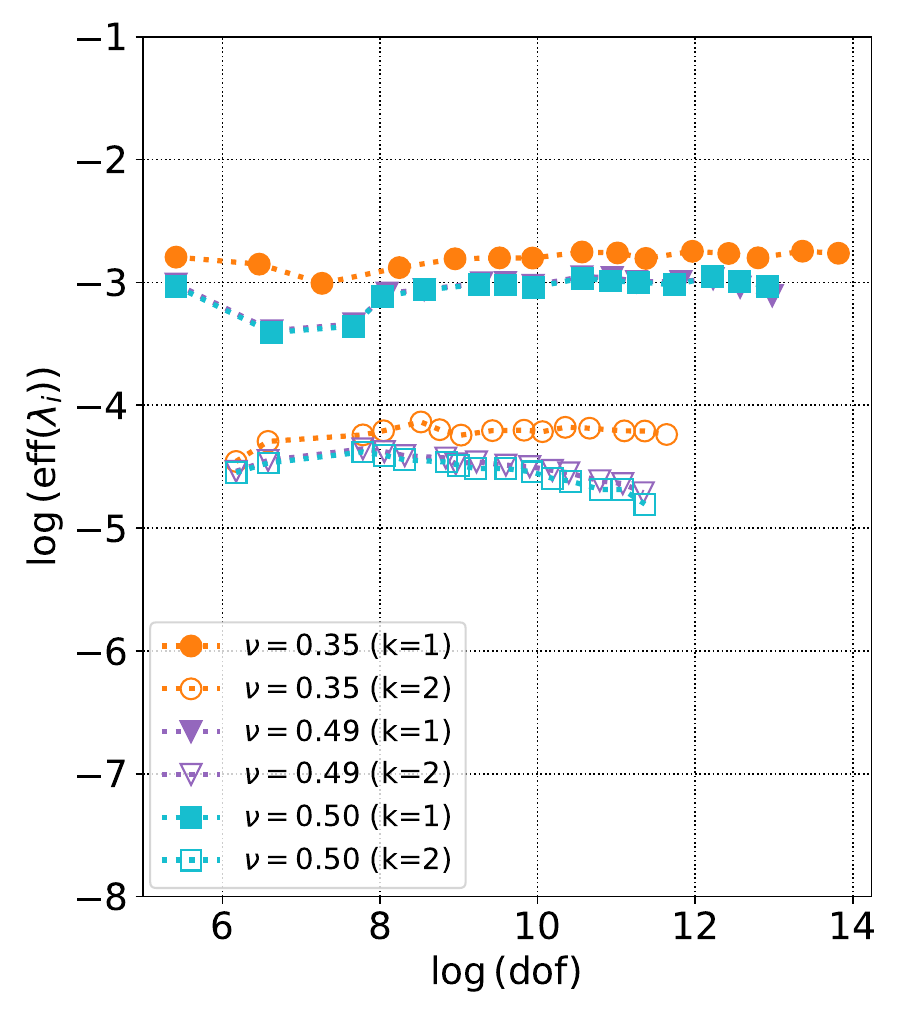}
	\end{minipage}
	\begin{minipage}{0.49\linewidth}\centering
		\includegraphics[scale=0.4]{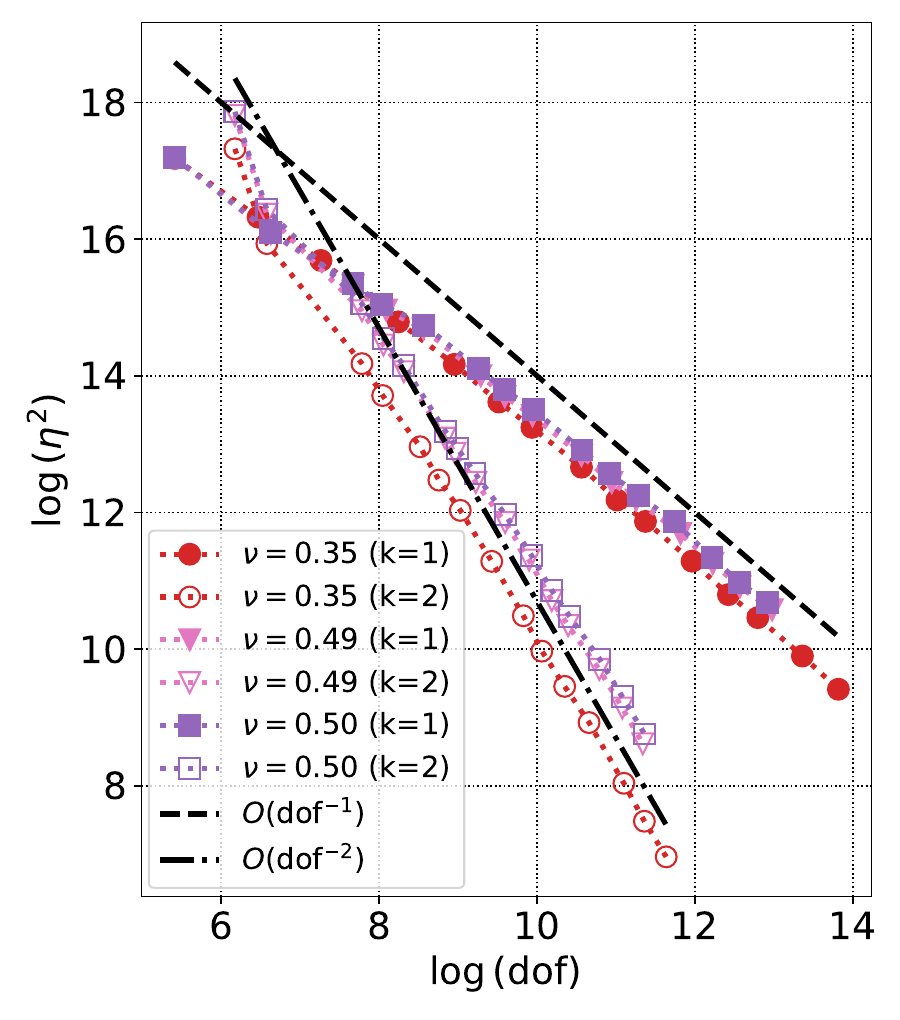}
	\end{minipage}\\
	\caption{Test \ref{subsec:square2D_two_domain}.  Estimator and efficiency curves obtained from the adaptive algorithm in the square domain with different values of $\nu$ and $k=1,2$.}
	\label{fig:lshape2D-effectivity}
\end{figure}

\subsection{The 3D L-shape domain}\label{subsec:Lshape3D-domain}
We now extend the results of our numerical scheme to three dimensions. In this experiment we test the estimator in the case of a domain with a dihedral singularity. The domain under consideration is the  L-shaped domain defined by 
$$
\Omega:=(-1,1)\times(-1,1)\times(-1,0)\backslash\bigg((-1,0)\times(-1,0)\times(-1,0) \bigg).
$$
For simplicity, the lowest order $k=1$ and $8$ adaptive iterations are considered in this experiment. Fully fixed boundary conditions are assumed across $\partial\Omega$. The domain is divided in two subdomains $\Omega_1$ and $\Omega_2$ with different materials, where
$$
\begin{aligned}
&\Omega_1:=(-1,1)\times(-1,1)\times(-1,-0.5)\backslash\bigg((-1,0)\times(-1,0)\times(-1,-0.5) \bigg),\\
&\Omega_2:=(-1,1)\times(-1,1)\times(-0.5,0)\backslash\bigg((-1,0)\times(-1,0)\times(-0.5,0) \bigg).
\end{aligned}
$$
The physical parameters are similar to those of Section \ref{subsec:square2D_two_domain}, namely, the Young modulus and density are taken similar to \eqref{eq:young-modulus}. A sample of the computational domain is depicted in Figure \ref{fig:lshape3D-modes}. Here, the initial mesh is such that $h\approx 1/2$. Because of the singularities presented in this domain, the first eigenvalue is singular for each value of $\nu$. The extrapolated values, computed using highly refined meshes are given in Table \ref{table:extrapolated-lshape3D}.

We present the first eigenmodes on Figure \ref{fig:lshape3D-modes} consisting on the displacement field streamlines. We observe that high values of the displacement magnitude are concentrated on $\Omega_1$ (bottom domain) because of the difference in the density of the materials. Different adaptive meshes for this domain are presented in Figure \ref{fig:lshape3D-meshes}, where we observe that the estimator is capable of detecting the singular behavior of the pressure for each Poisson ratio and refine near that zone. Also, the algorithm marks more elements near the line singularity belonging to $\overline{\Omega}_1$, as expected. 

Let us finish this experiment with Figure \ref{fig:lshape2D-error}. The figure represent the error curves and estimator values obtained with adaptive refinements, where an optimal rate is observed. For the 3D case, a rate of $\mathcal{O}(h^2)\approx \mathcal{O}(\texttt{dof}^{-0.66})$ was expected.
\begin{figure}[hbt!]
	\centering
	\begin{minipage}{0.32\linewidth}\centering
		\includegraphics[scale=0.07,trim=67cm 3cm 67cm 2cm,clip]{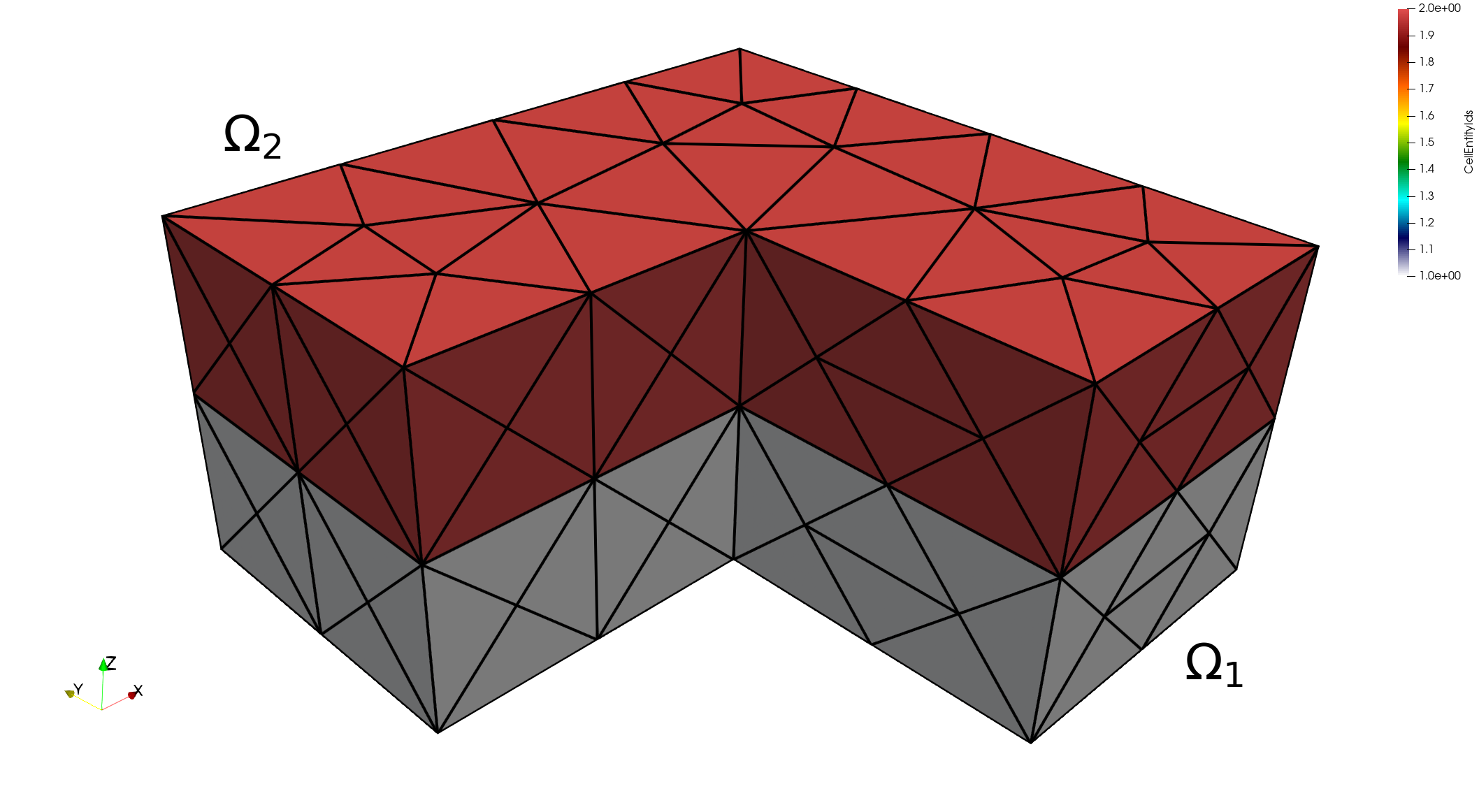}
	\end{minipage}
	\begin{minipage}{0.32\linewidth}\centering
		\includegraphics[scale=0.07,trim=67cm 3cm 67cm 2cm,clip]{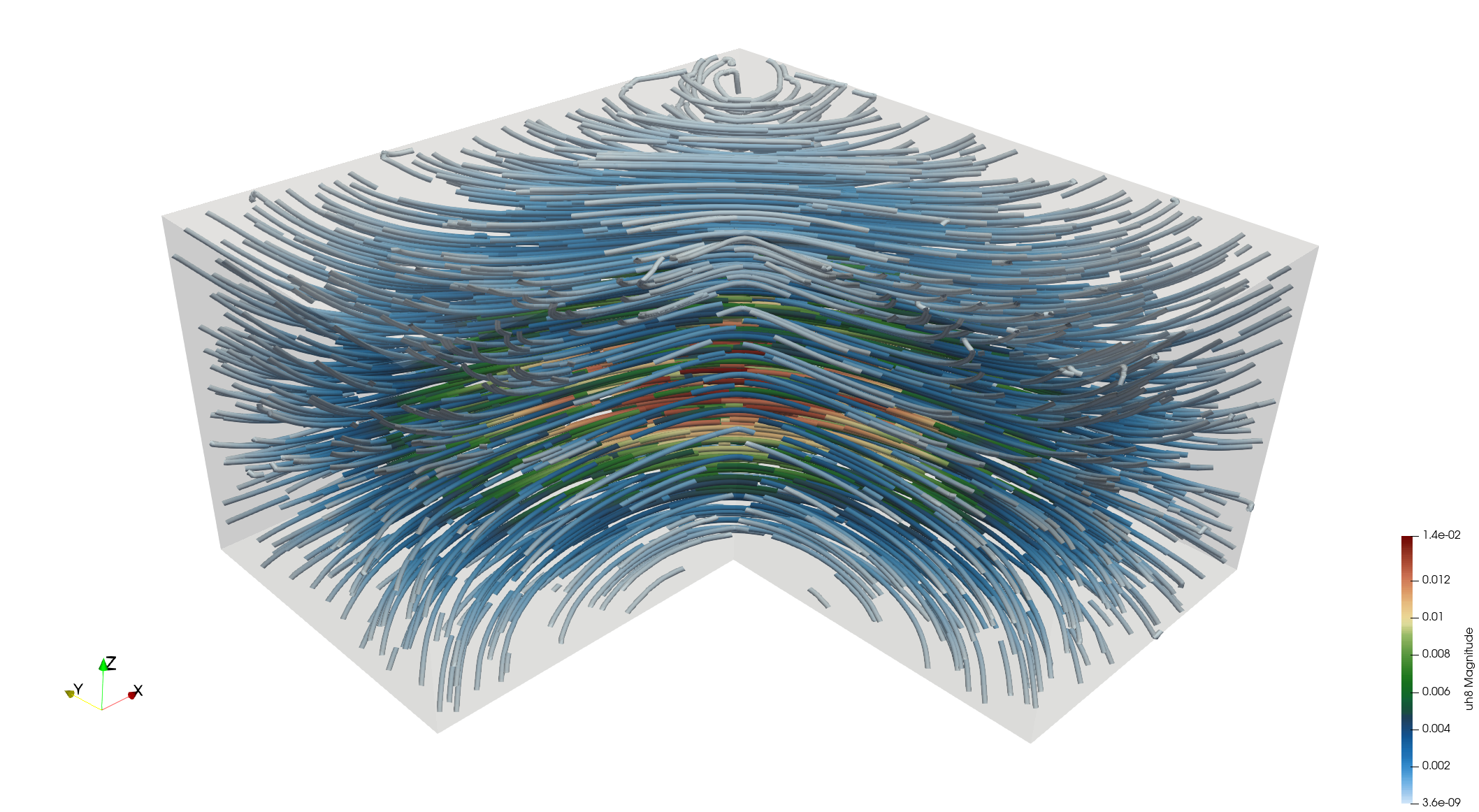}
	\end{minipage}
	\begin{minipage}{0.32\linewidth}\centering
		\includegraphics[scale=0.07,trim=67cm 3cm 67cm 2cm,clip]{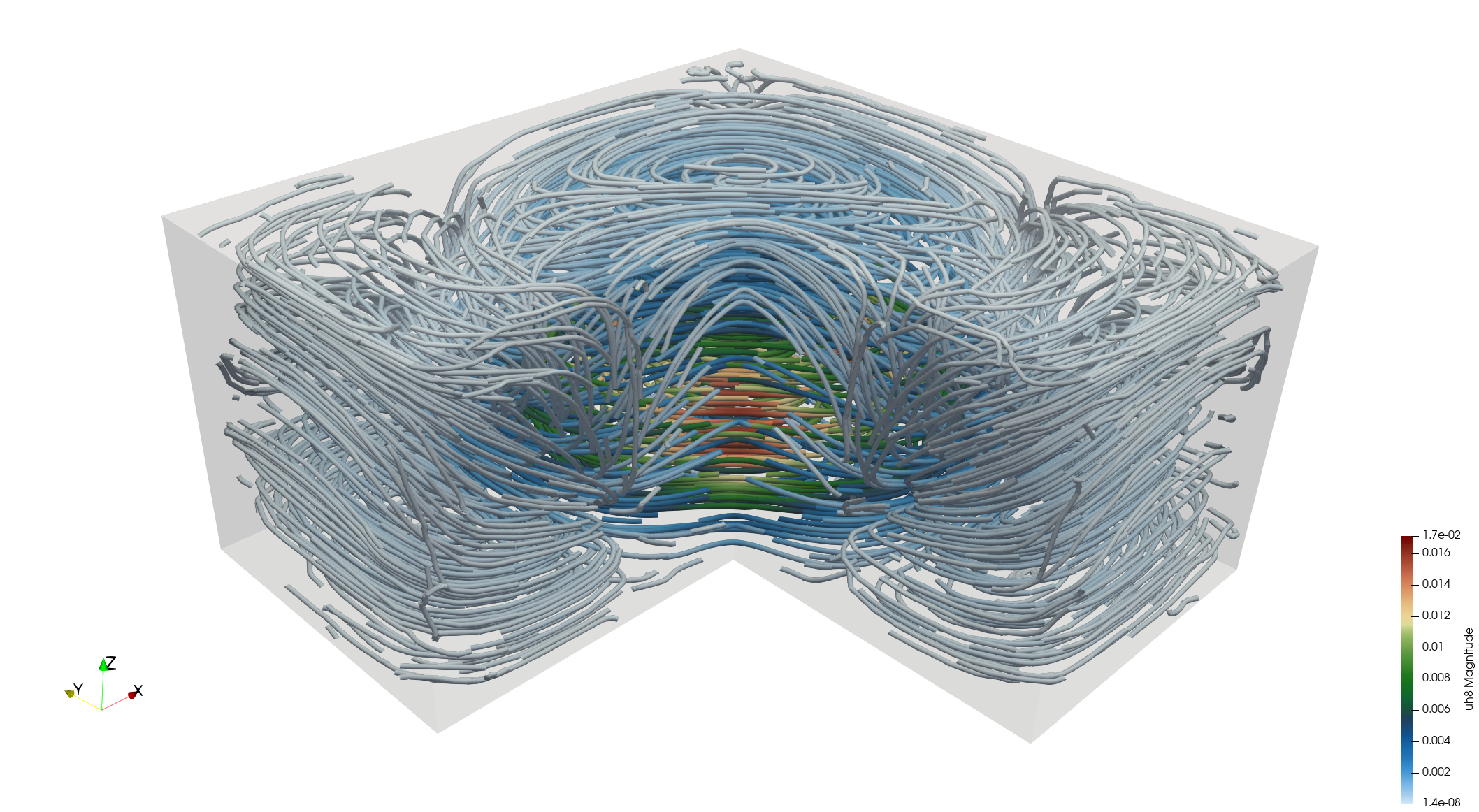}
	\end{minipage}\\
	\caption{Test \ref{subsec:Lshape3D-domain}. Initial mesh for the subdivided computational domain $\Omega:=\Omega_1\cup\Omega_2$ (left) together with a comparison of the first displacement eigenmodes for $\nu=0.35$ (middle) and $\nu=0.5$ (right) at the last adaptive iteration.}
	\label{fig:lshape3D-modes}
\end{figure}
\begin{table}[hbt!]
	{\footnotesize
		\begin{center}
			\caption{Test \ref{subsec:Lshape3D-domain}. Reference lowest computed eigenvalues in the 3D Lshape domain.}
			\begin{tabular}{|c|c|}
			\hline
			\hline
			$\nu$             & $\widehat{\kappa}_1$ \\
			\hline
			0.35  & $7187.76090678178$  \\
			0.49 & $8591.22953946104$   \\
			0.5 &   $8600.14945418864$   \\
			\hline
			\hline
		\end{tabular}
		\end{center}
		}
		\label{table:extrapolated-lshape3D}
\end{table}
\begin{figure}[hbt!]
	\centering
	\begin{minipage}{0.32\linewidth}\centering
		{\tiny $\nu=0.35, iter=5$}\\
		\includegraphics[scale=0.07,trim=67cm 3cm 67cm 2cm,clip]{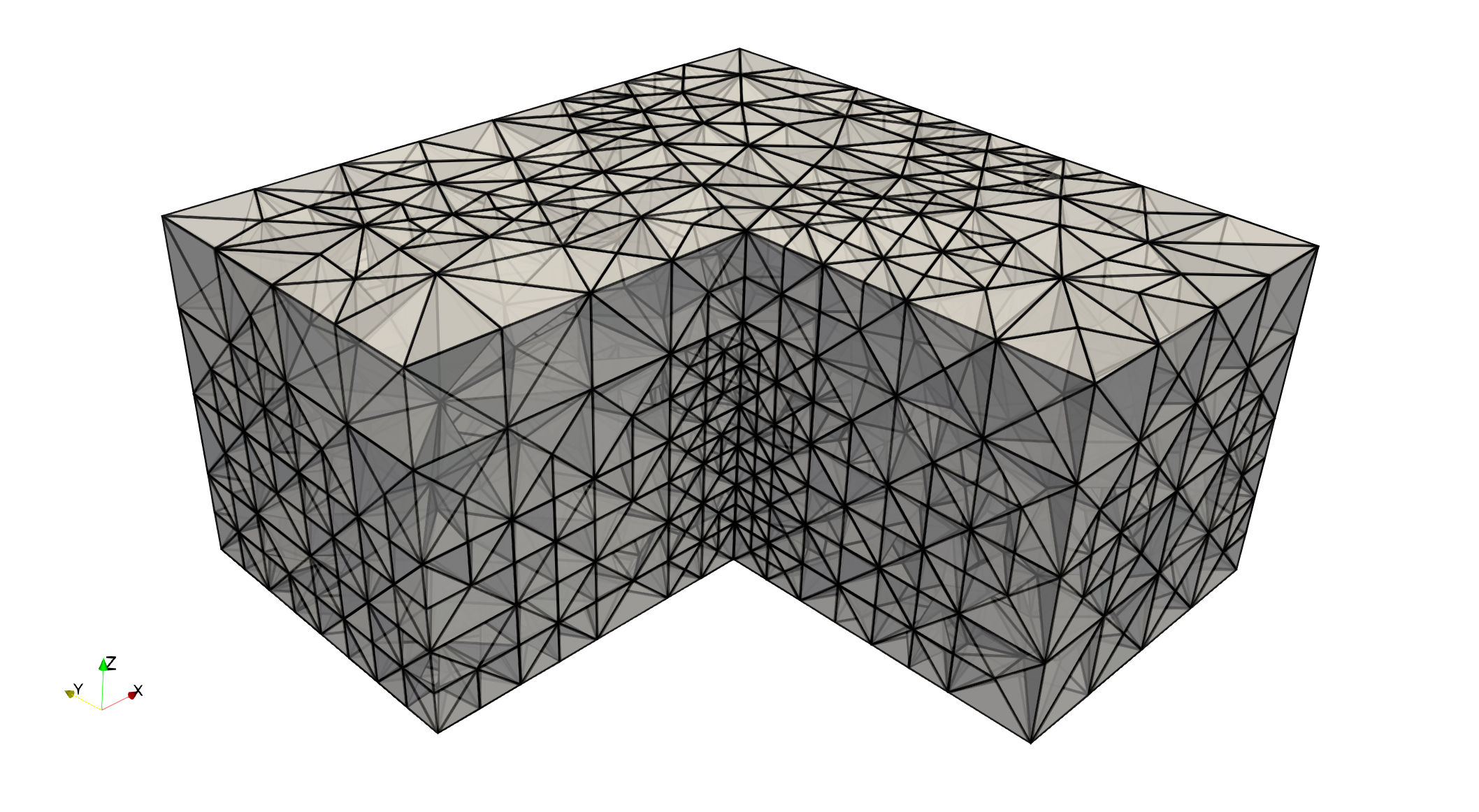}
	\end{minipage}
	\begin{minipage}{0.32\linewidth}\centering
		{\tiny $\nu=0.35,  iter =7$}\\
		\includegraphics[scale=0.07,trim=67cm 3cm 67cm 2cm,clip]{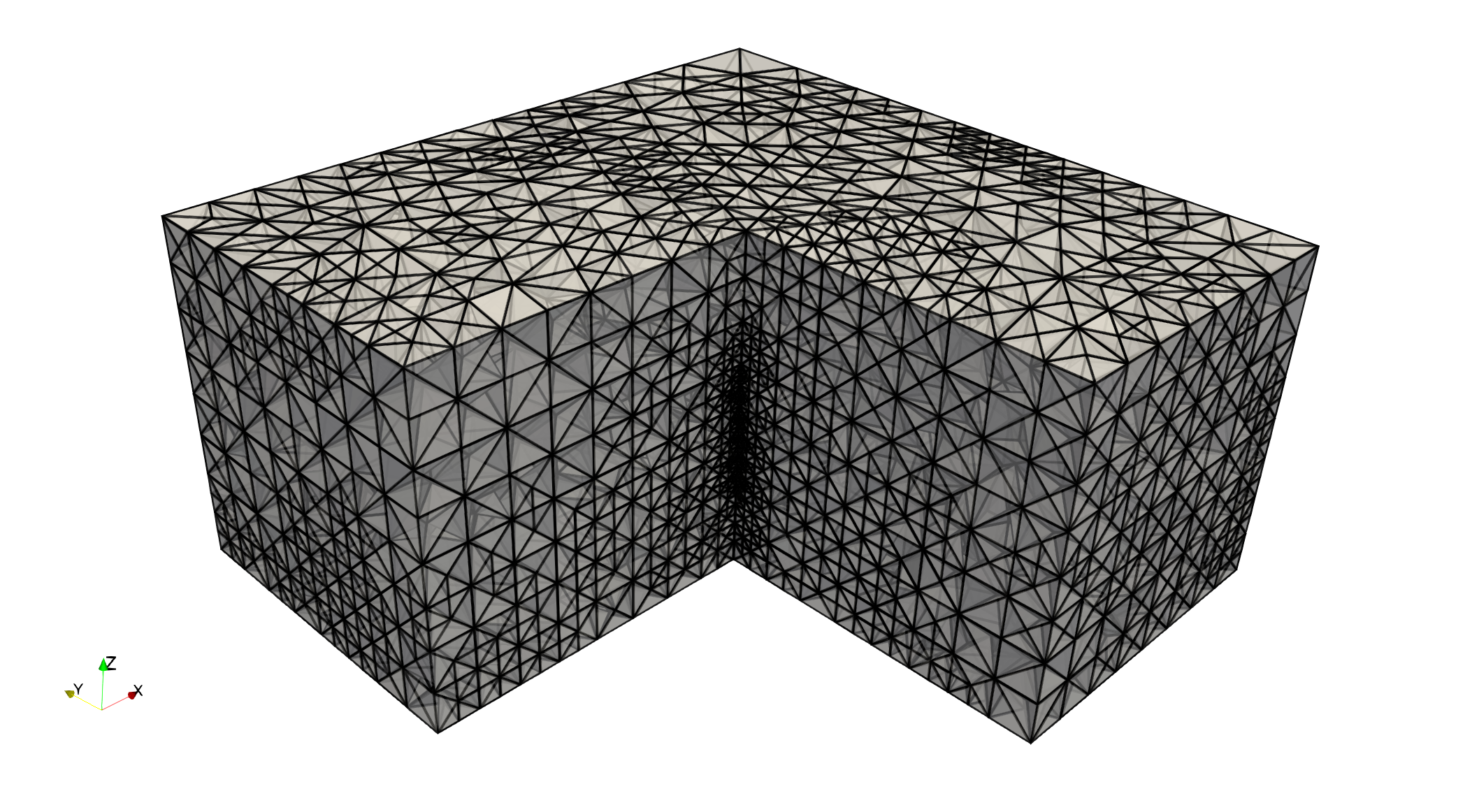}
	\end{minipage}
	\begin{minipage}{0.32\linewidth}\centering
		{\tiny $\nu=0.35,  iter=8$}\\
		\includegraphics[scale=0.07,trim=67cm 3cm 67cm 2cm,clip]{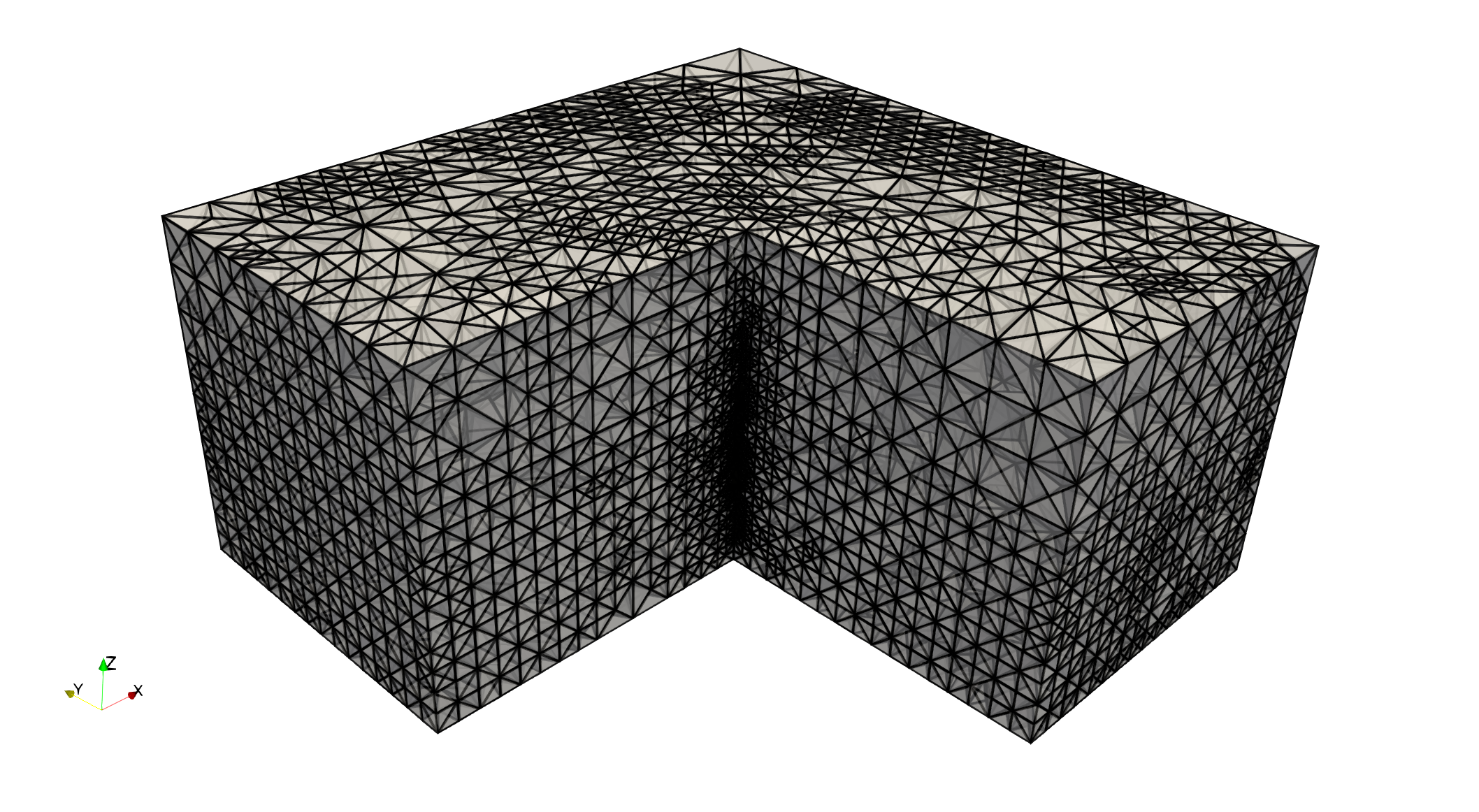}
	\end{minipage}\\
	\begin{minipage}{0.32\linewidth}\centering
		{\tiny $\nu=0.50, iter=5$}\\
		\includegraphics[scale=0.07,trim=67cm 3cm 67cm 2cm,clip]{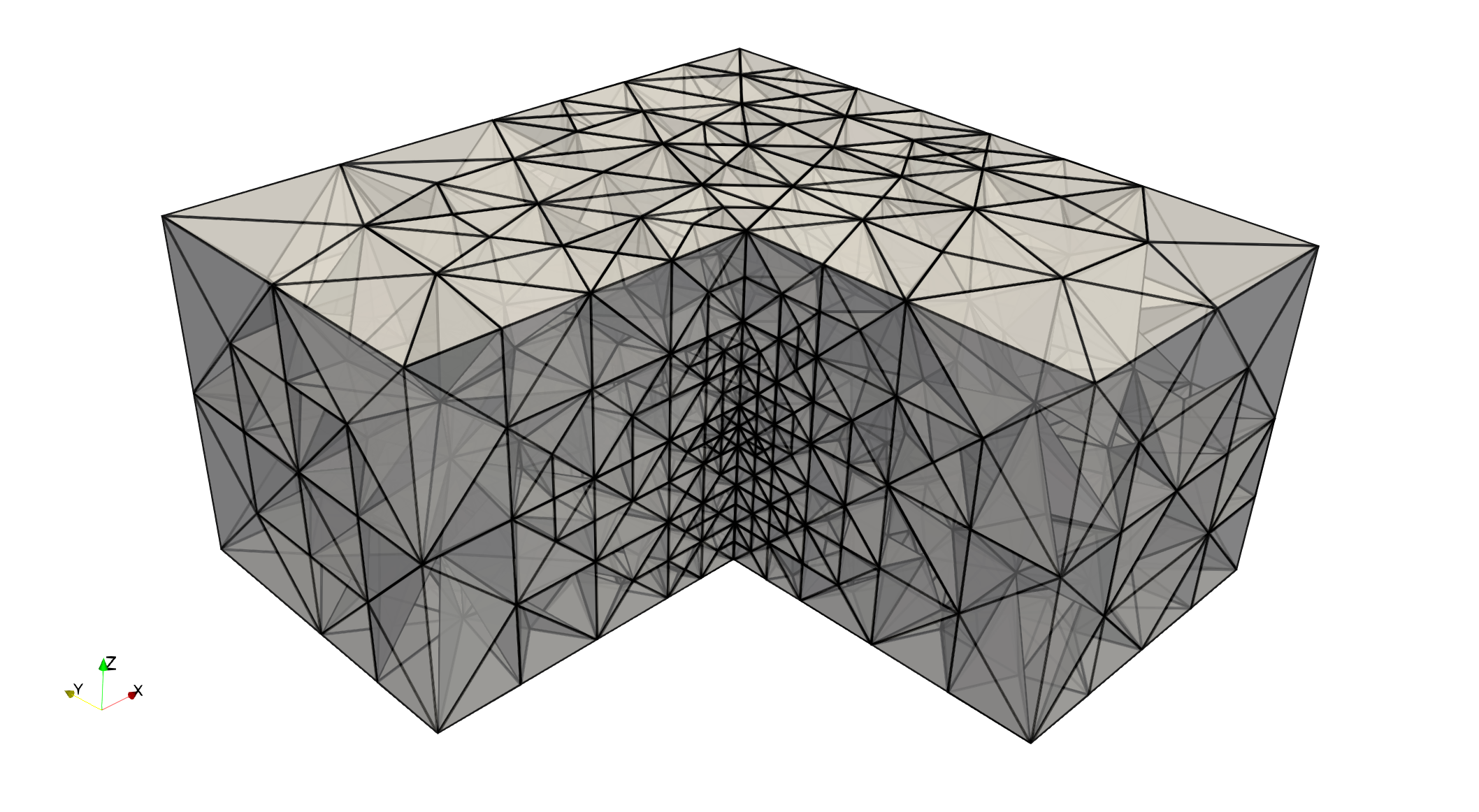}
	\end{minipage}
	\begin{minipage}{0.32\linewidth}\centering
		{\tiny $\nu=0.50, iter =7$}\\
		\includegraphics[scale=0.07,trim=67cm 3cm 67cm 2cm,clip]{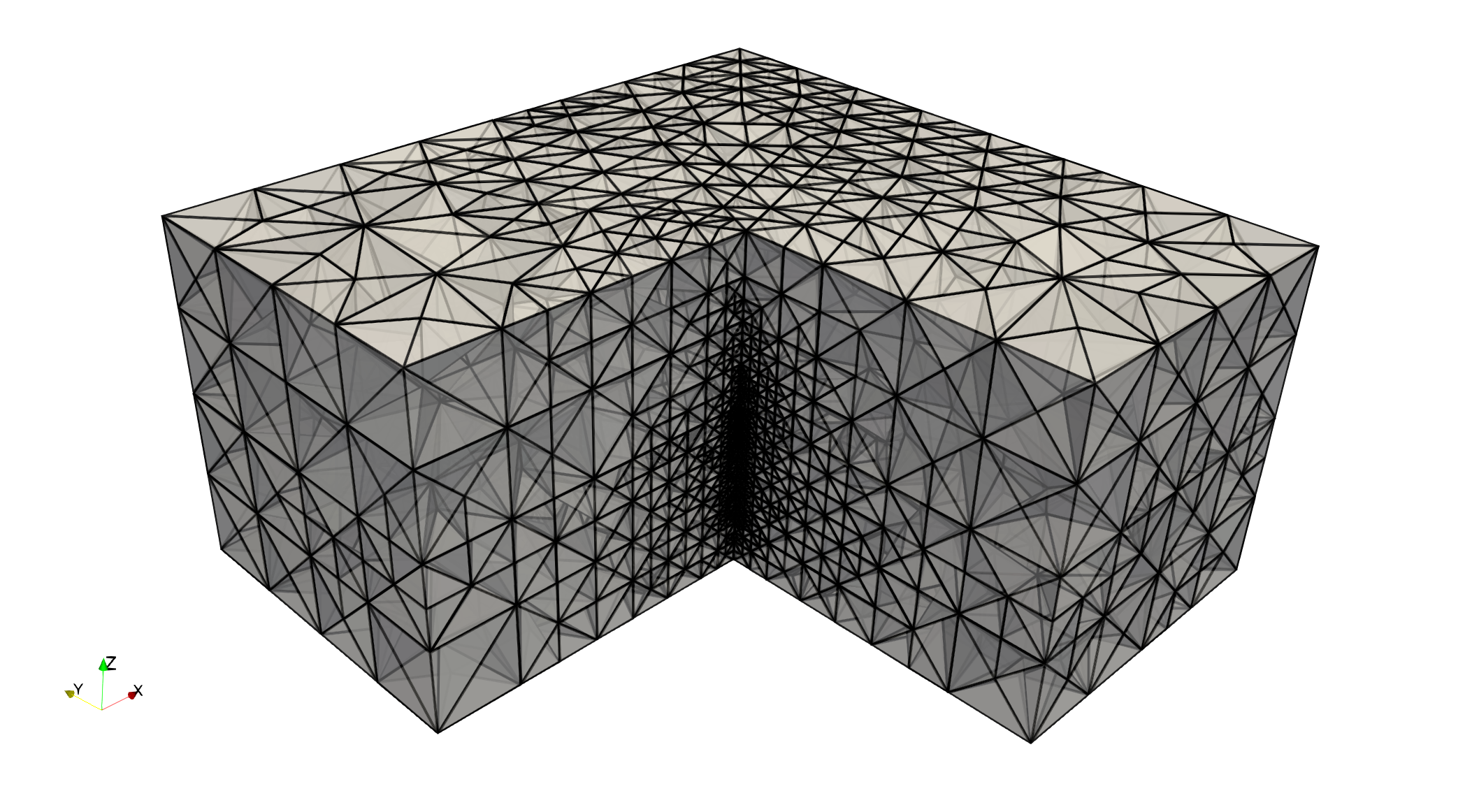}
	\end{minipage}
	\begin{minipage}{0.32\linewidth}\centering
		{\tiny $\nu=0.50, iter=8$}\\
		\includegraphics[scale=0.07,trim=67cm 3cm 67cm 2cm,clip]{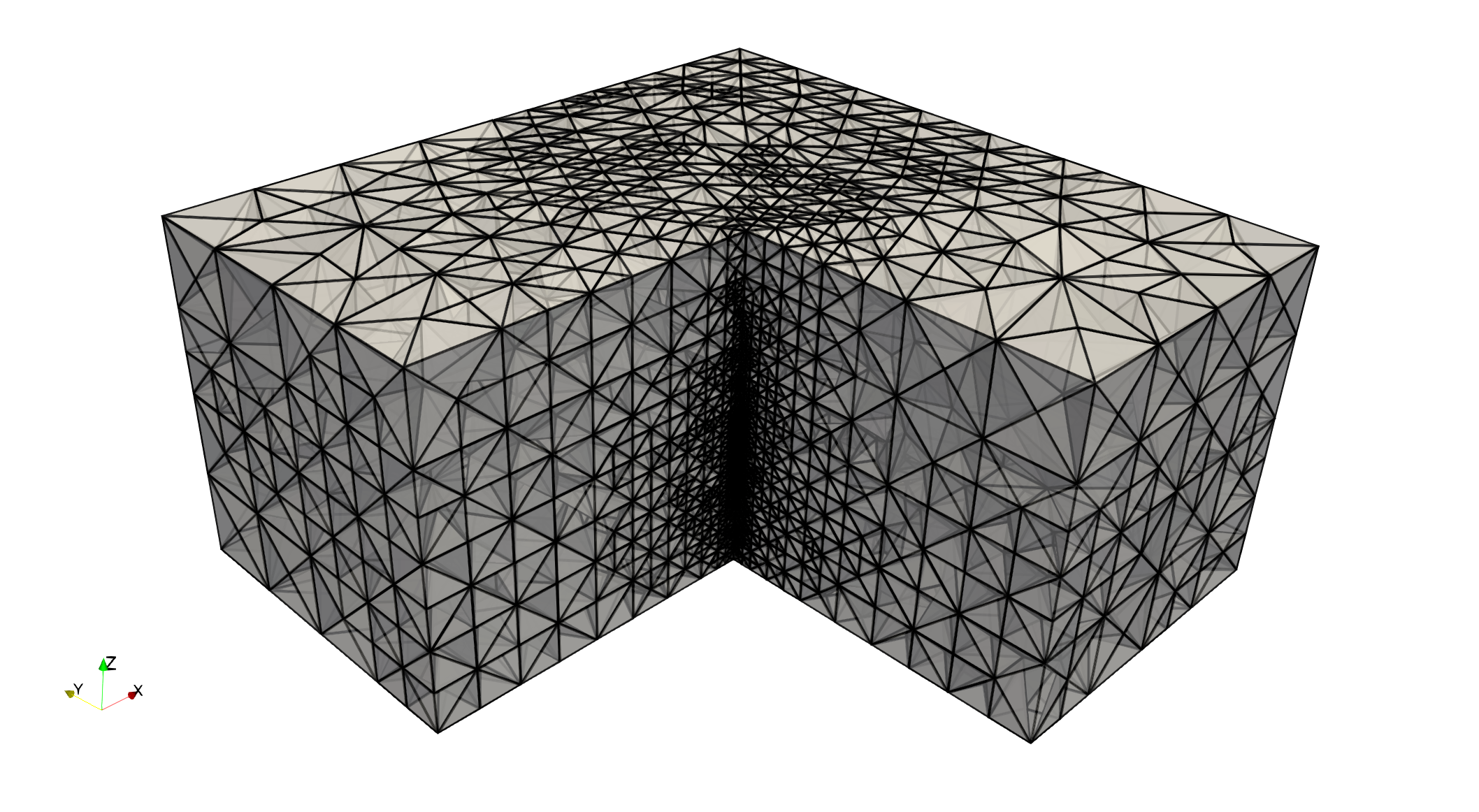}
	\end{minipage}\\
	\caption{Test \ref{subsec:Lshape3D-domain}.  Intermediate meshes of the square domain with heterogeneous media obtained with the adaptive algorithm and different values of $\nu$.}
	\label{fig:lshape3D-meshes}
\end{figure}
\begin{figure}[hbt!]
	\centering
	\includegraphics[scale=0.42]{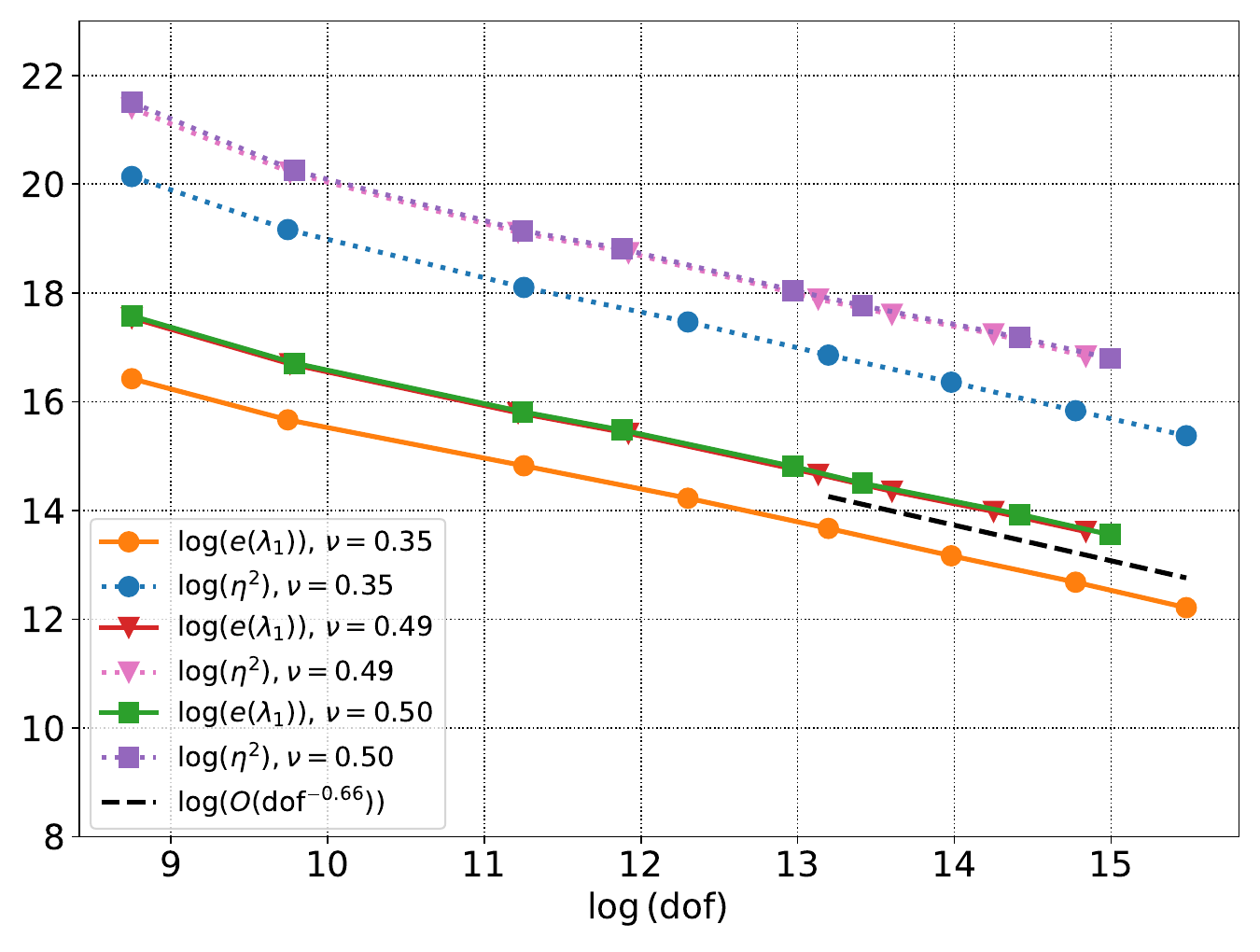}
	\caption{Test \ref{subsec:Lshape3D-domain}.  Error curves and estimator values obtained from the adaptive algorithm in the Lshape domain compared with the line $\mathcal{O}(\texttt{dof}^{-0.66})$.}
	\label{fig:fichera3D-error}
\end{figure}

\subsection{Robustness in 2D and 3D geometries}\label{subsec:robustness}
This final experiment aims to assess the robustness of the proposed DG methods. For the two-dimensional case, we consider again the unit square domain $\Omega:=(0,1)^2$ and the unit cube domain $\Omega:=(0,1)^3$ is considered for the three-dimensional case. The square and the cube are considered to be clamped at the bottom and free of stress on the rest of the facets. In the cube domain, we consider \textit{the bottom} as the set of facets such that $y=0$. The extrapolated values for this experiment are calculated from a reference value for each value of $\nu$. More precisely, in each test we have the extrapolated values in Table \ref{table:robustness-references} for $E=10^j, j=1,2,3...$.
\begin{table}[hbt!]
	{\footnotesize
		\begin{center}
			\caption{Test \ref{subsec:robustness}. Reference eigenvalues for different values of $\nu$ and $E=10^j, j=1,2,3...$ .}
			\begin{tabular}{|c|c| c |}
			\hline
			\hline
			&Square domain & Cube domain\\
			\hline
			$\nu$             & $\widehat{\kappa}_1$ & $\widehat{\kappa}_1$ \\
			\hline
			0.35  & $0.46355423498481496\cdot E$ &$0.444317882233217\cdot E$ \\
			0.49 & $0.48938358373431\cdot E$   &$0.44833605908518\cdot E$  \\
			0.5 &   $0.492273855811713\cdot E$  &$0.44915941661888\cdot E$  \\
			\hline
			\hline
		\end{tabular}
		\end{center}
		\label{table:robustness-references}}
\end{table}
A  stabilization parameter $\texttt{a}=10$ is chosen for all the choices of $E$ and $\nu$.  From the results displayed in Table \ref{tabla:robustness-square-2D3D_SIP} we observe that a almost a constant value for $\eff(\cdot)$ is attained for the selected values of $E$ in two and three dimensions. We recall that similar results were obtained when considering Young modulus $E\in\{10^{-1},10^{-2},10^{-4}\}$. Hence, the proposed estimator is robust with respect to the physical parameters.
\begin{table}[hbt!]\setlength{\tabcolsep}{2.7pt}
	{\footnotesize
		\centering
		\caption{Test \ref{subsec:robustness}. Efficiency indexes for different values of $\nu$ and $E$ on the unit square and unit cube domain when using the SIP method ($\varepsilon=1$).}
		\begin{tabular}{|l|ccc|ccc|ccc|}
			\hline
			&\multicolumn{3}{c}{$\nu=0.35$} & \multicolumn{3}{|c}{$\nu=0.49$} & \multicolumn{3}{|c|}{$\nu=0.50$}\\
			\hline
			\multicolumn{10}{|c|}{Unit square domain $\Omega=(0,1)^2$}\\
			\hline
			\texttt{dof}	&\multicolumn{3}{c}{$\texttt{eff}(\widehat{\kappa}_1)$} & \multicolumn{3}{|c}{$\texttt{eff}(\widehat{\kappa}_1)$} & \multicolumn{3}{|c|}{$\texttt{eff}(\widehat{\kappa}_1)$}\\
			\hline
				126 & 9.01e-02 & 9.01e-02 & 9.01e-02 & 8.17e-02 & 8.19e-02 & 8.19e-02 & 8.13e-02 & 8.15e-02 & 8.15e-02 \\
			350 & 9.60e-02 & 9.60e-02 & 9.60e-02 & 9.16e-02 & 9.18e-02 & 9.18e-02 & 9.11e-02 & 9.14e-02 & 9.14e-02 \\
			1134 & 1.02e-01 & 1.02e-01 & 1.02e-01 & 9.92e-02 & 9.94e-02 & 9.94e-02 & 9.86e-02 & 9.89e-02 & 9.89e-02 \\
			4046 & 1.08e-01 & 1.08e-01 & 1.08e-01 & 1.04e-01 & 1.04e-01 & 1.04e-01 & 1.03e-01 & 1.04e-01 & 1.04e-01 \\
			15246 & 1.12e-01 & 1.12e-01 & 1.12e-01 & 1.07e-01 & 1.07e-01 & 1.07e-01 & 1.05e-01 & 1.06e-01 & 1.06e-01 \\
			\hline
			\multicolumn{10}{|c|}{Unit cube domain $\Omega=(0,1)^3$}\\
			\hline
			\texttt{dof}	&\multicolumn{3}{c}{$\texttt{eff}(\widehat{\kappa}_1)$} & \multicolumn{3}{|c}{$\texttt{eff}(\widehat{\kappa}_1)$} & \multicolumn{3}{|c|}{$\texttt{eff}(\widehat{\kappa}_1)$}\\
			\hline
			9750 & 4.53e-01 & 4.54e-01 & 4.54e-01 & 2.21e-01 & 2.21e-01 & 2.21e-01 & 1.80e-01 & 1.80e-01 & 1.80e-01 \\
			78000 & 8.05e-01 & 8.06e-01 & 8.06e-01 & 2.33e-01 & 2.33e-01 & 2.33e-01 & 1.43e-01 & 1.43e-01 & 1.43e-01 \\
			263250 & 1.13e+00 & 1.13e+00 & 1.13e+00 & 2.60e-01 & 2.60e-01 & 2.60e-01 & 1.27e-01 & 1.28e-01 & 1.28e-01 \\
			624000 & 1.43e+00 & 1.44e+00 & 1.44e+00 & 2.93e-01 & 2.94e-01 & 2.94e-01 & 1.19e-01 & 1.19e-01 & 1.19e-01 \\
			1218750 & 1.72e+00 & 1.72e+00 & 1.72e+00 & 3.29e-01 & 3.30e-01 & 3.30e-01 & 1.13e-01 & 1.14e-01 & 1.14e-01 \\
			\hline
			$E$& $10$ &$10^2$ & $10^4$  & $10$ &$10^2$ & $10^4$ & $10$ &$10^2$ & $10^4$ \\
			\hline
			\hline
		\end{tabular}
		\label{tabla:robustness-square-2D3D_SIP}}
\end{table}

\bibliographystyle{amsplain}
\bibliography{bib_LOQ}
\end{document}

%% file: ex_shared.tex
\usepackage{lipsum}
\usepackage{amsmath,amssymb,amsfonts,latexsym,stmaryrd}
\usepackage{graphicx,float}
\usepackage{mathrsfs} 
\usepackage{color}
\usepackage{epstopdf}
\usepackage{algorithmic}
\usepackage{multirow}
\ifpdf
  \DeclareGraphicsExtensions{.eps,.pdf,.png,.jpg}
\else
  \DeclareGraphicsExtensions{.eps}
\fi
\usepackage{diagbox}

\def\ds{\displaystyle}
\def\G{\Gamma}
\def\O{\Omega}

\renewcommand\sp{\mathop{\mathrm{Sp}}\nolimits}

\newcommand{\set}[1]{\lbrace #1 \rbrace}

\newcommand{\jump}[1]{\llbracket #1 \rrbracket}
\newcommand{\mean}[1]{\{#1\}}

\newcommand{\norm}[1]{\lVert#1\rVert}

\newcommand{\n}{\boldsymbol{n}}
\usepackage{booktabs}
\usepackage{float}

\newcommand\bu{\boldsymbol{u}}
\newcommand\bv{\boldsymbol{v}}

\newcommand\bx{\boldsymbol{x}}

\newcommand\bs{\boldsymbol{s}}
\newcommand\bn{\boldsymbol{n}}

\def\hdel{\widehat{\delta}}
\def\CM{\mathcal{X}}
\def\CN{\mathcal{Y}}

\newcommand\bF{\boldsymbol{f}}
\newcommand\bI{\boldsymbol{I}}

\newcommand\bT{\boldsymbol{T}}

\newcommand\0{\mathbf{0}}

\newcommand{\cF}{\mathcal{F}}

\newcommand\cP{\mathcal{P}}
\newcommand\cT{\mathcal{T}}
\def\CT{{\mathcal T}}

\newcommand\beps{\boldsymbol{\varepsilon}}

\newcommand\btau{\boldsymbol{\tau}}

\newcommand\R{\mathbb{R}}

\renewcommand\div{\mathop{\mathrm{div}}\nolimits}

\newcommand\tr{\mathop{\mathrm{tr}}\nolimits}

\renewcommand\sp{\mathop{\mathrm{sp}}\nolimits}

\newcommand\LO{\L^2(\O)}

\newcommand\HsO{\H^s(\O)}

\newcommand{\vertiii}[1]{{\left\vert\kern-0.25ex\left\vert\kern-0.25ex\left\vert #1 
    \right\vert\kern-0.25ex\right\vert\kern-0.25ex\right\vert}}
\newcommand{\vertiH}[1]{{\left\vert\kern-0.25ex\left\vert #1 
		\right\vert\kern-0.25ex\right\vert_{\H(h)}}}

\newtheorem{theorem}{Theorem}[section]
\newtheorem{lemma}[theorem]{Lemma}
\newtheorem{corollary}[theorem]{Corollary}
\theoremstyle{definition}

\theoremstyle{remark}
\newtheorem{remark}[theorem]{Remark}

\numberwithin{equation}{section}

\title[DG for the nearly incompressible elasticity eigenvalue problem]{Interior penalty discontinuous Galerkin methods  for the nearly incompressible elasticity eigenvalue problem with heterogeneous media}

\author{Arbaz Khan}
\address{Department of Mathematics, Indian Institute of Technology Roorkee, Roorkee 247667, India}
\curraddr{}
\email{arbaz@ma.iitr.ac.in}
\thanks{The first author was  partially
	supported by the Sponsored Research \& Industrial Consultancy (SRIC), Indian Institute of Technology Roorkee,	India through the faculty initiation grant MTD/FIG/100878; by SERB MATRICS grant	MTR/2020/000303; by SERB Core research grant CRG/2021/002569.}

\author{Felipe Lepe}
\address{GIMNAP-Departamento de Matem\'atica, Universidad del B\'io - B\'io, Casilla 5-C, Concepci\'on, Chile}
\curraddr{}
\email{flepe@ubiobio.cl}
\thanks{The second author was partially supported by  DIUBB through project 2120173 GI/C Universidad del B\'io-B\'io.}

\author{Jesus Vellojin}
\address{GIMNAP-Departamento de Matem\'atica, Universidad del B\'io - B\'io, Casilla 5-C, Concepci\'on, Chile}
\curraddr{}
\email{jvellojin@ubiobio.cl}
\thanks{The third author was partially supported by the National Agency for Research and Development, ANID-Chile through FONDECYT Postdoctorado project 3230302}

\usepackage{amsopn}


%% file: elast_incomp_DG.bbl
\providecommand{\bysame}{\leavevmode\hbox to3em{\hrulefill}\thinspace}
\providecommand{\MR}{\relax\ifhmode\unskip\space\fi MR }
\providecommand{\MRhref}[2]{%
  \href{http://www.ams.org/mathscinet-getitem?mr=#1}{#2}
}
\providecommand{\href}[2]{#2}
\begin{thebibliography}{10}

\bibitem{MR1885308}
Mark Ainsworth and J.~Tinsley Oden, \emph{A posteriori error estimation in
  finite element analysis}, Pure and Applied Mathematics (New York),
  Wiley-Interscience [John Wiley \& Sons], New York, 2000. \MR{1885308}

\bibitem{MR2220929}
Paola~F. Antonietti, Annalisa Buffa, and Ilaria Perugia, \emph{Discontinuous
  {G}alerkin approximation of the {L}aplace eigenproblem}, Comput. Methods
  Appl. Mech. Engrg. \textbf{195} (2006), no.~25-28, 3483--3503. \MR{2220929}

\bibitem{barrata2023dolfinx}
Igor~A Barrata, Joseph~P Dean, J{\o}rgen~S Dokken, Michal HABERA, Jack HALE,
  Chris Richardson, Marie~E Rognes, Matthew~W Scroggs, Nathan Sime, and Garth~N
  Wells, \emph{{DOLFINx}: The next generation fenics problem solving
  environment},  (2023).

\bibitem{MR2263045}
Annalisa Buffa and Ilaria Perugia, \emph{Discontinuous {G}alerkin approximation
  of the {M}axwell eigenproblem}, SIAM J. Numer. Anal. \textbf{44} (2006),
  no.~5, 2198--2226. \MR{2263045}

\bibitem{BCGKDS}
Bernardo Cockburn, Guido Kanschat, and Dominik Sch\"otzau, \emph{A note on
  discontinuous {G}alerkin divergence-free solutions of the {N}avier-{S}tokes
  equations}, J. Sci. Comput. \textbf{31} (2007), no.~1-2, 61--73.

\bibitem{MR483400}
Jean Descloux, Nabil Nassif, and Jacques Rappaz, \emph{On spectral
  approximation. {I}. {T}he problem of convergence}, RAIRO Anal. Num\'{e}r.
  \textbf{12} (1978), no.~2, 97--112, iii. \MR{483400}

\bibitem{MR483401}
\bysame, \emph{On spectral approximation. {II}. {E}rror estimates for the
  {G}alerkin method}, RAIRO Anal. Num\'{e}r. \textbf{12} (1978), no.~2,
  113--119, iii. \MR{483401}

\bibitem{MR0975121}
E.~B. Fabes, C.~E. Kenig, and G.~C. Verchota, \emph{The {D}irichlet problem for
  the {S}tokes system on {L}ipschitz domains}, Duke Math. J. \textbf{57}
  (1988), no.~3, 769--793. \MR{975121}

\bibitem{MR4071826}
Joscha Gedicke and Arbaz Khan, \emph{Divergence-conforming discontinuous
  {G}alerkin finite elements for {S}tokes eigenvalue problems}, Numer. Math.
  \textbf{144} (2020), no.~3, 585--614. \MR{4071826}

\bibitem{geuzaine2009gmsh}
Christophe Geuzaine and Jean-Fran{\c{c}}ois Remacle, \emph{Gmsh: A 3-{D} finite
  element mesh generator with built-in pre-and post-processing facilities},
  International journal for numerical methods in engineering \textbf{79}
  (2009), no.~11, 1309--1331.

\bibitem{grisvard1986problemes}
P~Grisvard, \emph{Problemes aux limites dans les polygones. mode d'emploi},
  Bulletin de la Direction des Etudes et Recherches Series C Mathematiques,
  Informatique \textbf{1} (1986), 21--59.

\bibitem{MR1886000}
Peter Hansbo and Mats~G. Larson, \emph{Discontinuous {G}alerkin methods for
  incompressible and nearly incompressible elasticity by {N}itsche's method},
  Comput. Methods Appl. Mech. Engrg. \textbf{191} (2002), no.~17-18,
  1895--1908. \MR{1886000}

\bibitem{HSW}
Paul Houston, Dominik Sch\"otzau, and Thomas~P. Wihler, \emph{Energy norm a
  posteriori error estimation for mixed discontinuous {G}alerkin approximations
  of the {S}tokes problem}, J. Sci. Comput. \textbf{22/23} (2005), 347--370.

\bibitem{GKDS}
Guido Kanschat and Dominik Sch\"otzau, \emph{Energy norm a posteriori error
  estimation for divergence-free discontinuous {G}alerkin approximations of the
  {N}avier-{S}tokes equations}, Internat. J. Numer. Methods Fluids \textbf{57}
  (2008), no.~9, 1093--1113.

\bibitem{khan2023finite}
Arbaz Khan, Felipe Lepe, David Mora, and Jesus Vellojin, \emph{Finite element
  analysis of the nearly incompressible linear elasticity eigenvalue problem
  with variable coefficients}, 2023.

\bibitem{MR3973678}
Arbaz Khan, Catherine~E. Powell, and David~J. Silvester, \emph{Robust a
  posteriori error estimators for mixed approximation of nearly incompressible
  elasticity}, Internat. J. Numer. Methods Engrg. \textbf{119} (2019), no.~1,
  18--37. \MR{3973678}

\bibitem{MR4623018}
Felipe Lepe, \emph{Interior penalty discontinuous {G}alerkin methods for the
  velocity-pressure formulation of the {S}tokes spectral problem}, Adv. Comput.
  Math. \textbf{49} (2023), no.~4, Paper No. 61, 31. \MR{4623018}

\bibitem{MR3962898}
Felipe Lepe, Salim Meddahi, David Mora, and Rodolfo Rodr\'{\i}guez, \emph{Mixed
  discontinuous {G}alerkin approximation of the elasticity eigenproblem},
  Numer. Math. \textbf{142} (2019), no.~3, 749--786. \MR{3962898}

\bibitem{MR4077220}
Felipe Lepe and David Mora, \emph{Symmetric and nonsymmetric discontinuous
  {G}alerkin methods for a pseudostress formulation of the {S}tokes spectral
  problem}, SIAM J. Sci. Comput. \textbf{42} (2020), no.~2, A698--A722.
  \MR{4077220}

\bibitem{MR4673997}
Felipe Lepe, David Mora, and Jesus Vellojin, \emph{Discontinuous {G}alerkin
  methods for the acoustic vibration problem}, J. Comput. Appl. Math.
  \textbf{441} (2024), Paper No. 115700. \MR{4673997}

\bibitem{MR4430561}
Felipe Lepe, Gonzalo Rivera, and Jesus Vellojin, \emph{Mixed methods for the
  velocity-pressure-pseudostress formulation of the {S}tokes eigenvalue
  problem}, SIAM J. Sci. Comput. \textbf{44} (2022), no.~3, A1358--A1380.
  \MR{4430561}

\bibitem{MR1011446}
L.~Ridgway Scott and Shangyou Zhang, \emph{Finite element interpolation of
  nonsmooth functions satisfying boundary conditions}, Math. Comp. \textbf{54}
  (1990), no.~190, 483--493. \MR{1011446}

\bibitem{scroggs2022basix}
Matthew~W Scroggs, Igor~A Baratta, Chris~N Richardson, and Garth~N Wells,
  \emph{Basix: a runtime finite element basis evaluation library}, Journal of
  Open Source Software \textbf{7} (2022), no.~73, 3982.

\bibitem{MR3059294}
R\"{u}diger Verf\"{u}rth, \emph{A posteriori error estimation techniques for
  finite element methods}, Numerical Mathematics and Scientific Computation,
  Oxford University Press, Oxford, 2013. \MR{3059294}

\end{thebibliography}
